\documentclass[12pt]{amsart}

\usepackage{amsaddr}
\usepackage{amsmath}
\usepackage{amssymb}
\usepackage{amscd}
\usepackage{latexsym}
\usepackage{amsthm}
\usepackage{mathrsfs}
\usepackage{color}
\usepackage{amsfonts}
\usepackage{enumitem}
\usepackage{array}
\usepackage{setspace}
\usepackage{mathtools}
\usepackage{tikz}
\usepackage{tikz-cd}
\usepackage{hyperref}
\usetikzlibrary{arrows}
\usetikzlibrary{knots,patterns}

\usepackage{stmaryrd,graphicx}
\newtheorem{thm}{Theorem}[subsection]
\newtheorem*{thm*}{Theorem}
\newtheorem{cor}[thm]{Corollary}
\newtheorem*{cor*}{Corollary}
\newtheorem{lem}[thm]{Lemma}
\newtheorem*{lem*}{Lemma}
\newtheorem{prop}[thm]{Proposition}
\newtheorem*{prop*}{Proposition}

\theoremstyle{definition}
\newtheorem{defn}[thm]{Definition}
\newtheorem*{defn*}{Definition}

\newtheorem*{conjecture*}{Conjecture}
\newtheorem*{condition*}{Condition}
\newtheorem*{assumption*}{Assumption}

\theoremstyle{remark}
\newtheorem{rem}[thm]{Remark}
\newtheorem*{rem*}{Remark}
\newtheorem{example}[thm]{Example}

\newtheorem*{problem*}{Problem}

\numberwithin{equation}{subsection}

\newcommand{\toiso}{\xrightarrow{\sim}}

\newcommand{\BQ}{\mathbb Q}
\newcommand{\BR}{\mathbb R}
\newcommand{\BC}{\mathbb C}
\newcommand{\BF}{\mathbb F}

\newcommand{\BZ}{\mathbb Z}
\newcommand{\BP}{\mathbb P}

\newcommand{\CE}{\mathcal E}
\newcommand{\CF}{\mathcal F}

\newcommand{\CP}{\mathcal P}

\newcommand{\CX}{\mathcal X}

\newcommand{\GIT}{\mathrm {GIT}}
\newcommand{\sm}{\mathrm {sm}}
\newcommand{\sing}{\mathrm {sing}}
\newcommand{\reg}{\mathrm {reg}}
\newcommand{\uni}{\mathrm {uni}}

\newcommand{\ol}{\overline}

\newcommand{\z}{z}

\newcommand{\const}{\mathrm{const}}

\newcommand{\parab}{\mathrm{par}}

\newcommand{\ad}{\mathrm{ad}}

\newcommand{\nondeg}{\mathrm{nondeg}}

\DeclareMathOperator{\codim}{codim}
\DeclareMathOperator{\Gr}{Gr}

\DeclareMathOperator{\GL}{GL}
\DeclareMathOperator{\PGL}{PGL}
\DeclareMathOperator{\MH}{MH}
\DeclareMathOperator{\MHC}{MHC}
\DeclareMathOperator{\Hom}{Hom}
\DeclareMathOperator{\Tot}{Tot}
\DeclareMathOperator{\Var}{Var}
\DeclareMathOperator{\QVar}{\mathbb{Q}Var}
\DeclareMathOperator{\CVar}{\mathcal{C}Var}

\DeclareMathOperator{\Mat}{Mat}

\DeclareMathOperator{\rot}{rot}

\DeclareMathOperator{\pLog}{Log}
\DeclareMathOperator{\trace}{Tr}
\DeclareMathOperator{\tr}{tr}

\DeclareMathOperator{\inv}{inv}
\DeclareMathOperator{\noinv}{noinv}

\newcommand{\Br}{\mathrm{Br}}
\DeclareMathOperator{\cone}{cone}
\DeclareMathOperator{\kernel}{Ker}
\DeclareMathOperator{\image}{Im}
\DeclareMathOperator{\cokernel}{Coker}

\DeclareMathOperator{\sign}{sign}

\DeclareMathOperator{\Id}{Id}

\title[Cell decompositions of character varieties]{Cell decompositions of character varieties}
\author{Anton Mellit}
\email{anton.mellit@univie.ac.at}
\address{Faculty of Mathematics, University of Vienna, \\
Oskar-Morgenstern-Platz 1, 1090 Vienna, Austria}
\date{\today}

\begin{document}
\onehalfspacing

\begin{abstract}
	We establish curious Lefschetz property for generic character varieties of Riemann surfaces conjectured by Hausel, Letellier and Rodriguez-Villegas. Our main tool applies directly in the case when there is at least one puncture where the local monodromy has distinct eigenvalues. We pass to a vector bundle over the character variety, which is then decomposed into cells, which look like vector bundles over varieties associated to braids by Shende-Treumann-Zaslow. These varieties are in turn decomposed into cells that look like $\BC^{*d-2k}\times \BC^k$. The curious Lefschetz property is shown to hold on each cell, and therefore holds for the character variety. To deduce the general case, we introduce a fictitious puncture with trivial monodromy, and show that the cohomology of the character variety where one puncture has trivial monodromy is isomorphic to the sign component of the $S_n$ action on the cohomology for the character variety where trivial monodromy is replaced by regular semisimple monodromy. This involves an argument with the Grothendieck-Springer sheaf, and analysis of how the cohomology of the character variety varies when the eigenvalues are moved around.
\end{abstract}

\maketitle

\tableofcontents

\section{Introduction}
\subsection{Character varieties}
Let $g\geq 0$, $k>0$ be integers, and let $C_1,\ldots,C_k$ be a collection of elements of $G=\GL_n(\BC)$. The corresponding character variety $X$ is the space which parametrizes local systems on the Riemann surface of genus $g$ with $k$ punctures with local monodromies around the punctures required to be conjugate to $C_1$,\ldots, $C_k$. Following \cite{hausel2011arithmetic}, we restrict our attention to the case when $C_i$ are diagonal (i.e. \emph{semisimple}), and the following \emph{genericity} assumption holds (Definition \ref{def:generic}):
\begin{enumerate}
	\item We have $\prod_{i=1}^k \det C_i=1$. This is certainly a necessary condition for $X$ to be non-empty.
	\item For any $n'$, $1< n'< n$, and any choice of $n'$ eigenvalues $\lambda^{(i)}_1, \ldots, \lambda^{(i)}_{n'}$ of $C_i$ for each $i$, we have $\prod_{i=1}^k\prod_{j=1}^{n'} \lambda^{(i)}_j\neq 1$.
\end{enumerate}
This guarantees that the character variety is an affine non-singular algebraic variety.

One example is given by $k=1$, $C_1=\zeta_n\Id_n$, where $\zeta_n=e^{\frac{2\pi i}{n}}$. This is the so-called \emph{twisted character variety} of genus $g$, studied in \cite{hausel2008mixed}. Another example, perhaps the most well-studied case, is $g=0$, $n=2$, $k=4$, $C_i=\begin{pmatrix}\lambda_i & 0 \\ 0 & \lambda_i^{-1} \end{pmatrix}$ goes back to Fricke (see \cite{fricke1965vorlesungen}). Denote $a_i=\lambda_i+\lambda_i^{-1}$. Then the character variety can be presented by a single equation in $3$ variables $x,y,z$:
\begin{multline}
x y z + x^{2} + y^{2} + z^{2} - (a_{1} a_{2} + a_{3} a_{4}) x - (a_{2} a_{3} +  a_{1} a_{4}) y - (a_{1} a_{3} + a_{2} a_{4}) z \\ + a_{1} a_{2} a_{3} a_{4}  +  a_{1}^{2} + a_{2}^{2} + a_{3}^{2} + a_{4}^{2} - 4=0.
\end{multline}
One can easily find this equation using computer algebra, and then verify that the surface is singular in the following cases: $\lambda_1=\pm 1$, $\lambda_1 \lambda_2=\lambda_3 \lambda_4$, $\lambda_1 \lambda_2 \lambda_3=\lambda_4$, $\lambda_1\lambda_2\lambda_3\lambda_4=1$, the complete list is obtained by permuting the indices. Cohomology of this variety together with some interesting group action was studied in \cite{goldman2005homological}. The analogous case with $5$ punctures has been studied by Simpson (\cite{simpson2017explicit}), by Donagi and Pantev, and perhaps by many others.

\subsection{Cohomology}
Following \cite{hausel2008mixed}, we want to understand the mixed Hodge structure on the cohomology of $X$. We work with \emph{cohomology with compact supports}, but since the character varieties we study are smooth, cohomology can be obtained from Poicar\'e duality. One way to package the numerical information about this mixed Hodge structure goes as follows. Let us define the \emph{weight polynomial} of $X$ by
\[
W(X;q,t) = \sum_{i,j} (-1)^j \dim \Gr^W_i H^j_c(X,\BC) q^{\frac{j}{2}} t^{\frac{j-i}{2}} \in \BQ[q^{\frac12}, t^{\frac12}].
\]
For instance, in the example with $g=0,n=2,k=4$ above we have $W(X;q,t)=q(q+t+4)$. In \cite{hausel2011arithmetic}, it was conjectured that the cohomologies of character varieties are Hodge-Tate, i.e. all Hodge numbers $h^{j,p,q}(X)$ for $p\neq q$ vanish. A conjectural closed form expression for $W$ of all semisimple generic character varieties was suggested. The formula was proven for the specialization $W(X;q,q^{-1})$, which gives the number of points of $X$ over $\BF_q$ and also the so-called Euler polynomial.

For fixed $g$ and $k$ the formula is
\begin{multline}\label{eq:hlv formula}
-(q-1)(t-1) \pLog \sum_{\lambda} \frac{\prod_{a,l} \left(q^{a+\frac12} - t^{l+\frac12}\right)^{2g} \prod_{i=1}^k \tilde H_\lambda[X_i;q,t]}{\prod_{a,l} (q^{a+1} - t^{l})(q^{a}-t^{l+1})}
\\
=\sum_{n=1}^\infty \sum_{|\mu^{(1)}|=\cdots=|\mu^{(k)}|=n} q^{-\frac{\dim}{2}} W(X^{g,n}_{\mu^{(1)},\ldots,\mu^{(k)}}; q,t) \prod_{i=1}^k m_{\mu^{(i)}}(X_i).
\end{multline}
On the left hand side we have the \emph{plethystic logarithm} of a sum over all partitions $\lambda$. For each $\lambda$, a product over $a,l$ means the products over the cells of $\lambda$, and for each cell $a$ is the arm length and $l$ is the leg length. Then we have the modified Macdonald polynomials $\tilde H_\lambda[X_i;q,t]$, each evaluated in its own set of variables $X_i$. On the right hand side, we have a sum over $k$-tuples of partitions $\mu^{(1)}, \ldots, \mu^{(k)}$ of same size $n$. The monomial symmetric function corresponding to a partition $\mu$ is denoted by $m_\mu$. The space $X^{g,n}_{\mu^{(1)},\ldots,\mu^{(k)}}$ is the rank $n$ character variety of the Riemann surface of genus $g$ with $k$ punctures, where the multiplicities of the eigenvalues of $C_i$ are prescribed by $\mu^{(i)}$ and assumed to be generic. $\dim$ stands for the complex dimension of the character variety.

The formula \eqref{eq:hlv formula} has also been established for the specialization $t=1$ by the author in \cite{mellit2017poincarea}.

Because of the symmetry $\tilde H_\lambda[X;q,t] = \tilde H_{\lambda'}[X;t,q]$ where $\lambda'$ denotes the conjugate partition, the left hand side of \eqref{eq:hlv formula} is symmetric in $q$, $t$. This gives a non-trivial implication that that $q^{-\frac{\dim}{2}} W(X^{g,n}_{\mu^{(1)},\ldots,\mu^{(k)}}; q,t)$ is symmetric in $q$ and $t$. For the dimensions of cohomology groups this implies the \emph{curious Poincar\'e duality}
\[
\dim \Gr_i^W H_c^j(X,\BC) = \dim \Gr_{2\dim-i}^W H_c^{j-i+\dim}(X,\BC).
\]

The cohomology ring of the twisted character variety is known to be generated by certain tautological classes by a result of Markman \cite{markman2002generators}. In particular, there is certain tautological class $\omega\in H^2(X,\BC)$. In \cite{hausel2008mixed}, it was suggested that the curious Poincar\'e duality would follow from the \emph{curious hard Lefschetz}, which claims that for each $m,j\geq 0$, the operator $(\cap\omega)^m$ induces an isomorphism
\[
	\Gr_{\dim-2m}^W H_c^j(X,\BC) \xrightarrow{\sim} \Gr_{\dim+2m}^W H_c^{j+2m}(X,\BC).
\]

It is believed that in the case with punctures, the cohomology also should be generated by some tautological classes, but details on construction of these classes and proof that the classes generate the cohomology have not been spelled out.

\subsection{$P=W$ conjecture}
By Simpson's non-abelian Hodge correspondence \cite{simpson1990harmonic}, the character variety is homeomorphic to a certain moduli space of stable Higgs bundles. Assume that the eigenvalues of $C_i$ have absolute value $1$ for all $i$. Denote the compact Riemann surface by $\Sigma$ and the divisor of punctures by $S=s_1+\cdots s_k$, $s_i\in \Sigma$. The corresponding moduli space is the moduli space of stable triples $(\CE, \theta, \CF)$ where $\CE$ is a holomorphic vector bundle on $\Sigma$, $\theta:\CE\to \CE\otimes\Omega(S)$ is a linear map, and $\CF=(\CF^{(1)},\ldots,\CF^{(k)})$ specifies an increasing filtration $\CF^{(i)}$ on the fiber $\CE(s_i)$ for each puncture $s_i$ satisfying $\theta(s_i) \CF^{(i)}_j \subset \CF^{(i)}_{j+1}$. The dimensions of the steps in $\CF^{(i)}$ match the multiplicities of the eigenvalues of $C_i$, and the eigenvalues themselves are translated into a stability condition. Denote this moduli space by $X_{D}$, where $D$ stands for ``Dolbeault'' to distinguish it from the character variety $X=X_B$, where $B$ stands for ``Betti''. The moduli space $X_D$ comes equipped with a $\BC^*$-action which scales $\theta$, and a characteristic polynomial map (Hitchin's map)
\[
\chi:X_D \to \BC^{\dim/2},
\]
which is proper. 

In \cite{cataldo2010topology}, in the case of twisted character variety, it was conjectured for arbitrary rank and proved in the rank $2$ case that the weight filtration of $H^i(X_B)$ coincides with the \emph{perverse Leray} filtration on $H^i(X_D)$. The filtration is induced by the map $X_D \to \BC^{\dim/2}$.

The perverse Leray filtration is constructed using the decomposition Theorem of \cite{beuilinson1982faisceaux} as follows. The derived push-forward of the constant sheaf $R \chi_* \BC_{X_D}$ admits a decomposition
\[
R \chi_* \BC_{X_D} \cong \bigoplus_m i_{m*} j_{m *!} L_m[-w_m],
\]
where $L_m$ is a local system on $U_m\subset Z_m\subset \BC^{\dim/2}$ where $U_m$ is an affine Zariski open subset of an irreducible Zariski closed subvariety $Z_m$, $j_m:U_m\to Z_m$, $i_m:Z_m\to \BC^{\dim/2}$ are natural embeddings, and $j_{m *!}$ is the intermediate extension functor. This induces a decomposition
\[
H^j_c(X_D,\BC) \cong \bigoplus_m IH_c^{j-w_m}(Z_m,L_m),
\]
where $IH_c(Z_m,L_m)$ is the intersection cohomology with compact supports of $Z_m$ with coefficients in $L_m$. The \emph{perversity} of each pair $L_m, Z_m$ is the number 
\[
p_m = w_m-\codim Z_m.
\]
This is the cohomological degree of  $i_{m*} j_{m *!} L_m[-w_m]$ with respect to the perverse t-structure. The perverse filtration is defined by
\[
P_i H^j_c(X_D,\BC) = \bigoplus_{m:p_m\leq i} IH_c^{j-w_m}(Z_m,L_m).
\]
The $P=W$ conjecture is the statement that under the identification $H^j_c(X_D,\BC)\cong H^j_c(X_B,\BC)$ we have
\[
P_i H^j_c(X_D,\BC)\cong W_{2i} H^j_c(X_B,\BC).
\]
The analogue of the curious hard Lefschetz property on the $X_D$ side is a direct consequence of the relative hard Lefschetz theorem (see \cite{beuilinson1982faisceaux}, \cite{cataldo2009hodge}).
\begin{thm}
Let $\omega$ be the Chern class of a relative ample line bundle for $\chi:X_D\to \BC^{\dim/2}$. Then for each $m,j\geq 0$, the operator $(\cap\omega)^m$ induces an isomorphism
\[
\Gr_{\dim/2-m}^P H_c^j(X_D,\BC) \xrightarrow{\sim} \Gr_{\dim/2+m}^P H_c^{j+2m}(X_D,\BC).
\]
\end{thm}

\subsection{Cell decompositions}
The standard approach to compute mixed Hodge structures is to compactify and resolve singularities. In \cite{totaro1996configuration} it is explained that sometimes even a bad compactification is good enough.  In this paper we avoid compactifications at all, and use cell decompositions instead.

The notion of a cell decomposition we use is a rather weak one. Suppose $X$ is a union of Zariski locally closed subsets $X_\alpha$ indexed by elements $\alpha$ of a partially ordered set $\CP$. We call it a cell decomposition if for any $\alpha_0\in\CP$ the union $\bigcup_{\alpha\leq \alpha_0} X_\alpha$ is closed.

The main ingredient in our results is a construction of a cell decomposition of the character variety $X$ under the following
\begin{assumption*}
	Suppose $C_k$ has distinct eigenvalues.
\end{assumption*}

In Sections \ref{sec:braid variety} and \ref{sec:decomposing} we show the following:
\begin{thm}\label{thm:cell decomposition}
	Under the above assumption, there is a vector bundle $\tilde X \to X$ of rank $\binom{n}2$ and a cell decomposition of $\tilde X$ such that each cell is isomorphic to $(\BC^*)^{\dim - 2 i} \times \BC^{i+\binom{n}2}$ for some $i$.
\end{thm}

The main idea of the construction goes as follows. The character variety classifies tuples of matrices $\alpha_1,\ldots,\alpha_g,\beta_1,\ldots,\beta_g,\gamma_1\ldots,\gamma_k$ satisfying
\[
\prod_{i=1}^g \alpha_i \beta_i \alpha_i^{-1} \beta_i^{-1} \prod_{i=1}^k \gamma_i^{-1} C_i \gamma_i = \Id,
\]
up to conjugation action of $\GL_n$ on $\alpha_i, \beta_i$ and right action on $\gamma_i$. Since $C_k$ has distinct eigenvalues, we can use the eigenvectors to form a basis. So $X$ is isomorphic to the variety classifying tuples $\alpha_1,\ldots,\alpha_g,\beta_1,\ldots,\beta_g,\gamma_1\ldots,\gamma_{k-1}$ satisfying
\[
\prod_{i=1}^g \alpha_i \beta_i \alpha_i^{-1} \beta_i^{-1} \prod_{i=1}^{k-1} \gamma_i^{-1} C_i \gamma_i = C_k,
\]
up to the action of diagonal matrices. Let $\tilde X$ be the variety obtained by replacing the above condition with the condition
\[
\left(\prod_{i=1}^g \alpha_i \beta_i \alpha_i^{-1} \beta_i^{-1} \prod_{i=1}^{k-1} \gamma_i^{-1} C_i \gamma_i\right) - C_k  \;\text{is strictly upper-triangular.}
\]
It is not hard to see that $\tilde X$ is a vector bundle over $X$ of rank $\binom{n}{2}$. The benefit of passing from $X$ to $\tilde X$ is that the difficult quotient $\GL_n/_\ad \GL_n$ is replaced by the easier double quotient $U\backslash \GL_n / U / T$ where we can apply the Bruhat decomposition, where $U$ is the subgroup of unipotent upper-triangular matrices and $T$ is the torus of diagonal matrices. Passing from $X$ to $\tilde X$ can be thought of as the \emph{horocycle transform}. 

The cell decomposition of $\tilde X$ is performed in two steps. For each point of $\tilde X$ we have a tuple of permutations $\pi^\alpha_1,\ldots,\pi^\alpha_g, \pi^\beta_1,\ldots,\pi^\beta_g, \pi^\gamma_1\ldots,\pi^\gamma_{k-1}$ telling us in which Bruhat cell each of the matrices $\alpha_i, \beta_i, \gamma_i$ lies. This defines a cell decomposition of $\tilde X$ indexed by the tuples of permutations
\[
\{\pi^\alpha_1,\ldots,\pi^\alpha_g, \pi^\beta_1,\ldots,\pi^\beta_g, \pi^\gamma_1\ldots,\pi^\gamma_{k-1}\;:\pi^\alpha_i,\pi^\beta_i\in S_n,\;\pi^\gamma_i\in S_n/S_{\mu^{(i)}} \}.
\]
Here $S_n$ is the permutation group, $\mu^{(i)}$ is the partition specifying the multiplicities of the eigenvalues of $C_i$, and $S_\mu=S_{\mu_1}\times\cdots S_{\mu_l}$. Elements in the quotient $S_n/S_{\mu^{(i)}}$ are represented by minimal representatives.

Each cell looks like the product of affine space and a certain variety associated to the Artin braid $b$ obtained by composing the positive lifts of the permutations 
\[
\pi^\alpha_1, \pi^\beta_1, (\pi^\alpha_1)^{-1}, (\pi^\beta_1)^{-1}, \ldots, \pi^\alpha_g, \pi^\beta_g, (\pi^\alpha_g)^{-1}, (\pi^\beta_g)^{-1}, (\pi^\gamma_1)^{-1}, \pi^\gamma_1,\ldots,(\pi^\gamma_{k-1})^{-1}, \pi^\gamma_{k-1}.
\]
When $g>0$ there are also certain genus corrections which only introduce factors of the form $\BC^*$. These braid varieties are very similar to the varieties studied in \cite{shende2017legendrian}. In particular, they admit cell decompositions that look like products $(\BC^*)^i \times \BC^j$ for $i,j\geq 0$. These cell decompositions are indexed by objects that we call \emph{walks}. A walk associated to a braid is a sequence of permutations where for each intersection of the braid we can choose if we apply the corresponding transposition or we skip it, satisfying the rule
\[
\text{if we would go up, we must go up, but if we would go down we can skip.}
\]
Here ``go up'', ``go down'' refers to the Bruhat order. The walk begins at the identity permutation and must end in the identity. The details are explained in  Section \ref{sec:braid variety}.

For example, in the case of $g=0, k=4, n=2$, we have $8$ possible triples of permutations
\[
(\tau, \tau, \tau),\; (\tau, \tau, \Id),\; (\tau, \Id, \tau),\; (\tau, \Id, \Id),\; (\Id, \tau, \tau),\; (\Id, \tau, \Id),\; (\Id, \Id, \tau),\;(\Id,\Id,\Id),
\]
where $\tau\in S_2$ is the transposition. Because $\tau=\tau^{-1}$, from $(\tau,\tau,\tau)$ we obtain the positive braid on $2$ stings with $6$ crossings. From the remaining tuples we obtain $3$ braids with $4$ crossings, $3$ braids with $2$ crossings, and $1$ trivial braid. For each braid we enumerate walks, for instance for a braid with $2$ crossings there is only one walk $\Id, \tau, \Id$. For a braid with $4$ crossings we have two walks:
\[
(\Id, \tau, \tau, \tau, \Id)\quad\text{(go up, stay, stay, go down),}
\]
\[
(\Id, \tau, \Id, \tau, \Id)\quad\text{(go up, go down, go up, go down).}
\]
For a braid with $6$ crossings we have the following walks:
\[
(\Id, \tau, \tau, \tau, \tau, \tau, \Id)\quad\text{(go up, stay, stay, stay, stay, go down),}
\]
\[
(\Id, \tau, \tau, \tau, \Id, \tau, \Id)\quad\text{(go up, stay, stay, go down, go up, go down),}
\]
\[
(\Id, \tau, \tau, \Id, \tau, \tau, \Id)\quad\text{(go up, stay, go down, go up, stay, go down),}
\]
\[
(\Id, \tau, \Id, \tau, \tau, \tau, \Id)\quad\text{(go up, go down, go up, stay, stay, go down),}
\]
\[
(\Id, \tau, \Id, \tau, \Id, \tau, \Id)\quad\text{(go up, go down, go up, go down, go up, go down).}
\]

Each braid variety $X_b$ naturally comes equipped with two structures: an action of the torus $T$ and a map to the torus $f:X_b\to T$. The collection of matrices $C_1,\ldots,C_k$ gives a point $t\in T$, and to obtain the corresponding cell of $\tilde X$ we need to take the fiber $f^{-1}(t)$, and divide by the $T$-action. The genericity assumption on $C_i$ translates into the condition that $t\in T$ does not belong to certain natural subtori. Some cells of the cell decomposition of the braid variety are mapped to these subtori, so we can ignore them. For each walk we have a graph on vertices $1,2,\ldots,n$. Each time we stay we connect the corresponding indices $i,j$ which would have been swapped. The walks producing disconnected graphs can be ignored.

In our running example, we can ignore all walks without stays. Thus we obtain $4$ walks for the braid with $6$ crossings and one walk for each of the $3$ braids with $4$ crossings. The first walk for the braid with $6$ crossings produces a $3$-dimensional cell isomorphic to $\BC^*\times \BC^* \times \BC$. Each of the other six walks produces a $2$-dimensional cell isomorphic to $\BC\times\BC$. So we obtain that a rank $1$ vector bundle over $X$ is isomorphic to $\BC^*\times \BC^* \times \BC\sqcup 6\; \BC\times\BC$. In this case (and in all cases of rank $2$) it is possible to cancel out $\BC$, and obtain
\[
X \cong \BC^*\times \BC^* \sqcup 6\; \BC.
\]

It is not clear if the vector bundle $\tilde{X}\to X$ can be ``cancelled'' with $\BC^{\binom{n}2}$ to obtain a cell decomposition for $X$. From the point of view of cohomology there is no difference between $X$ and $\tilde X$.

A direct corollary of Theorem \ref{thm:cell decomposition} and the long exact cohomology sequence is 
\begin{cor}
	If $C_k$ has distinct eigenvalues, the cohomology groups $H^j(X,\BC)$ are Hodge-Tate, i.e. the pure Hodge structure on $\Gr^W_i H^j(X,\BC)$ has only $p,q$-parts with $p=q$.
\end{cor}

\subsection{Main results}
Before proceeding to the main results, we clarify the situation with tautological classes for the character varieties with punctures. There is a tautological $\PGL_n$-bundle $\CE$ on $X\times \Sigma$. The usual way to produce cohomology classes on $X$ is to integrate the characteristic classes of $\CE$ along homology classes of $\Sigma$. This works well for the twisted character variety. In \cite{shende2016weights}, Shende showed that the tautological classes constructed in this way have expected Hodge weights. In the situation with punctures, we have the following kinds:
\begin{enumerate}
	\item Each puncture produces a principal bundle with structure group $Z(C_i)/\BC^*$, where $Z(C_i)$ denotes the centralizer of $C_i$. We can take the characteristic classes of this bundle.
	\item By computing the monodromy along a closed loop $\gamma$ on $\Sigma$, we have a map of stacks $X \to \GL_n/_\ad\GL_n$, and the cohomology classes on $\GL_n$ are obtained by transgression from the classifying space $*/\GL_n$. From the point of view presented in Section \ref{sec:symplectic form}, it makes sense to speak of a shifted map $X[1]\to */\GL_n$, and cohomology classes on $X$ can be obtained by the pullback $H^{2i}(*/\GL_n,\BC)\to H^{2i-1}(X,\BC)$.
	\item As explained in Section \ref{ssec:parabolic char stack}, there is a stack corresponding to the mapping cone
	\[
	BG_\parab := \cone\left(\bigsqcup_{i=1}^k */Z(C_i)[1] \to */\GL_n\right)
	\]
	together with a map $X[2]\to BG_\parab$. The cohomology of $BG_\parab$ fits into the exact sequence
	\[
	0\to \bigoplus_{m=1}^k H^{2i-2}(*/Z(C_i), \BC) \to H^{2i}(BG_\parab, \BC) \to H^{2i}(*/\GL_n,\BC) \to 0
	\]
	for each $i$. If we have split this sequence, for instance using the Hodge filtration, we obtain a map $H^{2i}(*/\GL_n,\BC)\to H^{2i-2}(X,\BC)$. This map does not have to preserve the mixed Hodge structure. So forms obtained in this way do not need to have pure weight.
\end{enumerate}
In Section \ref{sec:symplectic form} we follow the construction (iii) to obtain an explicit holomorphic closed $2$-form $\omega$ on $X$, representing a tautological class in $H^2(X,\BC)$. Our expression for the form is essentially the same as given in \cite{guruprasad1997group}. We show that its class lies in $F^2 H^2(X,\BC^*)$.

Most of the work done in Sections \ref{sec:braid variety}---\ref{sec:decomposing} is devoted to verification that the restriction of $\omega$ to each cell in our cell decomposition is given by an expression
\[
\sum_{i,j} \omega_{i,j} \frac{dt_i}{t_i}\wedge \frac{dt_j}{t_j}
\]
in some coordinates on $(\BC^*)^m$, where the anti-symmetric matrix $(\omega_{i,j})$ is \emph{non-degenerate}. This is achieved by following the suggestion of Shende to compare this matrix with the intersection matrix on $H_1$ of a certain \emph{Seifert surface}. We give a combinatorial construction of this surface from the data of a tuple of permutations and a walk. 

As explained in Section \ref{sec:generalities}, this implies
\begin{thm}[Curious hard Lefschetz under Assumption] \label{thm:curious hard lefschetz}
	If $C_k$ has distinct eigenvalues, then for each $m,j\geq 0$, the operator $(\cap\omega)^m$ induces an isomorphism
	\[
	\Gr_{\dim-2m}^W H_c^j(X,\BC) \xrightarrow{\sim} \Gr_{\dim+2m}^W H_c^{j+2m}(X,\BC).
	\]
\end{thm}

In Section \ref{sec:monodromic action}, we deduce the general case from the situation under the assumption that $C_k$ has distinct eigenvalues. In \cite{letellier2015character} Letellier studied the $W$-action ($W=S_n$) on the cohomology of the unipotent character variety. So the monodromy around the $k$-th puncture is unipotent and we have the extra data of a compete flag preserved by this monodromy. Letellier's conjectures imply that the cohomology of the character variety for $C_k=\Id$ equals to the sign component of this action. We prove this statement and deduce the curious hard Lefschetz in the general case. We consider the family of character varieties $\CX\to T_1$, where for each $t\in T_1=\{t\in T:\det t=1\}$ the monodromy around the $k$-th puncture has eigenvalues prescribed by $t$. The main observation (Proposition \ref{prop:locally constant} is that each cell in our decomposition factors as a direct product where one factor is $T_1$ and the map to $T_1$ is the projection. So the derived $!$-pushforward of the constant sheaf from the cell to $T_1$ is constant. Thus the derived $!$-pushforward of the constant sheaf from $\CX$ to $T_1$ has locally constant cohomology over the locus where the collection $C_1,\ldots,C_{k-1},t$ is generic, which includes $1$ if $C_1,\ldots,C_{k-1},1$ is generic.

We obtain
\begin{thm}
	For arbitrary generic semisimple $C_1,\ldots,C_k$, the cohomology groups $H^j(X,\BC)$ are Hodge-Tate.
\end{thm}

\begin{thm}[Curious hard Lefschetz]
	For arbitrary generic semisimple $C_1,\ldots,C_k$ and for each $m,j\geq 0$, the operator $(\cap\omega)^m$ induces an isomorphism
	\[
	\Gr_{\dim-2m}^W H_c^j(X,\BC) \xrightarrow{\sim} \Gr_{\dim+2m}^W H_c^{j+2m}(X,\BC).
	\]
\end{thm}

As an interesting corollary of the fact that $\omega\in F^2 H^2(X,\BC)$, we have (see Proposition \ref{prop:tate and lefschetz})
\begin{cor}
	The canonical filtration induced by the nilpotent operator $\cap\omega$ coincides with the Hodge 
	filtration and splits the weight filtration. The mixed Hodge structure splits, i.e. $F^i H^j_c(X,\BC)=\bar F^i H^j_c(X,\BC)$ for all $i,j$, if and only if $\omega$ has pure weight, i.e. $\omega \in \bar F^2 H^2(X,\BC)\cap F^2 H^2(X,\BC)$.
\end{cor}

It is clear from the way the class of $\omega$ was constructed as a pull back from $BG_\parab$, that in the case when all the eigenvalues of $C_i$ have absolute value $1$ $\omega$ has pure weight. We expect that this is not true in general, see Remark \ref{rem:split or does not}.

The proof of the curious hard Lefschetz property opens a way to the following line of attack on the $P=W$ conjecture: Suppose we have a proof of $P_i\subset W_{2i}$ or $W_{2i}\subset P_i$. Then the Lefschetz properties imply $P_i=W_{2i}$.

\section{Generalities}\label{sec:generalities}

\subsection{}\label{ssec:poset}
Any partially ordered set (poset) $\CP$ is naturally endowed with a topology so that
\[
U\subset \CP \;\text{is open if and only if}\; x\in U, y>x \Rightarrow y\in U,
\]
\[
Z\subset \CP \;\text{is closed if and only if}\; x\in Z, y<x \Rightarrow y\in Z.
\]
Note that continuous maps between posets are precisely order-preserving maps.

\begin{defn}
Let $X$ be a topological space. A \emph{cell decomposition} of $X$ is a pair $(\CP, \varphi)$ where $\CP$ is a finite poset and $\varphi:X\to \CP$ is a continuous map. 
\end{defn}

Explicitly, the data of a map $\varphi:X\to\CP$ is equivalent to a decomposition of $X$ into a union of subsets indexed by $\CP$. Denote
\[
\varphi_\sigma=\varphi^{-1}(\sigma),\; \varphi_{\leq \sigma}=\bigcup_{\alpha\leq \sigma} \varphi^{-1}(\alpha),\; \varphi_{\geq \sigma}=\bigcup_{\alpha\geq \sigma} \varphi^{-1}(\alpha),
\]
and similarly $\varphi_{<\sigma}$, $\varphi_{>\sigma}$. We have
\begin{prop}
	Let $X$ be a topological space, let $\CP$ be a finite poset, and let $\varphi:X\to\CP$ be a map. The following conditions are equivalent:
	\begin{enumerate}
		\item $(\CP,\varphi)$ is a cell decomposition;
		\item $\varphi_{\leq\sigma}$ is closed for every $\sigma\in\CP$;
		\item $\varphi_{\geq\sigma}$ is open for every $\sigma\in\CP$.
	\end{enumerate}
\end{prop}
Suppose $(\CP,\varphi)$ is a cell decomposition. Then sets of the form $\varphi_{<\sigma}$ are also closed, and sets of the form $\varphi_{>\sigma}$ are open. The \emph{cells} $\varphi_\sigma$ are locally closed. Thus $X$ is decomposed into a disjoint union of locally closed cells. Note that we do not require that $\overline{\varphi}_\sigma=\varphi_{\leq \sigma}$, which is usually required from a cell decomposition. An example to keep in mind when this is not true is given by 
\[
X=\{(x,y)\in \BC^2:\, xy=0\},\;\CP=\{0<1\},\;\varphi_0=\{(x,y):\, x=0\},\; \varphi_1=X\setminus \varphi_0.
\]
In this example $\overline{\varphi}_1=\{(x,y):\, y=0\}\neq X$. It is also not true in general that $\alpha\leq \beta$ implies $\dim\varphi_\alpha\leq \dim\varphi_\beta$. For an example with same $\CP$ as above, take $X$ the union of the plane $z=0$ and the line $x=y=0$ in $\BC^3$, take $\varphi_0$ the plane and $\varphi_1$ its complement.

\subsection{} Given two cell decompositions $(\CP, \varphi)$, $(\CP',\varphi')$ of $X$, we can construct their \emph{common refinement} as a product $(\CP\times\CP', \varphi\times\varphi')$.

Another useful construction goes as follows. Suppose $(\CP, \varphi)$ is a cell decomposition of $X$, and suppose each cell $\varphi_\alpha$ has a cell decomposition $(\CP^{(\alpha)}, \varphi^{(\alpha)})$. Let $\tilde\CP$ be the poset of pairs $(\alpha, \beta)$ where $\alpha\in\CP$, $\beta\in\CP^{(\alpha)}$ with lexicographic order:
\[
(\alpha, \beta)\leq (\alpha',\beta) \;\Leftrightarrow\; \alpha<\alpha'\;\text{or}\; \alpha=\alpha',\beta\leq\beta'.
\]
There is a natural continuous map $\tilde\varphi:X\to \tilde\CP$ so that cells of the cell decomposition $(\tilde\CP,\tilde\varphi)$ are precisely the cells of the cell decompositions of the cells of $(\CP, \varphi)$.

Note that we can always replace $\CP$ by $\image(\varphi)$ for a cell decomposition $(\CP,\varphi)$. We will consider the resulting cell decomposition indistinguishable from the original one.

Finally, we note that any partial order can be refined to a linear one, usually in many ways. Let $[n]$ denote the poset $\{1<\cdots<n\}$. Linear refinements are continuous maps $\CP\to [|\CP|]$. Each such map can be composed with $\varphi$ to obtain a cell decomposition indexed by $[|\CP|]$. The data of a cell decomposition indexed by $[n]$ is the same thing as the sequence of open subsets of $X$
\[
U_1\subset U_2 \subset \cdots \subset U_n=X,
\]
where the successive complements $U_i\setminus U_{i-1}$ are the cells of $(\CP, \varphi)$.

\subsection{} For $X$ an algebraic variety over $\BC$, $U\subset X$ a Zariski open subset and $Z=X\setminus U$, we have a long exact sequence in compactly supported cohomology (with coefficients in $\BC$ or $\BQ$:
\[
\cdots\to H^i_c(U) \to H^i_c(X) \to H^i_c(Z) \to H^{i+1}_c(U) \to \cdots. 
\]
For a cell decomposition indexed by $[n]$ we can combine all these exact sequences into a spectral sequence that looks as shown on Figure \ref{fig:spectral sequence}.
\begin{figure}
\[
\begin{tikzcd}
\vdots & \vdots & \vdots & \vdots \\
H^2_c(\varphi_1) \arrow[r] \arrow[rrd] & H^3_c(\varphi_2) \arrow[r] & H^4_c(\varphi_3) & \cdots\\
H^1_c(\varphi_1) \arrow[r] \arrow[rrd] & H^2_c(\varphi_2) \arrow[r] & H^3_c(\varphi_3) & \cdots \\
H^0_c(\varphi_1) \arrow[r] \arrow[rrd] & H^1_c(\varphi_2) \arrow[r] & H^2_c(\varphi_3) & \cdots \\
 & H^0_c(\varphi_2) \arrow[r] & H^1_c(\varphi_3) & \cdots\\
 & & H^0_c(\varphi_3) & \cdots
\end{tikzcd}
\]
\caption{The spectral sequence}
\label{fig:spectral sequence}
\end{figure}

\subsection{Weight filtration} Cohomology and compactly supported cohomology of a complex algebraic variety is equipped with a mixed Hodge structure (\cite{deligne1974theorie}, \cite{peters2008mixed}). Since we do not study any rationality questions here, all cohomologies are assumed with complex coefficients. A pure Hodge structure of weight $k$ is a $\BC$-vector space $V$ equipped with a direct sum decomposition
\[
V\times\BC = \sum_{p+q=k} V^{p,q}.
\]
A \emph{mixed Hodge structure} (MHS) is a $\BC$-vector space $V$ equipped with an increasing weight filtration $W_\bullet V$ and two decreasing filtrations $F^\bullet V$, $\overline{F}^\bullet V$ on $V$ such that the filtrations induced by $F^\bullet$, $\bar F^\bullet$ on $W_k/W_{k-1}\otimes \BC$ are $k$-opposed, i.e. $W_k/W_{k-1}$ is a pure Hodge structure of weight $k$,
\[
W_k/W_{k-1}\otimes \BC = \bigoplus_{p+q=k} V^{p,q},
\]
and we have
\[
(W_k\cap F^i + W_{k-1})/W_{k-1} = \bigoplus_{p+q=k,\,q\geq i} V^{p,q},
\]
\[
(W_k\cap \bar F^i + W_{k-1})/W_{k-1} = \bigoplus_{p+q=k,\,p\geq i} V^{p,q}.
\]
Note that we do not require that $\bar F^i$ is complex conjugate to $F^i$. The main properties of the category of mixed Hodge structure are summarized in 

\begin{prop}\label{prop:MHSs}
The category of MHSs, viewed as the full subcategory of the category of vector spaces endowed with three filtrations, is abelian. The functors $V\to V^{p,q}$ are exact for all $p,q$.
\end{prop}

Useful implications of this are

\begin{cor}
If $V, V'$ are MHSs and $f:V\to V'$ is a map preserving the three filtrations such that $f(W_i V) \subset W_{i-1} V'$, then $f=0$.
\end{cor}

\begin{cor}
If $V, V'$ are MHSs and $f:V\to V'$ is a map which preserves the three filtrations and is an isomorphism of vector spaces, then $f$ is an isomorphism of MHSs.
\end{cor}

We are mostly interested in the following special case:
\begin{defn}
A MHS $(V, W_\bullet, F^\bullet, \overline{F}^\bullet)$ is \emph{Tate} if $V^{p,q}=0$ for $p\neq q$.
\end{defn}

For such MHSs we have
\begin{prop}
A vector space $V$ with three filtrations $W_\bullet V$, $F^\bullet V$, $\overline{F}^\bullet V$ is a Tate MHS if and only if the following conditions are satisfied:
\begin{itemize}
\item The weights are even, i.e. $W_{2k+1}=W_{2k}$ for all $k\in\BZ$.
\item The filtration $F^\bullet$ splits the filtration $W_{2\bullet}$.
\item The filtration $\overline{F}^\bullet$ splits the filtration $W_{2\bullet}$.
\end{itemize}
\end{prop}

This motivates
\begin{defn}
A MHS $(V, W_\bullet, F^\bullet, \overline{F}^\bullet)$ is \emph{split Tate} if it is Tate and $F^\bullet = \overline{F}^\bullet$.
\end{defn}

We extend our definitions to algebraic varieties: If $X$ is an algebraic variety, we denote by $H^k(X)$ resp. $H_c^k(X)$ the cohomology resp. compactly supported cohomology with coefficients in $\BC$.
\begin{defn} An algebraic variety $X$ is Tate, resp. split Tate if $H_c^k(X)$ is Tate, resp. split Tate for all $k$.
\end{defn}

The category of Tate MHSs is a full abelian subcategory of the category of all MHSs, and it is closed under extensions. This implies the following:
\begin{prop}\label{prop:cell decomposition tate}
If an algebraic variety $X$ admits a cell decomposition such that all cells are Tate, then $X$ is Tate.
\end{prop}

Note that if all cells are split Tate, $X$ is not necessarily split Tate.

\subsection{An operator of degree $2$ with a filtration}\label{ssec:operator with filtration}
Here we recall linear algebra results concerning pairs consisting of a filtered vector space and an operator of degree $2$. Let $V$ be a vector space. Let $F_\bullet$ resp. $F^\bullet$ be a finite increasing resp. decreasing filtration on $V$. Let $\omega:V\to V$ be an operator of degree $2$, i.e. 
\[
\omega F_i\subset F_{i+2}\quad\text{resp.}\quad \omega F^i\subset F^{i+2}.
\]
\begin{defn}
We say that $\omega:V\to V$ satisfies \emph{Lefschetz property} with \emph{middle weight} $k$ with respect to an increasing resp. decreasing filtration $F_\bullet$ resp. $F^\bullet$ if it has degree $2$ and for all $i\in\BZ_{>0}$ we have an isomorphism
\[
\omega^i:F_{k-i}/F_{k-i-1} \to F_{k+i}/F_{k+i-1}\quad\text{resp.}\quad
\omega^i:F^{k-i}/F^{k-i+1} \to F^{k+i}/F^{k+i+1}.
\]
\end{defn}
Shifting degrees one can always assume that the middle weight is $0$.
Although the definition in the increasing/decreasing case is essentially the same, the behavior is completely different. For the decreasing case note that the operator must be nilpotent since $F^i=0$ for $i$ big enough. We have
\begin{prop}\label{prop:monodromic}
For a given finite dimensional vector space $V$ and a nilpotent operator $\omega:V\to V$ there exists a unique decreasing filtration $F^\bullet$, called \emph{monodromic filtration}, of middle weight $0$ with respect to which $\omega$ satisfies Lefschetz property.
\end{prop}
\begin{proof}
For the uniqueness, notice that $F^i\subset \image \omega^i$ for $i\geq 0$. Thus for all $i,j\geq 0$ we have $\omega^j F^{i-2j}\subset\image\omega^i$, so for $i\geq j\geq0$ we have $F^{i-2j}\subset \kernel\omega^j + \image\omega^{i-j}$. Hence 
\[
F^{i-j}\subset \kernel\omega^j+\image\omega^i\qquad(i,j\geq 0).
\]
Dually, we have $\kernel\omega^j \subset F^{-j+1}$ for $j\geq 0$, which implies $\omega^i \kernel\omega^j\subset F^{-j+1+2i}$ for $i,j\geq 0$ and $\kernel \omega^{j-i} \cap \image\omega^i\subset F^{-j+1+2i}$ for $j\geq i\geq 0$. Hence
\[
\kernel \omega^j\cap\image\omega^i\subset F^{i-j+1}\qquad(i,j\geq 0).
\]
For all $m$ (we set $\kernel\omega^j=0$ and $\image\omega^j=V$ for $j\leq 0$) we obtain
\[
F^m\subset (\kernel \omega^0 + \image\omega^{m})\cap(\kernel \omega^1+\image\omega^{m+1})\cdots=\image\omega^{m}\cap(\kernel\omega^1+\image\omega^{m+1}\cap(\cdots
\]
\[
=\image\omega^m\cap\kernel\omega^1+\image\omega^{m+1}\cap\kernel\omega^2+\cdots\subset F^m.
\]
Thus uniqueness follows and we leave it to the reader to verify that
\[
F^m = \bigcap_{i-j=m} (\image\omega^i+\kernel\omega^j) = \sum_{i-j=m} \image\omega^i\cap\kernel\omega^{j+1}
\]
satisfies Lefschetz property.
\end{proof}
So in the decreasing case the filtration is uniquely determined by the operator. In the increasing case there can be many filtrations for a given operator. For instance, the filtration $F_{-1}=0$, $F_0=V$ always works with middle weight $0$. However, one can use $\omega$ to split the filtration. This is called \emph{Deligne splitting}, see \cite{cataldo2009hodge}, \cite{cataldo2010topology}. 
\begin{prop}\label{prop:deligne splitting}
Suppose $F_\bullet$ is an increasing filtration on $V$ and $\omega:V\to V$ satisfies the Lefschetz property with middle weight $k$. Then there is a unique direct sum decomposition
\[
V = \bigoplus_{i \geq j \geq 0} Q^{i,j}\quad\text{with}\quad F_m = \bigoplus_{k-i+2j\leq m} Q^{i,j}
\]
satisfying
\[
\omega^j:Q^{i,0} \xrightarrow{\sim} Q^{i,j}\quad (i \geq j \geq 0),\qquad \omega^{i+r+1} Q^{i,0} \subseteq  F_{k+i+r}\quad (i,r\geq 0).
\]
\end{prop}
\begin{proof}
Note that given $Q^{i,0}$, we can uniquely recover $Q^{i,j}$ as $w^j Q^{i,0}$. The conditions imply $Q^{i,0}\subset F_{k-i}$. The conditions also imply
\[
\omega F_m =  \sum_{\substack{k-i+2j\leq m+2\\j\geq 1}} Q^{i,j} \pmod {F_m}.
\]
Hence abstractly, we have  
\begin{equation}\label{eq:isomorphism 1}
Q^{i,0}\cong F_{k-i}/(F_{k-i-1}+\omega F_{k-i-2})\qquad(i\geq 0).
\end{equation}
In order to realize $Q^{i,0}$ as a subspace of $V$, we have to provide a splitting of the surjective map $F_{k-i}\to Q^{i,0}$. First note that modulo $F_{k-i-1}+\omega F_{k-i-3}=F_{k-i-1}$ the splitting exists and is determined uniquely by vanishing of $\omega^{i+1} Q^{i,0}$ in $F_{k+i+2}/F_{k+i+1}$. So we uniquely upgrade \eqref{eq:isomorphism 1} to an embedding
\[
Q^{i,0} \subset F_{k-i}/F_{k-i-1}\quad\text{such that}\quad \omega^{i+1} Q^{i,0}\subset F_{k+i+1}.
\]
Then we continue to upgrade the embedding as follows. At step $r=0,1,2,\ldots$ we have an embedding
\[
Q^{i,0} \subset F_{k-i}/F_{k-i-r-1}\quad\text{such that}\quad \omega^{i+r+1} Q^{i,0}\subset F_{k+i+r+1}.
\]
Using the isomorphism 
\[
\omega^{i+r+1}:F_{k-i-r-1}/F_{k-i-r-2} \xrightarrow{\sim} F_{k+i+r+1}/F_{k+i+r}
\]
we uniquely upgrade the embedding to 
\[
Q^{i,0} \subset F_{k-i}/F_{k-i-r-2}\quad\text{such that}\quad \omega^{i+r+1} Q^{i,0}\subset F_{k+i+r}.
\]
Then we have $\omega^{i+r+2} Q^{i,0}\subset F_{k+i+r+2}$ and we can continue to step $r+1$.
\end{proof}

If we start with any vector $f\in Q^{i,0}$ and apply $\omega$ to it several times, then we see that initially we will go up the filtration two steps at a time, and after $i$ iterations only one step at a time. If $\omega^{i+1} f$ has a component $0\neq f_{i',j'}\in Q^{i',j'}$, then we necessarily have $\omega^{i'-j'} f_{i',j'}\in F_{k+i+i'-j'}$, which implies $k+i'\leq k+i+i'-j'$, so $j'\leq i$. Conversely, for any choice of vector spaces $Q^{i,0}$ and maps
\[
\omega_i:Q^{i,0} \to \bigoplus_{0\leq j'\leq \min(i,i')} Q^{i',j'} \cong \bigoplus_{0\leq i'} \left(Q^{i',0}\right)^{\oplus (\min(i,i')+1)}
\]
we can uniquely define $\omega$ on $V=\bigoplus_{i\geq 0} \left(Q^{i,0}\right)^{\oplus (i+1)}$ so that $\omega^{i+1}|_{Q^{i,0}}=\omega_i$ and the decomposition of $V$ satisfies the properties of Deligne splitting. Thus we have obtained a complete classification of triples $(V,F_\bullet,\omega)$ such that $V$ is a finite dimensional vector space, $F_\bullet$ is an increasing filtration and $\omega:V\to V$ satisfies the Lefschetz property. The following special case needs a separate name:
\begin{defn}
We say that $\omega:V\to V$ satisfies \emph{split Lefschetz property} with \emph{middle weight} $k$ with respect to an increasing filtration $F_\bullet$ if it satisfies Lefschetz property and additionally one of the following equivalent conditions:
\begin{enumerate}
\item The Deligne splitting (Proposition \ref{prop:deligne splitting}) satisfies $\omega^{i+1} Q^{i,0}=0$ for all $i\geq 0$.
\item There is a grading $V=\bigoplus_i G_i$ so that $\omega G_i\subset G_{i+2}$ and $F_i=\bigoplus_{j\leq i} G_j$.
\item The monodromic filtration of $\omega$ (Proposition \ref{prop:monodromic}) splits $F_\bullet$.
\end{enumerate}
\end{defn}
Equivalence of these conditions is clear from the above. Finally, we note the following property, which is easy to prove:
\begin{prop}
	If $\omega:V\to V$ satisfies Lefschetz property with middle weight $k$ with respect to two increasing filtrations $F_\bullet$, $F_\bullet'$, and if $F_i\subset F_i'$ holds for all $i$, then $F_\bullet = F_\bullet'$.
\end{prop}

\subsection{Weight filtration and a two-form}
Let $X$ be an algebraic variety of dimension $d$ over $\BC$. For $k\in\BZ_{\geq 0}$ we have
\[
0=W_0 H^k_c(X) \subset \cdots W_k H^k_c(X)=H^k_c(X).
\]
Let $\omega\in H^2(X)$ be any class.
We know that (Theorem 5.39 in \cite{peters2008mixed}) $H^2(X)=W_4 H^2(X)$, and $W_4 H^2/W_3 H^2$ only can have non-trivial $2,2$-component, which means 
\[
W_4 = W_3 + F_2,\quad W_4 = W_3 + \ol F_2.
\]
Denote by $[\omega]$ the projection of the class of $\omega$ to $W_4/W_3$. It belongs to $F_2$ and $\ol F_2$.

There is a cup product (see Section 6.3 in \cite{peters2008mixed}) that respects MHSs:
\[
H^2(X)\otimes H^k_c(X) \to H^{k+2}_c(X).
\]
Thus we obtain that the map $\cup\omega$ acts as
\[
\cup \omega: W_i H^j_c(X) \to W_{i+4} H^{j+2}_c(X).
\]
Passing to the associated graded with respect to $W_\bullet$ we obtain a map of degrees $(2,2,2)$:
\begin{equation}\label{eq:omega associated graded}
\cup [\omega]: H^k_c(X)^{p,q}\to H^{k+2}_c(X)^{p+2,q+2}.
\end{equation}
We are going to apply theory developed in Section \ref{ssec:operator with filtration} to the situation of vector space $H^\bullet_c(X)=\bigoplus_{k=0}^{2 \dim X} H^k_x(X)$, filtration $W_\bullet$ and operator $\cup\omega$.

\begin{defn}
We say that $\omega\in H^2(X)$ satisfies \emph{curious Lefschetz property} with middle weight $2 d$ if all weights of $X$ are even and the operator $\cup \omega: H^\bullet_c(X) \to H^\bullet_c(X)$
satisfies Lefschetz property with respect to the increasing filtration $W_{2\bullet}$, i.e. for all $k,i\geq 0$ we have an induced isomorphism
\[
(\cup\omega)^i: W_{2d-2i} H^k_c(X)/W_{2d-2i-1} H^k_c(X) \toiso W_{2d+2i} H^{k+2i}_c(X)/W_{2d+2i-1} H^{k+2i}_c(X).
\]
If moreover $\cup\omega$ satisfies split Lefschetz property, then we say that $\omega$ satisfies \emph{split curious Lefschetz property}.
\end{defn}
The property was discovered in \cite{hausel2008mixed}. It is called curious because the action on cohomological degrees is not like in the usual hard Lefschetz theorem.

\begin{rem}
Suppose we have a curious Lefschetz property. Let $\BC^*$ act on $H^\bullet_c(X)$ by the degree $k$ action on $H^k_c(X)$ for each $k$. Then $\omega$ is conjugated by degree $2$ action. By uniqueness, Deligne splitting of Proposition \ref{prop:deligne splitting} must be preserved by the action. Thus we have direct sum decompositions 
\[
H_c^k(X)=\bigoplus_{0\leq j\leq i} Q^{i,j,k}
\]
for each $k$.
\end{rem}

\begin{rem}
The operator $\cup\omega$ satisfies curious Lefschetz if and only if the operator $\cup[\omega]$ from \eqref{eq:omega associated graded} does.
\end{rem}

In parallel to Proposition \ref{prop:cell decomposition tate} we have
\begin{prop}\label{prop:weak cell lefschetz}
Suppose an algebraic variety $X$ has a cell decomposition $(\CP, \varphi)$. Let $\omega\in H^2(X)$ be a class and let $d$ be a number such that the restriction of $\omega$ to each cell $\varphi_\sigma$ satisfies curious Lefschetz property with middle weight $2d$. Then $\omega$ satisfies curious Lefschetz property on $X$ with middle weight $2d$.
\end{prop}
\begin{proof}
It is enough to prove the statement in the situation of a space $X$ and a closed subspace $Z$ with complement $U$. Then we have a commutative diagram with exact rows
\[
\begin{tikzcd}
H^{k-1}_c(Z) \arrow{r} \arrow{d}& H^k_c(U) \arrow{r}\arrow{d} & H^k_c(X) \arrow{r}\arrow{d} &H^k_c(Z) \arrow{r}\arrow{d} &H^{k+1}_c(U)\arrow{d}\\
H^{k-1+2i}_c(Z) \arrow{r} & H^{k+2i}_c(U) \arrow{r} & H^{k+2i}_c(X) \arrow{r} &H^{k+2i}_c(Z) \arrow{r} &H^{k+1+2i}_c(U)\\
\end{tikzcd}
\]
where the vertical arrows are given by $(\cup\omega)^i$. The rows are formed by morphisms of MHSs. If we apply the functor $W_{2d-2i}/W_{2d-2i-1}$ to the top row and the functor $W_{2d+2i}/W_{2d+2i-1}$ to the bottom row, the diagram will remain commutative and by Proposition \ref{prop:MHSs} the rows will remain exact. Because the two left and the two right vertical arrows are isomorphisms, the middle vertical arrow also is.
\end{proof}

Similarly to the property of being split Tate, it is not true that if we have split curious Lefschetz for each cell of $X$ then we have it for $X$. We have the following:

\begin{prop}\label{prop:tate and lefschetz}
Suppose $X$ is Tate and $\omega\in H^2(X)$ satisfies curious Lefschetz property with middle weight $2d$.
\begin{enumerate}
\item If $\omega \in F^2 H^2(X)$, then $\omega$ satisfies split curious Lefschetz property on $X$. In particular, the monodromic filtration for $\omega$ coincides with the Hodge filtration $F^\bullet$ (up to the shift of indexing by $d$) and splits the weight filtration.
\item If $\omega \in F^2 H^2(X)\cap \ol F^2 H^2(X)$, then $X$ is split Tate.
\end{enumerate}
\end{prop}

An interesting consequence of this claim is that we can calculate the Hodge numbers of $X$ from the dimensions of kernels of powers of $\cup\omega$:
\begin{prop}
Suppose $\omega\in F^2 H^2(X)$ satisfies the curious Lefschetz property and $X$ is Tate. For each $i>0$ let
\[
g_i(u)=\sum_{k=0}^{2\dim X} u^k \dim\kernel \left((\cup\omega)^i:H^k_c(X)\to H^{k+2i}_c(X)\right)
\]
and let $g_i(u)=0$ for $i\leq 0$. For each $i\in\BZ$ let
\[
f_i(u)=(u^2+1)g_{i+1}(u)-g_i(u)-u^2 g_{i+2}(u),
\]
which is zero for all but finitely many values of $i$.
Then the mixed Hodge polynomial of $X$ is given by
\[
\MH_X(u,v)=\sum_{k,i} u^k v^i \dim H_c^k(X)^{i,i} = \sum_{i=0}^\infty v^{d-i} f_i(u) \frac{(uv)^{i+1}-1}{uv-1}.
\]
\end{prop}
\begin{proof}
Let $\BC^*$ with coordinate $u$ act on $H_c^\bullet(X)$ so that the action on $H_c^i(X)$ is given by multiplication by $u^i$. Recall the Deligne decomposition $Q^{\bullet,\bullet}$ from Proposition \ref{prop:deligne splitting}, which is preserved by $u$. Denote 
\[
f_i(u) = \trace u|_{Q^{i,0}}\quad (i\geq 0),\quad\text{so that}\quad
\trace u|_{Q^{i,j}} = u^{2j} f_i(u).
\]
Then we have 
\[
\MH_X(u,v) = \sum_{0\leq j\leq i} v^{d-i+2j} u^{2j} f_i(u) = \sum_{i=0}^\infty v^{d-i} f_i(u) \frac{(uv)^{i+1}-1}{uv-1}.
\]
On the other hand, we have
\[
g_i(u) = \trace u|_{\kernel (\cup \omega)^i} = \sum_{\substack{0\leq j' \leq i'\\j'+i>i'}} u^{2j'} f_{i'}(u)\quad  (i\geq 0),
\]
from which we extract for all $i\geq 1$
\[
g_i(u)-g_{i-1}(u) = \sum_{\substack{0\leq j' \leq i'\\j'+i-1=i'}} u^{2j'} f_{i'}(u) = f_{i-1}(u) + u^2 f_i(u) + u^4 f_{i+1}(u)+\cdots
\]
and
\[
f_i(u) = (1+u^2) g_{i+1}(u)-g_{i}(u) - u^2 g_{i+2}(u)\quad(i\geq 0).
\]
\end{proof}

Simplest varieties satisfying the curious Lefschetz property are symplectic tori:
\begin{prop}\label{prop:torus lefschez}
	Let $X=(\BC^*)^n$ be a torus and let $\omega$ be a log-canonical non-degenerate $2$-form, i.e. a form written in coordinates $x_1,\ldots,x_n$ as 
	\[
	\sum_{i<j} \omega_{ij} d\log x_i\wedge d \log x_j\qquad(\omega_{ij}\in\BC^*).
	\]
	Then $\omega$ satisfies split curious Lefschetz property on $X$ with middle weight $n$.
\end{prop}
\begin{proof}
	This should be well-known, so we only sketch the proof.
	The cohomology of $X$ is the exterior algebra over the character lattice $\hat X$. The cohomology degrees range from $0$ to $n$, and the weights are all even and go from $0$ to $2 n$, so that the weight of $H^i$ is $2 i$. So, in order to establish the Lefschetz property on cohomology, it is sufficient to show that 
	\[
	(\cup \omega)^k: \Lambda^{n/2-k} \hat X \to \Lambda^{n/2+k} \hat X
	\]
	is an isomorphism for all $k$. In some coordinates on $\hat X\otimes \BC$, the form can be written as
	\[
	\omega = \sum_{i=1}^{n/2} e_{i}\wedge f_{i}.
	\]
	Let us view $e_i$, $f_i$ as multiplication operators, and let us define the dual operators $e^i$, $f^i$ on monomials by by 
	\[
	e^i(e_i\wedge x)=x,\quad e^i (x)=0 \quad\text{if $x$ does not contain $e_i$},
	\]
	and similarly $f^i$. Then we have identities
	\[
	e^i e_i + e_i e^i = 1,\quad f^i f_i + f_i f^i = 1.
	\]
	Define
	\[
	\omega^* = \sum_{i=1}^{n/2} e^{i} f^{i}.
	\]
	A direct computation gives
	\[
	[\omega,\omega^*] = -n/2 + d,
	\]
	where $d$ is the degree operator. The triple $(\omega,\omega^*,-n/2+d)$ forms a representation of the Lie algebra $\mathfrak{sl}_2$. Decomposing the cohomology into irreducible representations of $\mathfrak{sl}_2$ makes the claim obvious. By duality, we obtain the corresponding statement for the cohomology with compact supports.
\end{proof}		

\section{Differential forms on groups}
\subsection{Notations}
We fix $n$ and let $G=\GL_n$, $B\subset G$ the subgroup of upper-triangular matrices, and $U\subset B$ the subgroup of unipotent upper-triangular matrices, $T\subset G$ is the torus of diagonal matrices. The Weyl group is the permutation group $S_n$.

\subsection{Bruhat decomposition}
For a permutation $\pi$, we let 
\[
\inv(\pi) = \{(i,j):i<j\;\&\;\pi(i)>\pi(j)\},\quad \noinv(\pi) = \{(i,j):i<j\;\&\;\pi(i)<\pi(j)\},
\]
then define the subgroups
\[
U_\pi^+ = \Id + \sum_{(i,j)\in\noinv(\pi)} \BC e_{i,j}, \quad U_\pi^- = \Id + \sum_{(i,j)\in\inv(\pi)} \BC e_{i,j}.
\]
We have decompositions\footnote{The reader can easily reformulate all this in terms of roots} 
\[
U = U_\pi^+ U_\pi^- = U_\pi^- U_\pi^+.
\]
The Bruhat decomposition is
\[
G = \bigsqcup_{\pi\in S_n} B \pi B,
\]
it corresponds to a continuous function $G\to S_n$, where $S_n$ is viewed as a poset with respect to the Bruhat order. Each cell $B \pi B$ as an algebraic variety looks like $T\times \BC^{l(\pi)+\binom{n}2}$, and can be explicitly parametrized using the decompositions
\begin{equation}\label{eq:bruhat}
B \pi  B = U_{\pi^{-1}}^- \pi T U = U \pi T U_\pi^-.
\end{equation}

\subsection{$2$-forms}\label{ssec:2 forms}
 Let $X$ be an algebraic variety. Given two morphisms $f,g:X\to G$, we define a $2$-form on $X$ by
\[
(f|g):=\trace\left(f^{-1} df \wedge dg g^{-1}\right).
\]
Explicitly, in coordinates we have
\[
(f|g) = \sum_{i,j,k,l} (f^{-1})_{i j} (g^{-1})_{l i} d f_{jk} \wedge d g_{kl}.
\]
The main property of this construction is the $2$-cocycle property:
\begin{equation}\label{eq:2cocycle}
(f|gh) + (g|h) = (fg|h) + (f|g),
\end{equation}
which can be verified by equating both sides to 
\[
(f|g) + (g|h) + \trace\left(f^{-1} df g \wedge dh h^{-1} g^{-1}\right)
\]
using the matrix Leibnitz rule $d(fg) = df g + f dg$ and the cyclic property of the trace. So we can define for any sequence $f_1,\ldots,f_k$ of maps $X\to \GL_n(\BC)$ a $2$-form
\[
(f_1|f_2|\cdots|f_k):=(f_1|f_2) + (f_1 f_2 | f_3) + \cdots + (f_1 f_2 \cdots f_{k-1}|f_k),
\]
so that it has the property
\begin{equation}\label{eq:2cocycle2}
(f_1|f_2|\cdots|f_k) = (f_i|f_{i+1}) + (f_1|\cdots|f_{i-1}|f_i f_{i+1}|f_{i+2}|\cdots |f_k).
\end{equation}
It is straightforward to verify the following:
\begin{prop}\label{prop:form vanishes}
The form $(f|g)$ vanishes in all of the following cases:
\begin{enumerate}
	\item if $f$ or $g$ is constant;
	\item if $g=f^{-1}$;
	\item if $f\in U$, $g\in B$ or vice versa.
\end{enumerate}	
Furthermore, we have $(f_1|f_2) = -(f_2^{-1}|f_1^{-1})$ for any two functions $f_1$, $f_2$.
\end{prop}
Our main tool is to use \eqref{eq:2cocycle2} in situations when $(f_i|f_{i+1})$ vanishes by Proposition \ref{prop:form vanishes}. In such a sitation we can remove $|$ between $f_i$ and $f_{i+1}$ or insert it.
To illustrate this, we prove
\begin{prop}\label{prop:bruhat form}
	Let $f:X\to G$ be a morphism whose image is contained in the Bruhat cell $B\pi B$ for some $\pi\in S_n$. Choose a factorization of $f$ as $u_1 \pi t u_2$, where $u_1, u_2:X\to U$ and $t:X\to T$ are algebraic morphisms and consider the $2$-form
	\[
	(u_1 |\pi t u_2).
	\]
	This form does not depend on the choice of the factorization.
\end{prop}
\begin{proof}
There is a unique factorization satisfying $u_1(X)\subset U^-_{\pi^{-1}}$. Suppose $u_1',t',u_2'$ gives another factorization. Using the isomorphism $U=U_{\pi^{-1}}^- U_{\pi^{-1}}^+$, write $u_1'=u_1'' v$. So we have
\[
f = u_1'' v \pi t' u_2' = u_1'' \pi t' v' u_2' \qquad \text{for $v':X\to U$,}
\] 
which implies $u_1''=u_1$, $t=t'$, $v' u_2'= u_2$. So we have
\[
(u_1' | \pi t u_2') = (u_1 v | \pi t u_2') = (u_1| v | \pi| t | u_2'),
\]
where we insert $|$ between $u_1$ and $v$ because they are in $U$, then between $\pi$ and $t u_2'$ because $\pi$ is constant, and then between $t$ and $u_2'$ because $t$ is in $B$ and $u_2'$ is in $U$. Now we can remove $|$ between $v$ and $\pi$, move $v$ past $\pi$ and reinsert $|$ to obtain
\[
(u_1 | \pi | {}^\pi v | t | u_2').
\]
Now we have ${}^\pi v:X\to U_\pi^+$, $t:X\to T$, so we can continue to move ${}^\pi v$ past $t$ and obtain
\[
(u_1 | \pi | t | v' | u_2') = (u_1 | \pi  t  v'  u_2') = (u_1 | \pi t u_2).
\]
\end{proof}

\subsection{Differential}\label{ssec:d omega}
It will be useful to have an explicit formula for the differential
\[
d (f|g) = -\trace(f^{-1} df\wedge f^{-1} df\wedge dg\, g^{-1}) - \trace(f^{-1} df\wedge dg\, g^{-1}\wedge dg\, g^{-1}).
\]
On the other hand, we have
\[
\trace\left(\left((fg)^{-1} d(fg)\right)^{\wedge 3}\right) = \trace\left(\left(g^{-1}\, dg\right)^{\wedge 3}\right) + \trace\left(\left(f^{-1}\, df\right)^{\wedge 3}\right)
\]
\[
+3\trace(f^{-1} df\wedge f^{-1} df\wedge dg\, g^{-1}) + 3\trace(f^{-1} df\wedge dg\, g^{-1}\wedge dg\, g^{-1}).
\]
Introduce the notation
\[
\{f\} = \trace\left(\left(f^{-1}\, df\right)^{\wedge 3}\right).
\]
Then we have
\[
3 d (f|g) = \{g\} + \{f\} - \{fg\},\qquad 3 d(f_1|\cdots|f_k) = \sum_{i=1}^k \{f_i\} - \{f_1 f_2\cdots f_k\}.
\]
Symbol $\{f\}$ satisfies
\[
\{f\} = -\{f^{-1}\}.
\]

\section{Symplectic form on the character variety}\label{sec:symplectic form}
In this section we write down an explicit formula for the symplectic form $\omega$ on the parabolic character variety following \cite{guruprasad1997group}. Then we show that this symplectic form can be obtained by a version of higher transgression from the characteristic class $c^2$ and conclude that $\omega\in F^2$ for the Hodge filtration. The main idea came from an attempt to make the corresponding result of Shende \cite{shende2016weights} in the non-parabolic case more explicit.

\subsection{Complexes of varieties and higher transgressions}
In \cite{deligne1974theorie} Deligne constructs mixed Hodge structure on the cohomology of a simplicial variety. We need a mild generalization of this construction to the case of a complex of varieties. This is probably well known to experts. 

\begin{defn}
Let $\Var$ be the category of algebraic varieties over $\BC$. Let $\QVar$ be the category whose objects are objects of $\Var$ and whose morphisms are defined as follows. For any $X,Y\in\Var$ let $\Hom_{\QVar}$ be the free $\BQ$-vector space on the set of triples $(X', Y', f)$ where $X'$ is a connected component of $X$, $Y'$ is a connected component of $Y$ and $f:X'\to Y'$ is a morphism in $\Var$. The composition is defined on basis elements by
\[
(Y'', Z', g) \circ (X', Y', f) = \begin{cases}
(X', Z', g\circ f) & \text{if $Y''=Y'$,}\\
0 & \text{otherwise}
\end{cases}.
\]
Then $\QVar$ is an additive category where the direct sum is given by the disjoint union.
\end{defn}

In fact $\QVar$ is the additive envelope of the category of connected algebraic varieties with morphisms formal linear combinations of morphisms.

Then we can consider homologically indexed complexes of varieties $(X_\bullet,d_\bullet)$, unbounded in the positive direction that look like
\[
\cdots \xrightarrow{d_3} X_2 \xrightarrow{d_2} X_1 \xrightarrow{d_1} X_0,
\]
we have $d_i\in\Hom_{\QVar}(X_i, X_{i-1})$ such that $d_i d_{i+1}=0$.
The category of such complexes will be denoted by $\CVar$. Since $\QVar$ is additive, the category $\CVar$ is a \emph{strongly pre-triangulated dg-category}, which means that there are functorial constructions of shifts and mapping cones. The category $\CVar$ is more useful than the category of simplicial varieties because any simplicial variety produces a complex in a natural way, but there are more morphisms between complexes than there are morphisms between simplicial varieties. 

The definition of cohomology of a complex and the construction of functorial mixed Hodge structure on the cohomology is done exactly like in the simplicial case. We recall the main steps here. Having a complex of smooth varieties $(X_\bullet,d_\bullet)$ as above we go over $i=0,1,2,\ldots$ and construct compactifications $X_i\subset \tilde X_i$ such that $\tilde X_i\setminus X_i$ is a simple normal crossing divisor in such a way that all the $\Var$-components of $d_i$ extend to morphisms $\tilde X_i\to \tilde X_{i-1}$. Since $d_i$ is a finite linear combination of morphisms in $\Var$, this is possible. Thus we obtain a complex of pairs
\begin{equation}\label{eq:complex of pairs}
\cdots \xrightarrow{\tilde d_3} (X_2,\tilde X_2) \xrightarrow{\tilde d_2} (X_1, \tilde X_1) \xrightarrow{\tilde d_1} (X_0,\tilde X_0),
\end{equation}
where each $\tilde d_i$ is a linear combination of morphisms of pairs $(X_i,\tilde X_i)\to (X_{i-1},\tilde X_{i-1})$. Using the construction of logarithmic de Rham complex, composing it with any functorial way to compute $R\Gamma$, e.g. using the Godement resolution, we obtain a contravariant functor from pairs $(X,\tilde X)$ to mixed Hodge complexes $(X,\tilde X)\to \MHC(X,\tilde X)$. We do not mention the definition of mixed Hodge complexes, which the reader can find in \cite{deligne1974theorie} or \cite{peters2008mixed}. Important point for us is that there is a functorial construction of mapping cones in the category of mixed Hodge complexes (see Section 3.4 in \cite{peters2008mixed}). A reader interested in details is welcome to unravel the definition and see that our constructions make sense, but to get the general idea one can pretend that we are working with ordinary complexes in some additive category. We apply the functor $\MHC$ to the complex \eqref{eq:complex of pairs} and obtain a double complex 
\[
\cdots \xleftarrow{\tilde d_3^*} \MHC(X_2,\tilde X_2) \xleftarrow{\tilde d_2^*} \MHC(X_1, \tilde X_1) \xleftarrow{\tilde d_1^*} \MHC(X_0,\tilde X_0),
\]
To this double complex we associate the total complex $\Tot^\bullet \MHC(X_\bullet, \tilde X_\bullet)$ by iterating the cone construction. More explicitly, let us start by setting $C_0=\MHC(X_0,\tilde X_0)$. Then  we set $C_1=\cone(d_1^*)[-1]$, which fits into a short exact sequence of mixed Hodge complexes
\[
0\to \MHC(X_1,\tilde X_1)[-1] \to C_1 \to C_0\to 0.
\]
The condition $d_2^* d_1^*=0$ implies that the map $d_2^*$ extends to a map
$C_1\to \MHC(X_2,\tilde X_2)[-1]$. So we take the cone of this map and shift it by $[-1]$, and so on. Proceeding in this way we obtain a sequence of mixed Hodge complexes together with maps
\[
\cdots C_2\to C_1 \to C_0,
\]
which stabilizes in each degree after finitely many iterations. So the limit is a well-defined mixed Hodge complex which we denote by $\Tot^\bullet \MHC(X_\bullet, \tilde X_\bullet)$. The cohomology of a complex of varieties (which can be also computed using algebraic de Rham complexes or complexes of singular chains as the total complex of the double complex) is identified with the cohomology of this mixed Hodge complex, so it naturally carries a mixed Hodge structure which does not depend on the choice of compactifications,
\[
H^\bullet(X_\bullet)=H^\bullet(\Tot^\bullet\MHC(X_\bullet,\tilde X_\bullet)).
\]
For any morphism $X_\bullet\to Y_\bullet$ in $\CVar$ given by a commutative diagram of varieties and formal linear combinations of maps
\[
\begin{tikzcd}
\cdots \arrow{r} & X_2 \arrow{r}\arrow{d} & X_1  \arrow{r}\arrow{d} & X_0\arrow{d}\\
\cdots \arrow{r} & Y_2 \arrow{r} & Y_1  \arrow{r} & Y_0
\end{tikzcd}
\]
we can choose compactifications $Y_i\subset \tilde Y_i$ and then compactifications $X_i\subset \tilde X_i$ so that we obtain a diagram of pairs. Then we apply $\MHC$ and $\Tot^\bullet$ to obtain morphisms of mixed Hodge structures
\[
H^i(Y_\bullet)\to H^i(X_\bullet).
\] 
The category $\CVar$ has a shift functor $(X_\bullet,d_\bullet) \to (X_\bullet[k],d_\bullet[k])$ for each $k\in\BZ_{\geq 0}$ defined by
\[
X_i[k] = \begin{cases}
X_{i-k}&(i\geq k)\\
\varnothing&(i<k)
\end{cases},\qquad d_i[k] = \begin{cases}
d_{i-k}&(i> k)\\
0&(i\leq k)
\end{cases},
\]
and we have 
\[
H^i(X_\bullet[k]) \cong \begin{cases}
H^{i-k}(X_\bullet)&(i\geq k)\\
0&(i<k)
\end{cases}.
\]
Thus a morphism of type $f:X_\bullet[k]\to Y_\bullet$ produces a map in cohomology which preserves mixed Hodge structures, but shifts cohomological degrees:
\[
f^*: H^{i+k}(Y_\bullet) \to H^{i}(X_\bullet).
\]
We call such maps \emph{higher transgressions}.

Any usual variety $X$ can be considered as a complex of varieties where we put $X$ in degree $0$ and empty sets in other degrees with zero differentials.

\subsection{Characteristic classes}
Suppose an algebraic group $G$ acts on an algebraic variety $X$. Then the bar construction produces a complex of varieties as follows. For $k=0,1,2,\ldots$ let 
\begin{equation}\label{eq:presentation1}
G^k X = G^{\times k} \times X.
\end{equation}
It is sometimes more convenient to write
\begin{equation}\label{eq:presentation2}
G^k X = G\backslash (G^{k+1} \times X).
\end{equation}
We will denote points on $G^k X$ by $(g_1,\ldots,g_k,x)$ when we use presentation \eqref{eq:presentation1} and by $[g_0,\ldots,g_k,x]$ when we use presentation \eqref{eq:presentation2}. The identification of the two notations is given by
\[
[g_0, g_1, \ldots, g_k, x]=(g_0^{-1} g_1, g_1^{-1} g_2, \ldots, g_{k-1}^{-1} g_k, g_k^{-1} x)
\]
and the inverse map is given by
\[
(g_1, \ldots, g_k, x) = [1, g_1, g_1 g_2, \ldots, g_1 g_2\cdots g_k, g_1 g_2\cdots g_k x].
\]

The morphisms $d_k:G^k X \to G^{k-1} X$ for $k\geq 1$ are defined by
\[
d_k = \sum_{i=0}^k (-1)^i d_{k,i}, \qquad d_{k,i}[g_0,\ldots,g_k,x]=[g_0,\ldots,\widehat g_i,\ldots,g_k,x]
\]
using the presentation \eqref{eq:presentation2}. Using \eqref{eq:presentation1} we have:
\begin{equation}\label{eq:dk long}
d_{k,i}(g_1,\ldots,g_k,x) = \begin{cases} (g_2, g_3,\ldots,g_k,x) & (i=0)\\
(g_1, \ldots, g_{i} g_{i+1},\ldots,g_k,x) & (0<i<k)\\
(g_1,\ldots,g_{k-1},g_k x) & (i=k)
\end{cases}.
\end{equation}
For instance, in small degrees we have 
\[
d_1(g_1, x) = (x)-(g_1 x), \quad d_2(g_1, g_2, x) = (g_2, x) - (g_1 g_2,x)+(g_1, g_2 x).
\]
This defines a complex of varieties $G^\bullet X$, which will be called the \emph{quotient stack}. The reader can recognize that our definition is simply the complex obtained from the usual simplicial construction of the quotient stack. 

Note that if $G$ is reductive, the categorical quotient $X/G$ is a well-defined algebraic variety and we have a natural map $G^\bullet X\to X/G$ called the \emph{augmentation map}. If $X$ is a principal $G$-variety, the augmentation map induces an isomorphism on cohomology (see \cite{deligne1974theorie}). 

In the case when $X$ is a point we obtain the simplicial realization of the \emph{classifying space} $*/G$. Suppose $G=GL_n$. Then the cohomology of $*/G$ is the polynomial ring in characteristic classes $c_1,c_2,\ldots,c_n$ with $c_i\in H^{2i}(*/G)^{i,i}$. There is a map $G[1]\to */G$:
\[
\begin{tikzcd}
G= & & & G\arrow{d}{\Id} & \\
*/G= & \cdots \arrow{r} & G^{\times 2} \arrow{r} & G \arrow{r}{0} & \text{point}
\end{tikzcd}
\]
which induces transgression maps $\tr:H^{i}(*/G) \to H^{i-1}(G)$ preserving mixed Hodge structures. It is well-known that the cohomology ring of $GL_n$ is the exterior algebra on the images of the characteristic classes $\tr(c_i)$.

\subsection{Character variety, character stack, and maps to $*/G$}\label{ssec:charvar}
We will be dealing with the following kinds of objects. Let $\Sigma$ be a triangulated topological space. The main example is a surface of genus $g$ with $k$ boundary components. Let $G$ be an algebraic group. The main example is $GL_n$. The space of homomorphisms from the fundamental groupoid of the triangulation to $G$ will be denoted $\CX$ and called the representation variety. Geometrically,  points of $\CX$ correspond to local systems on $\Sigma$ equipped with a choice of basis over each vertex of the triangulation. Then for each vertex the group $G$ acts on $\CX$ by change of basis. The product of these groups $G$ will be called the gauge group. The quotient of $\CX$ by the action of the gauge group will be called the character stack. Details follow. 

Let $T_i^\nondeg$ denote the set of simplices of the triangulation of $\Sigma$ of dimension $i$. The elements of $T_i^\nondeg$ will be called the \emph{non-degenerate simplices}. Denote by $[m]$ the set $\{0,1,\ldots,m\}$. We assume for each simplex $\tau\in T_i^\nondeg$ a choice of numbering of the vertices of $\tau$ by numbers from $[i]$. 

Then \emph{all simplices} of $\Sigma$ are defined as follows. The set of $m$-simplices in $\Sigma$, denoted $T_i$ is the set of pairs $(\tau, f)$ where $\tau\in T_i^\nondeg$ for some $i\leq m$ and $f:[m]\to [i]$ is a surjective map. One should think of $f$ as a map from the standard $m$-simplex $\Delta^m$ to $\Sigma$ whose image is $\tau\cong\Delta^i$, and the induced map $\Delta^m\to\Delta^i$ is an affine map sending the vertices of $\Delta^m$ to the vertices of $\Delta^i$ as prescribed by $f$. In the usual terminology, $T_\bullet$ is a simplicial set and $\Sigma$ is its topological realization.

The \emph{representation variety} $\CX$ is defined as the subvariety of $G^{T_1^\nondeg}$ of assignments $e\in T_1^\nondeg \to g_e\in G$ satisfying the following condition. Note that any $e\in T_1$ is either mapped to a vertex, or to an edge of the triangulation. If it is mapped to a vertex, we set $g_e=1\in G$. If it is mapped to an edge $e'$, we set $g_e=g_{e'}$ if the orientations of $e$ and $e'$ coincide, and $g_e=g_{e'}^{-1}$ otherwise. The condition is that for each triangle $\tau\in T_2^\nondeg$ whose boundary with the induced orientation consists of $e_1, e_2, e_3$ we have
\[
g_{e_1} g_{e_2} g_{e_3} = 1 \qquad \text{for each}\quad\begin{tikzpicture}[scale=2,baseline=15pt] \draw[->] (0,0)--(1,0) node[midway,below] {$e_1$};\draw[->] (1,0) -- (0.5,0.71) node[midway,right] {$e_2$}; \draw[->] (0.5,0.71) -- (0,0)node[midway,left] {$e_3$};\draw (0.5,0.3) node {$\tau$};\end{tikzpicture}\quad\in T_2^\nondeg.
\]

Note that $\CX$ is naturally a closed subvariety of $G^{T_1}$. Thus $\CX$ comes with natural maps $g_e:\CX\to G$ for $e\in T_1$. The matrix function $g_e$ will be called the \emph{transfer matrix} along $e$.  Each element $\tau\in T_m$ defines a map $g_\tau:\CX\to G^m$ as follows. Let $e_i\in T_1$ for $i=1,2,\ldots,m$ denote the edge of $\tau$ going from the vertex $i-1$ to the vertex $i$. Then we define 
\[
g_\Delta=(g_{e_1}, \ldots, g_{e_m}).
\]
It is clear that the system of maps $(g_\tau)_{\tau\in T_\bullet}$, $g_\tau:\CX\to G^{\dim \tau}$ is compatible with the natural face maps in the following sense. For each $\tau\in T_m$ and each $i\in[m]$, let $\tau_i\in T_{m-1}$ denote the $m-1$-simplex obtained from $\tau$ by removing the vertex $i$. Then we have a commutative diagram
\[
\begin{tikzcd}
\CX \arrow{r}{g_\tau} \arrow{dr}{g_{\tau_i}}& G^m\arrow{d}{d_{m,i}}\\
 & G^{m-1}
\end{tikzcd}
\]

Geometrically, the representation variety of a single $k$-simplex is $G^k$, and $g_\tau$ is the restriction map. Using the presentation $G^k=G\backslash G^{k+1}$, we can view this character variety more symmetrically. Put a base point in the center of the standard $k$-simplex. Then the transfer maps from the vertices to the center are encoded by a $k+1$-tuple $(g_0, \ldots, g_k)\in G^{k+1}$. Changing the basis in the center amounts to the diagonal multiplication on the left by elements of $G$.

We have $T_0=T_0^\nondeg$. The \emph{gauge group} $G^{T_0}$ acts on $\CX$ by framing transformations as follows. If $h\in G^{T_0}$ with $h_v\in G$ for all $v\in T_0$, then for each edge $e:v_1\to v_2$ we send $g_e$ to $h_{v_1} g_e h_{v_2}^{-1}$. In the presentation \eqref{eq:presentation2} the components of $h$ at the vertices $h_0, \ldots, h_k$ act as follows:
\[
[g_0, \ldots, g_k] \to [g_0 h_0^{-1}, \ldots, g_k h_k^{-1}].
\]

The \emph{character stack} is defined as the quotient stack of the representation variety by the gauge group $\CX/G^{T_0}=(G^{T_0})^\bullet \CX$. The restriction maps $g_\tau$ extend to the quotient stack as follows:
\begin{prop}\label{prop:stack natural maps}
There exists a collection of $\QVar$-morphisms $g_{\tau,r}: (G^{T_0})^r \CX \to G^{m+r}$ where $\tau\in T_m$ and $r\in\BZ_{\geq 0}$ such that
\begin{enumerate}
\item For $r=0$ we have $g_{\tau,0}=g_\tau:\CX\to G^m$.
\item For any $\tau\in T_m$ and $r\in\BZ_{\geq 0}$ we have
\begin{equation}\label{eq:stack morphisms boundary}
d_{m+r} g_{\tau,r} = g_{\tau,r-1} d_r + (-1)^r g_{\partial\tau,r},
\end{equation}
where $\partial\tau$ is the boundary of $\tau$, i.e. the formal linear combination $\sum_{i=0}^m (-1)^i \tau_i$ where $\tau_i$ is obtained from $\tau$ by removing the vertex $i$, and $g$ is extended to formal linear combinations of simplices by linearity.
\end{enumerate}
\end{prop}
\begin{proof}
A point of $\CX$ associates to $\tau$ an element of $G^m$ represented by a sequence $(g_0,\ldots, g_m)$, well-defined up to the left multiplication by $G$. A point in $(G^{T_0})^r\CX$ is represented by a point in $\CX$ and an $r+1$-tuple of elements of $G$ at each vertex, up to an action of $G^{T_0}$. So for each vertex $i$ of $\tau$ we have an $r+1$-tuple $(h_{0,i},\ldots,h_{r,i})$. This defines a framed local system on the product $\Delta^r\times \Delta^m$ by putting $g_i h_{i,j}$ in the vertex $(j,i)\in [r]\times [m]$. 

There is a standard way to triangulate the product $\Delta^r\times \Delta^m$ into $\binom{m+r}{m}$ simplices of dimension $m+r$. Simplices of this triangulation are given by lattice paths from $(0,0)$ to $(r,m)$ consisting of $(1,0)$ and $(0,1)$ steps. Denote such a path by $p$ and let $p_{i}\in [r]\times [m]$, be the position after $i$ steps, $0\leq i\leq r+m$. Let $|p|$ be the number of pairs $1\leq j<\leq j'\leq r+m$ such that the $j$-th step is $(0,1)$ and $j'$ step is $(1,0)$. The simplex corresponding to $p$ is denoted by $\Delta_p$. It has vertices $p_{0}, p_1,\ldots, p_{r+m}$, and it is taken with sign $(-1)^{|p|}$. So we write
\[
\Delta^r\times \Delta^m = \sum_p (-1)^{|p|} \Delta_p.
\]
Then we have
\begin{equation}\label{eq:product of simplices}
\partial(\Delta^r\times \Delta^m) = \partial(\Delta^r) \times \Delta^m + (-1)^r \Delta^r \times \partial(\Delta^m).
\end{equation}
Indeed, on the left hand side we go over all paths and remove one vertex from each path in all possible ways. So we have a sum over all paths from $(0,0)$ to $(r,m)$ where one step is $(2,0)$, $(1,1)$ or $(0,2)$ and all the other steps are as before. The paths with $(1,1)$ cancel out because each such path is obtained in two ways with opposite signs. The paths with $(2,0)$ step bijectively correspond to the summands of $\partial(\Delta^r) \times \Delta^m$ and the paths with $(0,2)$ step correspond to the summands of $(-1)^r \Delta^r \times \partial(\Delta^m)$.

Taking the formal sum over the simplices of $\Delta^r\times\Delta^m$ with signs we obtain a morphism in $\QVar$
\[
(G^{\times m+1})^{\times r+1} \times G^{\times m+1} \to G^{\times m+r+1},
\]
which is $G$-equivariant with respect to the left $G$-action on $G^{\times m+1}$ and $G^{\times m+r+1}$, and invariant under the action of $G^{\times m+1}$ on $(G^{\times m+1})^{\times r+1} \times G^{\times m+1}$. Thus we obtain a well-defined morphism in $\QVar$
\[
g_{\tau,r}:(G^{T_0})^{r} \CX \to G^{m+r}.
\]
Using \eqref{eq:product of simplices} we obtain \eqref{eq:stack morphisms boundary}.
\end{proof}

Now we can construct morphisms $\CX/G^{T_0}[m]\to */G$ by producing linear combinations of $m$-simplices of $\Sigma$ so that the contributions from  $g_{\partial\Delta,r}$ are annihilated. In particular, we obtain
\begin{cor}
	For any $m$, elements of the simplicial homology group $H_m(\Sigma, \BQ)$ give well-defined up to homotopy maps $\CX/G^{T_0}[m]\to */G$.
\end{cor}
These maps induce higher transgression maps on the cohomology which are compatible with the mixed Hodge structure. Thus we recover results of \cite{shende2016weights}. 

\subsection{Parabolic character stack}\label{ssec:parabolic char stack}
Let $\Sigma$ be a Riemann surface of genus $g$ with $k$ boundary components. Let $p_0\in\Sigma$ be a point in the interior, and let $p_i$ be a point on $i$-th boundary component. We can assume $\Sigma$ is obtained by gluing sides of a polygon whose vertices are mapped to $p_0$, $p_1$, \ldots, $p_k$. Choosing a triangulation of the polygon we obtain a triangulation of $\Sigma$. We assume the edges of the polygon are given by 
\begin{itemize}
\item The boundary edges $q_i:p_i\to p_i$ for $i=1,\ldots,k$.
\item The edges $\gamma_i:p_0\to p_i$ for $i=1,\ldots,k$.
\item Loops $\alpha_i,\beta_i:p_0\to p_0$ for $i=1,\ldots,g$.
\end{itemize}
Moreover, we assume that walking along the sides of the polygon we obtain the following loop, which becomes contractible in $\Sigma$:
\begin{equation}\label{eq:matrix equation}
[\alpha_1,\beta_1]\cdots [\alpha_g, \beta_g] \gamma_1^{-1} q_1 \gamma_1\cdots\gamma_k^{-1} q_k \gamma_k.
\end{equation}
The $\CX$ can be identified with the variety of $2k+2g$ tuples of matrices satisfying the condition expressed by \eqref{eq:matrix equation}.
Pick matrices $C_1,\ldots,C_k\in G$. Denote by $Z(C_i)\subset G$ the corresponding centralizer subgroups. Let the \emph{parabolic gauge group} be
\[
G_\parab=G\times Z(C_1) \times\cdots\times Z(C_k)\subset G^{k+1}.
\]
The \emph{parabolic representation variety} $\CX_{\parab}\subset \CX$ is the subvariety of collections 
\[
\CX_{\parab}=\{(g_e)_{e\in T_1^\nondeg}\in\CX:\;g_{q_i}=C_i\;\text{for all $i$}\}.
\]
 Then $G_\parab$ preserves $\CX_\parab$ and we have for each $i$ a closed embedding
\[
G_\parab^i \CX_\parab\subset (G^{k+1})^i\CX.
\]
In particular, constructions of Section \ref{ssec:charvar} extend to the \emph{parabolic character stack} defined by $\CX_\parab/G_\parab = G_\parab^\bullet \CX_\parab$.

We assume that the orientations of the $2$-simplices of the triangulation coincide with the orientations induced from $\Sigma$, and the orientation of $q_i$ is opposite to the orientation induced from $\partial\Sigma$ for each $i$. Let $[\Sigma]$ be the sum of the $2$-simplices of the triangulation $T_2^\nondeg$. By Proposition \ref{prop:stack natural maps}, we have
\[
\partial \Sigma = -\sum_{i=1}^k q_i, \qquad g_{[\Sigma],r}:G^r_\parab\CX_\parab\to G^{r+2},\quad d_{r+2} g_{[\Sigma],r} = g_{[\Sigma],r-1} d_r - (-1)^r \sum_{i=1}^k g_{q_i,r}.
\]
So the obstruction for $g_{[\Sigma],r}$ to become a morphism of complexes is given by the maps $g_{q_i,r}$. 
Since the element associated to $q_i$ is fixed to be $C_i$, the maps $g_{q_i,r}$ factor as follows:
\[
\begin{tikzcd}
G^r_\parab\CX_\parab \arrow{r}{g_{q_i,r}} \arrow{d} & G^{r+1}\\
Z(C_i)^r \arrow{ru}{\sigma_i} & 
\end{tikzcd}
\]
where the morphism of complexes $\sigma_i: */Z(C_i)[1] \to */G$ is given by
\[
[h_0,\ldots,h_r] \to [h_0,\ldots,h_r,C_i h_r]-[h_0,\ldots,h_{r-1},C_i h_{r-1}, C_i h_r]+\cdots.
\]
We correct $g_{[\Sigma],r}$ to obtain a morphism of complexes as follows. Let 
\[
BG_\parab := \cone\left(\bigoplus_{i=1}^k \sigma_i:\bigoplus_{i=1}^k */Z(C_i)[1] \to */G\right),
\]
or, more explicitly, 
\[
BG_\parab^r= G^r \sqcup Z(C_1)^{r-2}\sqcup\cdots\sqcup Z(C_k)^{r-2},
\]
and the differential is the sum of the usual differentials $G^{r}\to G^{r-1}$, $Z(C_i)^{r-2}\to Z(C_i)^{r-3}$ and the maps $(-1)^r \sigma_i:Z(C_i)^{r-2}\to G^{r-1}$. Then the morphisms
\[
\tilde g_{[\Sigma],r}: G_\parab^r\CX_\parab \to BG_\parab^{r+2}
\]
given by the sum of $g_{[\Sigma],r}$ and the projection maps $G^r_\parab\to Z(C_i)^r$ form a morphism of complexes, which we denote by $\tilde g_{[\Sigma]}$. 

\subsection{Cohomology classes on the stack}
We can construct cohomology classes on the stack $\CX_\parab/G_\parab$ by pulling back from $BG_\parab$. So we need to understand the cohomology of $BG_\parab$. We have a long exact sequence of mixed Hodge structures
\begin{equation}\label{eq:long parabolic sequence}
\cdots\to H^{r-1}(*/G) \to \bigoplus_{i=1}^k H^{r-2}(*/Z(C_i)) \to H^r(BG_\parab)   \to H^{r}(*/G)\to\cdots.
\end{equation}
Suppose from now on that $G=\GL_n$ and the matrices $C_i$ are diagonal, and the entries of the $C_i$ are ordered in such a way that if some eigenvalue repeats, its positions form a contiguous interval. This corresponds to the condition that $Z(C_i)$ is the group of block-diagonal matrices of certain shape depending on the multiplicities of the eigenvalues. The classifying space has only even cohomology. Therefore the long exact sequence splits into short exact sequences. In particular, to obtain the symplectic form we need to understand the following short exact sequence:
\[
0\to \bigoplus_{i=1}^k H^2(*/Z(C_i)) \to H^4(BG_\parab)\to H^{4}(*/G)\to 0.
\]
The Hodge filtration splits the sequence above in the sence that
\[
H^{4}(*/G) = F^2 H^4(BG_\parab),
\]
but it is not clear how to construct this splitting algebraically. So we will produce an explicit element of $H^4(BG_\parab)$ and then show that it belongs to $F^2$.

In the following formulas, the letters $x$, $y$, $z$ stand for varying elements in $G$. In the algebraic de Rham realization, the cohomology $H^4(BG_\parab)$ is computed from the cohomology of the total complex of the following tri-complex:
\[
\begin{tikzcd}
& \Omega^2(G\times G \times G) & &\\
\Omega^1(G\times G)\arrow{r}{d}\arrow{rddd} & \Omega^2(G\times G)\arrow{r}{d}\arrow{u}{(y,z)-(xy,z)+(x,yz)-(x,y)}\arrow{rddd} & \Omega^3(G\times G) &\\
& \Omega^2(G)\arrow{r}{d}\arrow{u}{(xy)-(x)-(y)} & \Omega^3(G) \arrow{u}{(xy)-(x)-(y)}\arrow{r}{d}& \Omega^4(G) \\
&\bigoplus_{i=1}^k \Omega^1(Z(C_i)\times Z(C_i))&\\
& \bigoplus_{i=1}^k \Omega^1(Z(C_i)) \arrow{r}{d} \arrow{u}{(xy)-(x)-(y)} &\bigoplus_{i=1}^k \Omega^2(Z(C_i))
\end{tikzcd}
\]

The classes in $H^2(*/Z(C_i))$ are represented as follows. Denote by $\widehat{Z(C_i)}$ the group of multiplicative characters. There is an isomorphism
\[
\bigoplus_{i=1}^k \widehat{Z(C_i)}\otimes\BC \toiso \bigoplus_{i=1}^k H^2(*/Z(C_i)),
\]
constructed as follows.
For every multiplicative character $f:Z(C_i)\to \BC^*$ we have a differential form $\frac{df}{f}\in\Omega^1(Z(C_i))$, which is closed and satisfies
\[
\frac{df(xy)}{f(xy)} - \frac{df(x)}{f(x)} - \frac{df(y)}{f(y)}=0.
\]
Thus we obtain an element in $\bigoplus_{i=1}^k \Omega^1(Z(C_i))$ which is annihilated by all the outgoing arrows. Classes obtained in this way belong to $F^1 \cap \overline F^1 \cap W_2$.

In the following calculations we often use the fact that for matrix-valued differential forms $\eta_1, \eta_2$ of degrees $|\eta_1|$, $|\eta_2|$ we have
\[
\trace(\eta_1\wedge\eta_2) = (-1)^{|\eta_1|\cdot|\eta_2|}\trace(\eta_2\wedge\eta_1).
\]
Define a class $\omega\in H^4(BG_\parab)$ by $\omega=\omega_2+\frac13 \omega_3$ where
\[
\omega_2 = \trace(x^{-1} dx\wedge dy\, y^{-1})\in\Omega^2(G\times G),\quad \omega_3 = \trace\left((x^{-1} dx)^{\wedge 3}\right)\in\Omega^3(G).
\]
\begin{prop}\label{prop:class in BGpar}
The combination of differential forms $\omega=\omega_2+\frac13 \omega_3$ is closed with respect to the total differential of the tri-complex and therefore defines a class in $H^4(BG_\parab)$.
\end{prop}
\begin{proof}
Using $d(x^{-1} dx) = -x^{-1} dx\wedge x^{-1} dx$, we obtain
\[
d \omega_3 = \trace\left((x^{-1} dx)^{\wedge 4}\right) = 0,
\]
\[
d \omega_2 = -\trace(x^{-1} dx\wedge x^{-1} dx\wedge dy\, y^{-1}) - \trace(x^{-1} dx\wedge dy\, y^{-1}\wedge dy\, y^{-1}).
\]
Using $(xy)^{-1} d(xy)=y^{-1} dy + y^{-1} (x^{-1} dx)\, y$, we obtain
\[
\trace\left(\left((xy)^{-1} d(xy)\right)^{\wedge 3}\right)
= \trace\left((x^{-1} dx)^{\wedge 3}\right) + \trace\left((y^{-1} dy)^{\wedge 3}\right)
\]
\[
+ 3 \trace(x^{-1} dx\wedge x^{-1} dx\wedge dy\, y^{-1}) + 3 \trace(x^{-1} dx\wedge dy\, y^{-1}\wedge dy\, y^{-1}),
\]
\[
\trace (y^{-1} dy\wedge dz\, z^{-1}) + \trace (x^{-1} dx\wedge d(yz)\, (yz)^{-1})
\]
\[
=\trace (y^{-1} dy\wedge dz\, z^{-1}) + \trace (x^{-1} dx\wedge dy\, y^{-1}) + \trace( x^{-1} dx\wedge y dz\,z^{-1} y^{-1})
\]
\[
=\trace (x^{-1} dx\wedge dy\, y^{-1}) + \trace((xy)^{-1} d(xy)\wedge dz\, z^{-1}).
\]
This shows that $\omega=\omega_2+\frac13 \omega_3$ is in the kernel of the total differential of the complex $\Omega^\bullet(G^\bullet)$ and therefore defines a class in $H^4(*/G)$. The map $\Omega^2(G\times G)\to \Omega^2(Z(C_i))$ is given by the pull-back along $x\to (x,C_i)-(C_i,x)$. So $\omega_2$ goes to zero. This shows that $\omega$ also defines a class in $H^4(BG_\parab)$.
\end{proof}

Next we show
\begin{prop}
The class $\omega$ lies in $F^2 H^4(BG_\parab)$.
\end{prop}
\begin{proof}
We use the splitting principle. The maps $\sigma_i$ can be restricted to the subgroups of $Z(C_i)$ and $G$ of upper-triangular matrices, which we denote by $B(C_i)$, $B$ with the corresponding cone denoted by $BB_\parab$. Then we can project to the groups of diagonal matrices which we denote by $T(C_i)$ and $T$ with the cone $BT_\parab$. The corresponding long exact sequences \eqref{eq:long parabolic sequence} then form commutative diagrams, and in particular we obtain a commutative diagram
\[
\begin{tikzcd}
0\arrow{r} & \bigoplus_{i=1}^k H^2(*/Z(C_i)) \arrow{r} \arrow[hookrightarrow]{d} & H^4(BG_\parab)\arrow{r} \arrow[hookrightarrow]{d} & H^{4}(*/G)\arrow{r} \arrow[hookrightarrow]{d} & 0\\
0\arrow{r} & \bigoplus_{i=1}^k H^2(*/B(C_i)) \arrow{r} & H^4(BB_\parab)\arrow{r} & H^{4}(*/B)\arrow{r} & 0\\
0\arrow{r} & \bigoplus_{i=1}^k H^2(*/T(C_i)) \arrow{r} \arrow{u}{\sim} & H^4(BT_\parab) \arrow{r} \arrow{u}{\sim} & H^{4}(*/T)\arrow{r} \arrow{u}{\sim} & 0
\end{tikzcd}
\]
The arrows pointing up are isomorphisms because the corresponding projection maps $B^m\to T^m$ and similarly for $Z(C_i)$ are homotopy equivalences. The arrows pointing down are injections because they are injections for the left arrow and for the right arrow by the splitting principle, and therefore for the middle arrow by the snake lemma.

Consider any torus of the form $T=(\BC^*)^a$. Denote the coordinates by $z_1, \ldots, z_a$. Denote by $\Omega_\mathrm{log}^r(T)$ the vector space of logarithmic $r$-forms, i.e. forms of the form 
\[
\frac{d z_{i_1}}{z_{i_1}}\wedge \cdots \wedge \frac{d z_{i_r}}{z_{i_r}}.
\]
For any compactification $\tilde T$ with simple normal crossings divisor $D=\tilde T\setminus T$, logarithmic forms extend to global sections of $\Omega^r_{\tilde T}(\log D)$ because functions $z_i$ are meromorphic functions on $\tilde T$ with zeroes and poles along $D$. Remember that the Hodge filtration is induced by the ``stupid filtration'' on $\Omega^\bullet_{\tilde T}(\log D)$. So when we have a complex of varieties $(X_\bullet, d_\bullet)$ all of whose connected components are tori, any logarithmic $r$-form $\eta\in\Omega_{\mathrm{log}}^r(X_k)$ satisfying $d_{k+1}^*\eta=0$ produces a class in $F^r H^{k+r}(X_\bullet,d_\bullet)$.

Using these observations let us show that the class corresponding to $\omega$ in $H^{4}(BG_\parab)$ belongs to $F^2$. When we restrict $\omega$ to $BB_\parab$, the component $\omega_3$ becomes zero and the component $\omega_2$ becomes a pull-back of a logarithmic form in $\Omega^2(T\times T)$. Therefore the corresponding class is in $F^2 H^4(BT_\parab)$. Because the map $H^4(BG_\parab)\to H^4(BT_\parab)$ is an injective morphism of mixed Hodge structures, we have that $F^2 H^4(BG_\parab) = F^2 H^4(BT_\parab)\cap H^4(BG_\parab)$, so the class of $\omega$ is in $F^2 H^4(BG_\parab)$.
\end{proof}

\begin{rem}\label{rem:split or does not}
In the case without marked points we have $H^4(BG_\parab)=H^4(*/G)$ and $H^4(*/G)$ is pure of weight $(2,2)$. Therefore $\omega\in \overline F^2 H^4(BG_\parab)$. This is not true in general. For example, consider the case of rank $2$ on $\BP^1$ with four marked points with generic semisimple monodromies. In this case the class of $\omega$ spans $F^4 H^2$. Suppose $\overline F^4 H^2 = F^4 H^2$. Then the line $F^4 H^2\subset H^2$ depends on the eigenvalues both holomorphically and anti-holomorphically, so it must be locally constant. On the other hand, by computing the Picard-Fuchs equation satisfied by the class of $\omega$ we can verify that this line is not locally constant.
From a different point of view, we can take complex conjugate $\bar\omega$ and try to relate it modulo coboundaries to $\omega$ using relations of the form
\[
\overline{\frac{d z_1}{z_1}\wedge \frac{d z_2}{z_2}} = \frac{d z_1}{z_1}\wedge \frac{d z_2}{z_2} - d\left(\log|z_1|^2 \frac{d z_2}{z_2} + \log|z_2|^2 \overline{\frac{d z_1}{z_1}}\right).
\]
If $C_i$ has eigenvalues whose absolute values are different from $1$, we will see that the class $\bar\omega$ is equivalent to $\omega$ plus a linear combination of classes coming from $*/Z(C_i)$. In the case of rank $2$ on $\BP^1$ with $4$ marked points, the character variety is a union of $\BC^*\times \BC^*$ and $6$ copies of $\BC$. These $\BC$ are attached to six points on $\BP^2\setminus \BC^*\times \BC^*$, which is a union of three lines $\BC$. So we obtain $6$ marked points on a triangle, $2$ on each side. This configuration is specified by three complex numbers, and we expect that the corresponding $\BR$-mixed Hodge structure ``knows'' absolute values of these numbers.
\end{rem}

All the classes constructed above by pullback along $\tilde g_{[\Sigma]}$ produce classes in $H^2(\CX_\parab/G_\parab)$. Note that $\omega_3$ vanishes under the pullback.
\begin{cor}\label{cor:classes on stack}
We have the following classes in $H^2(\CX_\parab/G_\parab)$:
\begin{enumerate}
\item For each $i$ and each character $f:Z(C_i)\to \BC^*$, we have the class represented by $\frac{df}{f}\in\Omega^1(G_\parab\times\CX_\parab)$. This class belongs to $F^1\cap\overline F^1\cap W_2$.
\item The pullback of the class of $\omega=\omega_2+\frac13\omega_3$ constructed above. This class belongs to $F^2\cap W_4$ and is represented by the following form, which we also denote by $\omega$:
\[
\omega = \sum_{\tau\in T_2^\nondeg} \trace(g_{e_1(\tau)}^{-1} d g_{e_1(\tau)} \wedge d g_{e_2(\tau)}\, g_{e_2(\tau)}^{-1}) \in\Omega^2(\CX_\parab),
\]
where $e_1(\tau)$ resp. $e_2(\tau)$ denote the edges $0-1$ resp. $1-2$ of the triangle $\tau$ and $g_{e_1(\tau)}$, $g_{e_2(\tau)}$ are the functions $\CX_\parab\to G$ corresponding to the transfer matrices along these edges. 
\end{enumerate}
\end{cor}

\begin{rem}
	If the genus of $\Sigma$ is positive, we can also construct classes using the  elements of $H_1(\Sigma,\BZ)$, but this construction is not different from non-parabolic case.
\end{rem}

\subsection{Cohomology classes on the character variety}
Suppose $G=\GL_n$. Define the \emph{parabolic character variety} as the GIT quotient $X_\parab=\CX_\parab\sslash G_\parab$. It comes with the augmentation map
\[
\CX_\parab/G_\parab \to X_\parab.
\]
Note that there is a diagonal embedding $\BC^*\to G_\parab$ which acts trivially on $\CX_\parab$. Recall the genericity assumption from \cite{hausel2011arithmetic}:
\begin{defn}\label{def:generic}
The tuple of matrices $C_1,\ldots,C_k\in G$ is \emph{generic} if 
\[
\prod_{i=1}^k \det C_k = 1,
\]
and for every $1\leq r<n$ and a choice of $r$ eigenvalues $\alpha_{i,1},\ldots,\alpha_{i,r}$ out of the list of $n$ eigenvalues of $C_i$ for each $i=1,\ldots,k$ we have 
\[
\prod_{i=1}^k\prod_{j=1}^r \alpha_{i,j}\neq 1.
\]
\end{defn}
It was shown in \cite{hausel2011arithmetic} that
\begin{thm}\label{thm:generic char var}
	If $C_1$,\ldots,$C_k$ is generic, then 
	\begin{enumerate}
		\item the representation variety $\CX_\parab$ is non-singular;
		\item the stabilizer of each point on $\CX_\parab$ equals $\BC^*$;
		\item the GIT quotient $X_\parab$ is non-singular;
		\item the map $\CX_\parab \to X_\parab$ is a principal $G_\parab/\BC^*$-bundle.
	\end{enumerate}
\end{thm}

\begin{rem}
	In \cite{hausel2011arithmetic} the setup was slightly different. The representation variety was defined as consisting of $k+2g$-tuples 
	\[
	\alpha_1,\ldots,\alpha_g, \beta_1,\ldots,\beta_g, M_1,\ldots,M_k \in G
	\]
	satisfying 
	\[
	[\alpha_1,\beta_1]\cdots [\alpha_g, \beta_g] M_1 \cdots M_k = \Id,\qquad M_i\in C_i^G \quad(i=1,\ldots,k).
	\]
	Denote this representation variety by $\CX_\parab'$. Then $G$ acted by conjugation, $X_\parab=\CX_\parab'\sslash G$, and it was shown that $\CX_\parab'\to X_\parab$ is a principal $G/\BC^*$-bundle. To adopt this to our approach, notice that the map $G\to C_i^G$ sending $g$ to $g C_i g^{-1}$ makes $G$ into a principal $Z(C_i)$-bundle over $C_i^G$. Thus our $\CX_\parab$ is a principal $Z(C_1) \times \cdots \times Z(C_k)$-bundle over $\CX_\parab'$, where $Z(C_1) \times \cdots \times Z(C_k)$ is viewed as a normal subgroup of $G_\parab/\BC^*$ with quotient group $G/\BC^*$. So we are in the situation of a tower of principal bundles and therefore $\CX_\parab \to X_\parab$ is a principal bundle.
\end{rem}

We have
\begin{prop}
Suppose $C_i$ are generic. Then the pullback maps $H^\bullet(X_\parab)\to H^\bullet(\CX_\parab/G_\parab)$ are injective.
\end{prop}
\begin{proof}
We factor the map in question as a composition of pullbacks
\begin{equation}\label{eq:composition of pullbacks}
H^\bullet(X_\parab) \to  H^\bullet(\CX_\parab/(G_\parab/\BC^*)) \to H^\bullet(\CX_\parab/G_\parab).
\end{equation}
By the genericity assumption and Theorem \ref{thm:generic char var}, $\CX_\parab$ is a principal homogeneous space for the group $G_\parab/\BC^*$. Therefore the first arrow in \eqref{eq:composition of pullbacks} is an isomorphism.

Define a subgroup
\[
SG_\parab = \{(g_0,\ldots, g_k):\prod_{i=0}^k \det g_i = 1\}\subset G_\parab.
\]
The composition of the two natural homomorphisms
\begin{equation}\label{eq:from SG to PG}
SG_\parab \to G_\parab \to G_\parab/\BC^*
\end{equation}
is surjective with kernel a finite subgroup $\Gamma\subset SG_\parab$. Clearly, $SG_\parab$ is connected and therefore $\Gamma$ acts trivially on $H^\bullet(SG_\parab)$. Hence \eqref{eq:from SG to PG} induces an isomorphism on the cohomology. Therefore the composition of the following pullback maps is an isomorphism:
\[
H^\bullet(\CX_\parab/(G_\parab/\BC^*)) \to H^\bullet(\CX_\parab/G_\parab)\to
H^\bullet(\CX_\parab/SG_\parab).
\]
This implies that the second arrow in \eqref{eq:composition of pullbacks} is injective.
\end{proof}

It is easy to identify which classes constructed in Corollary \ref{cor:classes on stack} come from classes in $H^2(X_\parab)$. For the first type of classes note that any character $f:\prod_{i=1}^k Z(C_i)\to \BC^*$ which vanishes on $\BC^*$ defines a line bundle on $X_\parab$ in a natural way. The first Chern class of this line bundle can be represented on $\CX_\parab/G_\parab$ by the form $\frac{df}{f}$. So classes of the first type are Chern classes of these tautological line bundles. For the second type of classes, note that $d_1^* \omega=0$, $d\omega=0$ implies that the form $\omega$ descends to the quotient $X_\parab$. We obtain

\begin{thm}\label{thm:classes on character variety}
Let $\Sigma$ be a triangulated Riemann surface of genus $g$ with $k$ boundary components.
Suppose $C_i\in\GL_n$ for $i=1,\ldots,k$ is a generic tuple. Then we have the following classes in the cohomology of the parabolic character variety $H^2(X_\parab)$:
\begin{enumerate}
\item For each $i$ and each character $f:\prod_{i=1}^k Z(C_i)\to \BC^*$ which vanishes on $\BC^*$, we have the first Chern class of the associated line bundle. This class belongs to $F^1\cap\overline F^1\cap W_2$.
\item The class of the canonical holomorphic symplectic form on $X_\parab$ which is obtained from the gauge-invariant form on the framed character variety $\CX_\parab$
\[
\omega = \sum_{\tau\in T_2^\nondeg} \trace(g_{e_1(\tau)}^{-1} d g_{e_1(\tau)} \wedge d g_{e_2(\tau)}\, g_{e_2(\tau)}^{-1}) \in\Omega^2(\CX_\parab),
\]
where the sum goes over the triangles of a triangulation of $\Sigma$, where $e_1(\tau)$ resp. $e_2(\tau)$ denote the edges $0-1$ resp. $1-2$ of the triangle $\tau$ and $g_{e_1(\tau)}$, $g_{e_2(\tau)}$ are the functions $\CX_\parab\to \GL_n$ corresponding to the transfer matrices along these edges. This class belongs to $F^2\cap W_4$.
\end{enumerate}
\end{thm}

\begin{rem}
	Using notations of Section \ref{ssec:2 forms}, we can write $\omega$ as follows:
	\begin{equation}\label{eq:form on charvar}
	\omega = (\alpha_1|\beta_1|\alpha_1^{-1}|\beta_1^{-1}|\cdots |\alpha_g|\beta_g|\alpha_g^{-1}|\beta_g^{-1}| \gamma_1^{-1}| q_1| \gamma_1|\cdots|\gamma_k^{-1}| q_k| \gamma_k),
	\end{equation}
	where we denote the transfer matrices and the edges of triangulation by the same letters.
	 One can compare our expression for $\omega$ with the one given in \cite{guruprasad1997group} for the canonical symplectic form on the parabolic character variety and see that they coincide.
\end{rem}

\section{Twisted symplectic varieties}\label{sec:braid variety}

\subsection{Geometric representation of positive braids}
Here we associate varieties to positive braids (for the monoid of positive braids see \cite{garside1969braid}). We expect our construction to be essentially equivalent to \cite{shende2017legendrian}, but our motivation came from an attempt to present certain natural $U\times B$-equivariant sheaves on $G$ in an effective way suitable for computations of the $2$-form and cell decompositions.
\begin{defn}
The \emph{monoid of positive braids} on $n$ strands is given by the following presentation:
\[
\Br_n^+=\left<\sigma_1,\ldots,\sigma_{n-1}: \sigma_i \sigma_j=\sigma_j \sigma_i\;(|i-j|>1),\; \sigma_i \sigma_{i+1} \sigma_i = \sigma_{i+1} \sigma_i \sigma_{i+1} \right>.
\]
\end{defn}
This monoid will have a representation in the following monoid:
\begin{defn}\label{def:TSV}
A \emph{twisted symplectic variety over $\GL_n$} (TSV) is a triple $(X,f,\omega)$ where $X$ is a smooth algebraic variety over $\BC$, $f:X\to \GL_n$ is a regular morphism and $\omega\in\Omega^2(X)$ is a holomorphic $2$-form satisfying
\begin{equation}\label{eq:d omega}
d \omega = -\frac13 \trace\left((f^{-1} df)^{\wedge 3}\right).
\end{equation}
Two triples $(X,f,\omega)$ and $(X',f',\omega')$ are isomorphic if there exists an isomorphism $\varphi:X\to X'$ such that $\varphi^*\omega'=\omega$ and $f=f'\circ \varphi$.
\end{defn}

\begin{defn}The \emph{convolution} of two triples $(X,f,\omega)$ and $(X',f',\omega')$ is defined as the triple 
\[
(X,f,\omega)*(X',f',\omega'):=(X\times X', ff', \omega+\omega'+\trace(f^{-1} df\wedge df'\,f'^{-1})).
\]
\end{defn}
Using notations of Section \ref{ssec:2 forms}, we write the form as 
\[
\omega+\omega'+(f|f').
\]
Using identities from the proof of Proposition \ref{prop:class in BGpar}, it is easy to check that the convolution of two TSVs is again a TSV, and that the convolution is associative. TSVs form a monoidal category.

We will sometimes denote a TSV by $X$, the underlying variety also by $X$ and the corresponding morphism and form by $f_X$ and $\omega_X$ respectively.

For each $i=1,\ldots,n-1$ let
\[
\sigma_i=(\BC, f_i, 0)
\]
where $f_i(z)$ for each $z\in\BC$ is the $n\times n$ matrix obtained from the identity matrix by replacing the $2\times 2$ block at the intersection of rows $i, i+1$ and columns $i,i+1$ by the matrix
\[
\begin{pmatrix}
0 & 1 \\
1 & z
\end{pmatrix}.
\]
In other words, we can write
\[
f_i(z) = \tau_i + z e_{i+1,i+1} = \tau_i(\Id+z e_{i,i+1}),
\]
where $\tau_i\in S_n$ is the transposition and $e_{i,j}$ denotes the elementary matrix with $1$ in the row $i$ and column $j$ and $0$ everywhere else.
Then using the following identity it is easy to check that the elements $\sigma_i$ satisfy the relations of $\Br_n^+$:
\[
f_1(z_1) f_2(z_2) f_1(z_3) = \begin{pmatrix}
0 & 0 & 1\\
0 & 1 & z_1\\
1 & z_3 & z_2
\end{pmatrix} = f_2(z_3) f_1(z_2-z_1 z_3) f_2(z_1).
\]
In this calculation the form can be ignored because on $\sigma_{i} \sigma_{i+1} \sigma_{i}=\sigma_{i+1} \sigma_{i} \sigma_{i+1}$ and $\sigma_i \sigma_j=\sigma_j \sigma_i$ for $|i-j|>1$ the form is zero. So we obtain
\begin{prop}
	There exists a homomorphism of monoids $\rho$ from $\Br_n^+$ to the monoid of TSVs such that $\rho(\sigma_i)=\sigma_i$.
\end{prop}

For a permutation $\pi$, we denote by $\inv(\pi)$ the set of inversions,
\[
\inv(\pi) = \{(i,j):i<j\;\text{and}\;\pi(i)>\pi(j)\}.
\]
\begin{prop}\label{prop:positive lift}
	Let $\beta_\pi\in\Br_n^+$ be the positive lift of a permutation $\pi\in S_n$. Let $\rho(\beta) = (\BC^{l(\pi)}, f, \omega)$. Then 
	\begin{enumerate}
		\item $f$ is an isomorphism with the subset $\pi U_\pi^- \subset GL_n(\BC)$,
		\item $\omega=0$.
	\end{enumerate}
\end{prop}
\begin{proof}
	We prove the first claim by induction on the length $l(\pi)$. Given that the claim holds for $\pi\in S_n$ let $m$ be such that $l(\tau_m \pi)>l(\pi)$. Let $f=\pi u$ for $u:\BC^{l(\pi)} \to U_\pi^-$. Then we have
\[
\tau_m(\Id + z e_{m,m+1}) \pi u = \tau_m \pi (\Id+ z e_{\pi^{-1}(m),\pi^{-1}(m+1)}) u.
\]
We have $\inv(\tau_m \pi) = \inv(\pi) \cup \{(\pi^{-1}(m),\pi^{-1}(m+1))\}$, so the claim follows.

To verify that $\omega=0$ we do the same induction step. So we have to verify that 
\[
(\tau_m(\Id + z e_{m,m+1})| \pi u ) = 0.
\]
This is clear from the following manipulations:
\[
(\tau_m(\Id + z e_{m,m+1})| \pi u ) = (\tau_m | (\Id + z e_{m,m+1})| \pi | u ) = 
(\tau_m | (\Id + z e_{m,m+1}) \pi | u )
\]
\[
= (\tau_m | \pi (\Id+ z e_{\pi^{-1}(m),\pi^{-1}(m+1)}) | u) = (\tau_m \pi | (\Id+ z e_{\pi^{-1}(m),\pi^{-1}(m+1)})  u) = 0.
\]
\end{proof}
\begin{example}
	The smallest example which is not a lift of a permutation is $\beta=\sigma_1^2$, $n=2$. We have $\rho(\beta)=(\BC^2,f,\omega)$ where $f:\BC^2 \to GL_2(\BC)$ is given by
	\[
	f(z_1,z_2) = \begin{pmatrix} 1 & z_2 \\ z_1 & 1+z_1 z_2 \end{pmatrix},
	\]
	and
	\[
	\omega = \trace\left((\tau-z_1 e_{1,1}) e_{2,2} d z_1 \wedge d z_2 e_{2,2} (\tau - z_2 e_{1,1})\right) = d z_1 \wedge d z_2.
	\]
\end{example}

\begin{example}\label{ex:torus TSV}
Another interesting kind of TSVs is constructed as follows. Let $1\leq i<i'\leq n$. Let
\[
s_{i,i'}=(\BC^*,f_{i,i'},0),
\]
where $f_{i,i'}:\BC^*\to T$ sends $z$ to the diagonal matrix
\[
f_{i,i'}(z)=(1,\cdots,z,\cdots, -z^{-1},\cdots,1)\quad\text{$z$ at position $i$, $z^{-1}$ at position $i'$.}
\]
\end{example}

\subsection{Torus equivariance}
Let $T\subset GL_n$ denote the torus consisting of diagonal matrices. Together with the inclusion map to $GL_n(\BC)$ and zero form it defines a TSV. For any $i\in\BZ$, we denote by $T^i$ the TSV with the underlying space $T$, the map $f_{T^i}:T\to GL_n(\BC)$ given by $t\to t^i$, and zero form. We have natural diagonal embeddings $T^{i+j}\to T^i * T^j$ for any $i,j\in \BZ$. If $t\in T$, we denote its coordinates by $t_i$. The action of $\pi\in S_n$ is denoted by ${}^\pi t = \pi t \pi^{-1}$.

If $X$ is a TSV, and $\pi\in S_n$, we construct a TSV $T\times_\pi X$ as follows. Start by taking the convolution
\[
T * X * T^{-1}.
\]
This comes with a map to $GL_n$ and a $2$-form. We restict this data to the subvariety
\[
T\times_\pi X:=\{(t,x,t')\in T\times X \times T\,:\, t = {}^\pi t'\}
\]
and obtain a TSV. As a variety, it is simply $T\times X$, and the map is given by $(t,x)\to {}^\pi t x t^{-1}$.

\begin{defn}\label{def:equivariant TSV}
	A TSV $X$ is called $\pi$-twisted torus equivariant, or simply $\pi$-equivariant, if we have an action of $T$ on $X$ such that the corresponding map $T\times_\pi X \to X$ is a morphism of TSVs.
\end{defn}

\begin{prop}
	If $X$ and $X'$ are respectively $\pi$, $\pi'$-equivariant TSVs, then $X*X'$ is $\pi\pi'$-equivariant.
\end{prop}
\begin{proof}
	The action of $T$ on $X*X'$ is given by
	\[
	(t,x,x') \to ({}^{\pi'} t x, t x').
	\]
	We have a sequence of embeddings
	\[
	T\times_{\pi\pi'}(X * X') \subset T * X * X' * T^{-1}\subset T * X * T^0 * X' * T^{-1} \subset T * X * T^{-1} * T * X' * T^{-1},
	\]
	where the second embedding is given by 
	\[
	(t_1, x, x', t_2) \to (t_1, x, {}^{\pi'} t_2, x', t_2).
	\]
	This allows us to view $T\times_{\pi\pi'}(X * X')$ as the sub-TSV of $T * X * T^{-1} * T * X' * T^{-1}$ consisting of 6-tuples of the form
	\[
	({}^{\pi\pi'} t, x, {}^{\pi'} t, {}^{\pi'} t, x', t).
	\]
	Thus we have an inclusion of TSVs
	\[
	T\times_{\pi\pi'}(X * X')\subset (T\times_{\pi}X)*(T\times_{\pi'}X'),
	\]
	and since the actions $T\times_{\pi}X \to X$, $T\times_{\pi'}X' \to X'$ preserve the TSV structure, we obtain that the action
	\[
	T\times_{\pi\pi'}(X * X') \to X*X'
	\]
	also does.
\end{proof}

We have
\begin{prop}
The building block of the braid monoid representation $\sigma_i$ is $\tau_i$-equivariant
\end{prop}
\begin{proof}
We have ($t\in T$)
\begin{equation}\label{eq:t past sigma}
{}^{\tau_i} t \;\tau_i (\Id + z e_{i,i+1}) t^{-1} =  \tau_i (\Id + t_{i}t_{i+1}^{-1} z e_{i,i+1}).
\end{equation}
So the action sends $z$ to $t_{i}t_{i+1}^{-1} z$.
Next we calculate the form on $T * \sigma_i * T^{-1}$ in coordinates $t, z, t'$:
\[
(t | \tau_i (\Id + z e_{i,i+1}) | t'^{-1}) = (t  \tau_i | (\Id + z e_{i,i+1})  t'^{-1})
= (\tau_i {}^{\tau_i} t | t'^{-1} (\Id +  t'_{i} t'^{-1}_{i+1} z e_{i,i+1}))
\]
\[
=(\tau_i | {}^{\tau_i} t | t'^{-1} | (\Id +  t'_{i} t'^{-1}_{i+1} z e_{i,i+1})) = ({}^{\tau_i} t | t'^{-1}) +  (\tau_i | {}^{\tau_i} t  t'^{-1} | (\Id +  t'_{i} t'^{-1}_{i+1} z e_{i,i+1})).
\]
The second summand vanishes, so we are left with
\[
({}^{\tau_i} t | t'^{-1}) = -\sum_{k=1}^n d \log t_{\tau_i(k)}\wedge d \log t_k',
\]
which clearly vanishes on $T\times_{\tau_i} \sigma_i$.
\end{proof}

The conclusion is
\begin{cor}\label{cor:equivariance}
	Let $\beta\in\Br_n^+$, and let $\pi(\beta)\in S_n$ be the associated permutation. Then the TSV $\rho(\beta)$ is $\pi(\beta)$-equivariant.
\end{cor}

\begin{rem}
The TSV $s_{i,i'}$ from Example \ref{ex:torus TSV} is not equivariant for the identity permutation:
\[
(t|f_{i,i'}(z)|t^{-1}) = (d\log t_i - d\log t_{i'}) \wedge d\log z + d\log z\wedge(d\log t_{i'} - d\log t_i)
\]
\[
=2(d\log t_i - d\log t_{i'}) \wedge d\log z\;\neq\; 0.
\]
On the other hand, for the transposition $\tau=\tau_{i,i'}=(i,i')$, we obtain
\[
({}^\tau t|f_{i,i'}(z)|t^{-1}) = (d\log t_{i'} - d\log t_{i}) \wedge d\log z + d\log z\wedge(d\log t_{i'} - d\log t_i) 
\]
\[
+ d\log t_{i'}\wedge d\log t_i + d\log t_{i}\wedge d\log t_{i'}
=0.
\]
So $s_{i,i'}$ is equivariant for the transposition $\tau_{i,i'}$.
\end{rem}

\subsection{Restriction to the torus}
\begin{defn}
	Let $X$ be a TSV. We denote by $X_0\subset X$ the preimage $f_X^{-1}(B)$, where $B$ is the subgroup of upper-triangular matrices. The composition $X_0\to B \to T$ is denoted by $g$.
\end{defn}
The right hand side of \eqref{eq:d omega} vanishes on $T$, so we obtain:
\begin{prop}
	The restriction of $\omega_X$ to $X_0$ is closed.
\end{prop}
We also have
\begin{prop}
	Suppose $X$ is $\pi$-equivariant. Then $X_0$ is preserved by the $T$-action, and $\omega_X|_{X_0}$ is  $T$-invariant. The composition map $X_0\to B \to T$ is $T$-equivariant with respect to the action of $T$ on itself given by $(t,t') \to {}^\pi t t^{-1} t'$.
\end{prop}
\begin{proof}
Let $X=(X,f,\omega)$. Clearly, $X_0$ is preserved by the action. The pullback along the action map $a:T\times X_0\to X_0$ must be given by the restriction of the form on $T * X_0 * T^{-1}$, so we calculate this restriction. It is given in coordinates $(t,x)$ by
\[
\omega|_{X_0} + (^\pi t|f(x)|t^{-1}).
\]
Since $f(X_0)\subset B$, this simplifies to
\[
\omega|_{X_0} + \sum_{i=1}^n\left( d\log t_{\pi^{-1}(i)} \wedge d\log f_{i,i}(x) - (d\log t_{\pi^{-1}(i)}+d\log f_{i,i}(x))\wedge d\log t_i\right).
\]
So we obtain
\begin{equation}\label{eq:moment map}
a^*\omega|_{X_0} = \omega|_{X_0} + \sum_{i=1}^n d\log t_i \wedge (d\log f_{\pi(i),\pi(i)}(x) + d\log f_{i,i}(x) - d\log t_{\pi(i)}).
\end{equation}
In particular, we see that the form is invariant. The statement about equivariance is clear.
\end{proof}

\begin{rem}
	The triple $(X_0,g,\omega|_{X_0})$ can be viewed as a $\pi$-equivariant TSV over $T$.
\end{rem}

\begin{rem}
	Equation \eqref{eq:moment map} in the case $\pi=\Id$, if we knew that the $2$-form is non-degenerate, would say that the logarithm of $X_0\to B \to T$ is the moment map for the $T$-action. 
\end{rem}

\begin{rem}
Notice that when $\pi\neq \Id$, the action of $T$ on itself above is non-trivial. So we can use this action to partially trivialize the family $X_0\to T$. More precisely, let 
\[
T^\pi=\kernel(\pi - \Id),\quad T_\pi = \cokernel(\pi - \Id),
\]
 where $\pi:T\to T$ is the homomorphism induced by $\pi$. Choose a splitting $T\cong T_1\times T_\pi$. Then the $T$-action trivializes the family in the $T_1$ direction. Restrict the family $X_0\to T$ to $T_\pi$. The resulting object is denoted $X_1$. It is endowed with a map to $T_\pi$ and a fiberwise action of $T^\pi$.
 The dimension of $T_\pi$ equals that of $T^\pi$ and equals to the number of cycles in $\pi$. If the restriction of $\omega$ to $X_1$ is non-degenerate, then the map to $T_\pi$ can be interpreted as the moment map for the $T^\pi$-action. Furthermore, we notice that $\BC^*\subset T^\pi$ acts trivially, which is mirrored by the fact that the image of the map to $T_\pi$ is contained in the subtorus
 \[
\{t\in T_\pi:\det(t)=\sign(\pi)\}.
 \]
\end{rem}

\begin{defn}\label{def:sympred}
	For $t\in T$ satisfying $\det(t)=\sign(\pi)$, we denote by $Y_\beta(t)$ the quotient stack $g^{-1}(t)/(T^\pi/\BC^*)$, where $T^\pi=\{t'\in T: \pi(t')=t'\}$.
\end{defn}

In Proposition \ref{prop:Y quotient} we will show that for generic values of $t$, $Y_\beta(t)$ coincides with the affine GIT quotient.

\subsection{Stratification}\label{ssec:stratification}
Let $\beta\in\Br_n^+$. Let $l=l(\beta)$ be the length of $\beta$ and choose a presentation 
\[
\beta=\sigma_{i_l} \cdots \sigma_{i_1}.
\]
We denote the permutation group on $n$ elements by $W$. Denote
\[
\pi_k = \tau_{i_k} \tau_{i_{k-1}} \cdots \tau_{i_1}\in W \qquad\text(k=0,1,\ldots,l),
\]
so that the associated permutation of $\beta$ is given by $\pi(\beta)=\pi_l$.

The space $\rho(\beta)_0$ is the closed subvariety of $\BC^l$ consisting of tuples $z_1,\ldots,z_l$ such that
\[
f_{i_l}(z_l) \cdots f_{i_2}(z_2) f_{i_1}(z_1) \in B,
\]
where $f_i(z)=\tau_i (\Id + z e_{i,i+1})$. 
Let $\rho(\beta)_0 \to W^{l+1}$ be the map that sends $z=(z_1,\ldots,z_l)$ to the sequence $p(z)=(p_0,p_1,\ldots,p_l)$ such that
\[
f_{i_k}(z_k) \cdots f_{i_1}(z_1) \in B p_k B\qquad(k=0,\ldots,l).
\]
For each point $z$ we have $p_0=p_l=\Id$. Furthermore:
\begin{prop}\label{prop:walk}
	For any point of $\rho(\beta)_0$ and every $k$ we have:
	\begin{equation}\label{eq:walk}
	p_{k+1} = \begin{cases}
	\tau_{i_{k+1}} p_k & \text{if $\tau_{i_{k+1}} p_k>p_k$,}\\
	\tau_{i_{k+1}} p_k \;\text{or}\; p_k & \text{if $\tau_{i_{k+1}} p_k<p_k$.}
	\end{cases}
	\end{equation}
\end{prop}
A sequence $p_0,\ldots,p_m$ satisfying the conditions $p_0=p_m=\Id$ and the conditions \eqref{eq:walk} will be called a \emph{walk}. If $p_{k+1}>p_k$, we will say that we \emph{go up}, if $p_{k+1}<p_k$ we \emph{go down}, and if $p_{k+1}=p_k$ we \emph{stay}. Then the rules \eqref{eq:walk} say that 
\[
\text{if we can go up, we must go up, otherwise we can stay.}
\]
This map $z\to p(z)$ is a continuous map when $W$ is viewed as a poset with Bruhat ordering. As explained in Section \ref{ssec:poset}, we obtain a stratification of $\rho(\beta)_0$
\[
\rho(\beta)_0 = \bigsqcup_{p} C_p
\]
 whose cells $C_p$ correspond to walks $p=(p_0,\ldots,p_l)$. It turns out, every cell has a simple shape. We denote by $U_p$, $D_p$ and $S_p$ the sets of positions where we go up, down and stay respectively. Note that $|U_p|=|D_p|$, and therefore $|S_p|=l(\beta)-2|U_p|$.
\begin{prop}\label{prop:cells}
	For any walk $p$, the cell $C_p$ is $T$-equivariantly isomorphic to $\BC^{U_p} \times (\BC^*)^{S_p}$.
\end{prop}
Before proving this statement, we give an explicit description of the action of $T$ on $(\BC^*)^{S_p}$, of the map from $(\BC^*)^{S_p}$ to $T$, and the restriction of the $2$-form on each cell:
\begin{prop}\label{prop:cell2form}
	Denote the coordinates on $(\BC^*)^{S_p}$ by $a_k$ for $k\in S_p$. For $k\in S_p$, let $\varphi_k:(\BC^*)^{S_p} \to T$ be the map
	\[
	\varphi_k((a_{k'})_{k'\in S}) = (1,\ldots,1,\underset{\text{position}\;p_{k}^{-1}(i_{k}+1)}{a_k},1,\dots,1,\underset{\text{position}\;p_{k}^{-1}(i_{k})}{-a_k^{-1}},1,\ldots,1).
	\]
	\begin{enumerate}
	\item The map $(\BC^*)^{S_p}\to T$ is given by $\prod_{k\in S_p} \varphi_k$.
	\item The action of $T$ on $(\BC^*)^{S_p}$ is given as follows: $(t_1,\ldots,t_n)\in T$ acts by multiplying each coordinate $a_k$ by $t_{\pi_{k}^{-1}(i_{k}+1)} t_{\pi_{k}^{-1}(i_{k})}^{-1}$.
	\item The form $\omega$ is given by
	\[
	\sum_{k,k'\in S_p:\;k>k'} \trace(d\log \varphi_k \wedge d\log \varphi_{k'}).
	\]
	\end{enumerate}	
\end{prop}

\begin{proof}[Proof of Propositions \ref{prop:walk}, \ref{prop:cells}, \ref{prop:cell2form}]
Consider the walk up to the position $k$. Let $X_{k}$ be the locally closed subset of $\BC^{k}$ consisting of tuples $z_1,\ldots,z_{k}$ such that for each $j\leq k$ we have 
\[
f_{i_j}(z_j) \cdots f_{i_1}(z_1) \in B p_j B.
\]
We will show our results by induction $k$. We factor the natural map $X_k\to B p_k B$ as the pointwise product $u_k p_k t^{(k)} v_k$ where $u_k:X_k\to U_{p_k}^-$, $t^{(k)}:X_k \to T$ and $v_k:X_k \to U$. We also decompose the $2$-form on $X_k$ as
\[
\omega = \omega_{\log} + (u_k |p_k t^{(k)} v_k).
\]
Now let us see what happens when we pass to $X_{k+1}$. Let $i=i_{k+1}$, $z=z_{k+1}$. Consider
\[\tag{*}
f_i(z) u_k(x) p_k t^{(k)}(x) v_k(x) = \tau_i (\Id + z e_{i,i+1}) u_k(x) p_k t^{(k)}(x) v_k(x).
\]
We have $2$ cases:

Case 1. Suppose $\tau_i p_k > p_k$. Apply \eqref{eq:bruhat} to $\tau_i (\Id + z e_{i,i+1}) u_k(x)$ to write
\[
(*)=u_k'(x,z) \tau_i (\Id + z'(x,z) e_{i,i+1}) p_k t^{(k)}(x) v_k(x).
\]
Since $\tau_i p_k>p_k$, we have $(i,i+1)\notin \inv(p_k^{-1}$). So ${}^{p_k^{-1}} (\Id + z'(z,x) e_{i,i+1}) \in U$ and we obtain
\[
(*)= u_k'(x,z) \tau_i p_k t^{(k)}(x) v_k'(x,z)
\]
for $u_k', v_k':X\times\BC \to U$. We see that the new Bruhat cell is $\tau_i p_k$ (if we can go up we go up), $X_{k+1}=X_k\times\BC$, and the map to $T$ did not change. The form on $X_{k+1}$ is given by
\[
\omega_{\log} + (u_k |p_k t^{(k)} v_k) + (\tau_i (\Id + z e_{i,i+1})| u_k p_k t^{(k)} v_k) = \omega_{\log} + (\tau_i (\Id + z e_{i,i+1})| u_k | p_k t^{(k)} v_k).
\]
\[
\omega_{\log} + (\tau_i | (\Id + z e_{i,i+1}) u_k | p_k t^{(k)} v_k) = \omega_{\log} + (u_k'(x,z) |\tau_i (\Id + z'(x,z) e_{i,i+1}) | p_k t^{(k)} v_k),
\]
where we used Proposition \ref{prop:bruhat form} to obtain the last equality. It is clear now that by inserting and removing $|$ we can move $(\Id + z e_{i,i+1})$ past $p_k$ and $t^{(k)}$ and arrive at
\[
\omega_{\log} + (u_k'(x,z) | \tau_i p_k t^{(k)}(x) v_k'(x,z)).
\]
So $\omega_{\log}$ on $X_{k+1}$ is the same as the one on $X_k$.

Case 2. Suppose $\tau_i p_k < p_k$. Write $p_k = \tau_i p'$ for $p'=\tau_i p_k$. Applying \eqref{eq:bruhat} to $u_k \tau_i$, we obtain
\[
u_k(x) \tau_i = (\Id + z_0(x) e_{i,i+1}) \tau_i u_k'(x),
\]
where $u_k':X_k\to U_{\tau_i}^+$. We decompose $\BC\times X_k$ into an open and a closed set. The closed set is 
\[
Z:=\{(z,x)\in\BC\times X_k:z=-z_0(x)\}.
\]
On $Z$ we have
\[
(*) = \tau_i (\Id - z_0(x) e_{i,i+1}) (\Id + z_0(x) e_{i,i+1}) \tau_i u_k'(x) p' t^{(k)} v_k(x) = u_k'(x) p' t^{(k)} v_k(x).
\]
So $p_{k+1}=p'$, which means we are going down.
On the complement of $Z$ we have
\[
(*) = \tau_i (\Id + z e_{i,i+1}) (\Id + z_0(x) e_{i,i+1}) \tau_i u_k'(x) p' t^{(k)} v_k(x).
\]
The product $\tau_i (\Id + z e_{i,i+1}) (\Id + z_0(x) e_{i,i+1}) \tau_i$ is non-trivial only in the $i$-th and $i+1$-th rows and columns where it looks like
\[
\begin{pmatrix} 0 & 1 \\ 1 & 0 \end{pmatrix} \begin{pmatrix} 1 & z + z_0(x) \\ 0 & 1 \end{pmatrix} \begin{pmatrix} 0 & 1 \\ 1 & 0 \end{pmatrix} = \begin{pmatrix} 1 & 0\\z + z_0(x) & 1 \end{pmatrix}.
\]
Let $a(z,x)=z+z_0(x)$, this is a map $\BC\times X_k\setminus Z\to \BC^*$. We Bruhat-decompose
\begin{equation}\label{eq:bruhat2by2}
\begin{pmatrix} 1 & 0 \\ a & 1 
\end{pmatrix} = \begin{pmatrix} -a^{-1} & 1 \\ 0 & a \end{pmatrix} \begin{pmatrix} 0 & 1 \\ 1 & 0 \end{pmatrix} \begin{pmatrix} 1 & a^{-1} \\ 0 & 1 \end{pmatrix} = \begin{pmatrix} -a^{-1} & 0 \\ 0 & a \end{pmatrix} \begin{pmatrix} 1 & -a \\ 0 & 1 \end{pmatrix} f_i(a^{-1}).
\end{equation}
We have $\tau_i p'>p'$, so, as shown in Case 1, (*) is in the Bruhat cell $\tau_i p'=p_k$. Hence outside $Z$ we stay. At this point Proposition \ref{prop:walk} has been proven. We split Case $2$ into two subcases:

Case 2a. Suppose $p_{k+1}=\tau_{k+1} p_k$. Then $X_{k+1}=Z\cong X_k$. Clearly, the map to $T$ does not change. The new form is
\[
\omega_{\log} + (\tau_i (\Id + z e_{i,i+1})| u_k | \tau_k p' t^{(k)} v_k) = \omega_{\log} + (\tau_i (\Id + z e_{i,i+1})| u_k| \tau_k | p' t^{(k)} v_k)
\]
\[
= \omega_{\log} + (\tau_i (\Id + z e_{i,i+1})| (\Id + z_0(x) e_{i,i+1}) |\tau_i u_k'(x) | p' t^{(k)}(x) v_k(x)) 
\]
\[
= \omega_{\log} + (u_k'(x) | p' t^{(k)}(x) v_k(x)),
\]
so $\omega_{\log}$ did not change.

Case 2b. Suppose $p_{k+1}=p_k$. Then $X_{k+1} = \BC\times X_k\setminus Z \cong \BC^* \times X_k$, where the coordinate on $\BC^*$ is $a=z+z_0(x)$. Let $t':\BC^*\to T$ be the map that sends $a$ to the matrix with $-a^{-1}$ at position $i$, $a$ at position $i+1$ and $1$ at all the other diagonal entries. Because of \eqref{eq:bruhat2by2}, the new map to $T$ is given by 
\[
{}^{p_k^{-1}} t'(a) t^{(k)}(x).
\]
To calculate the new form, we begin as in Case 2a and obtain
\[
\omega_{\log} + (\tau_i (\Id + a e_{i,i+1}) \tau_i |u_k'(x) | p' t^{(k)}(x) v_k(x))
\]
\[
= \omega_{\log} + ( t'(a) (\Id  - a e_{i,i+1}) f_i(a^{-1}) |u_k'(x) | p' t^{(k)}(x) v_k(x)).
\]
Since $( t'(a) (\Id  - a e_{i,i+1})| f_i(a^{-1})=0$ (after expanding the definition, we only get terms $da\wedge da$ which are zero), we can continue as 
\[\tag{**}
=\omega_{\log} + ( t'(a) (\Id  - a e_{i,i+1}) | f_i(a^{-1}) |u_k'(x) | p' t^{(k)}(x) v_k(x)). 
\]
Let $f_i(a^{-1}) u_k'(x)|p' t^{(k)} v_k=u_k''(x) p_k t^{(k)} v_k'(x)$ be the Bruhat decomposition. As in the proof of Case 1, we obtain
\[
(f_i(a^{-1}) |u_k'(x) | p' t^{(k)} v_k) = (u_k''(x) | p_k t^{(k)}(x) v_k'(x)).
\]
So we obtain
\[
(**) = \omega_{\log} + ( t'(a) (\Id  - a e_{i,i+1}) | u_k''(x) | p_k t^{(k)}(x) v_k'(x)) 
\]
\[
= \omega_{\log} + ( t'(a) (\Id  - a e_{i,i+1}) u_k''(x) | p_k t^{(k)}(x) v_k'(x)).
\]
Let $t'(a) (\Id  - a e_{i,i+1}) u_k''(x) = u_k'''(a,x) t'(a)$. Then
\[
(**) = \omega_{\log} + (u_k'''(a,x)| t'(a)|p_k t^{(k)}(x) v_k'(x)) = \omega_{\log} + (u_k'''(a,x)| t'(a)|p_k| t^{(k)}(x)| v_k'(x))
\]
\[
\omega_{\log} + (u_k'''(a,x)| p_k {}^{p_k^{-1}}t'(a) t^{(k)}(x)| v_k'(x)) + ({}^{p_k^{-1}}t'(a) | t^{(k)}(x)).
\]
So the new $\omega_{\log}$ is given by
\[
\omega_{\log} + ({}^{p_k^{-1}}t'(a) | t^{(k)}(x)).
\]
Note that when we finally arrive at $k=l$, we have
\[
\omega = \omega_{\log} + (u_l |p_l t^{(l)} v_l),
\]
where the second term is zero because $p_l=\Id$.

We have seen that in each case we have $X_{k+1}=X_k$, $X_{k+1}=\BC\times X_k$ or $X_{k+1}=\BC^*\times X_k$. This implies Proposition \ref{prop:cells}. When we follow how the map to $T$ and $\omega_{\log}$ is updated, we obtain claims (i) and (iii) of Proposition \ref{prop:cell2form}. To see how $T$ acts on $(a_1,\ldots,a_m)$ we note that, in notations of Case 2b, the coordinate $z$ is multiplied by $t_{\pi_{k}^{-1}(i+1)} t_{\pi_{k}^{-1}(i)}^{-1}$ by \eqref{eq:t past sigma}. Since the walk stratification is invariant under $T$, the critical value $z_0$ must be also multiplied by the same factor. Hence the sum $a=z+z_0$ is also multiplied by the same factor. This proves claim (ii) of Proposition \ref{prop:cell2form}.
\end{proof}	

We recognize in Proposition \ref{prop:cell2form} convolution of TSVs from Example \ref{ex:torus TSV}.
\begin{cor}\label{cor:cell is a convolution}
	For any walk $p$, the torus component of the cell $C_p$, together with the $2$-form, the map to $T$, and the $T$-action is $T$-equivariantly isomorphic to the convolution of TSVs of the form $s_{p_k^{-1}(i_k+1),p_k^{-1}(i_k)}$ where $k$ goes over $k\in S_p$ and the convolution is performed from right to left.
\end{cor}
\begin{proof}
	Only the statement about $T$-action is not obvious from Proposition \ref{prop:cell2form}. For each $k\in S_p$, we claim that the partial convolution
	\[
	\prod_{k'\leq k,\,k'\in S_p} s_{p_k^{-1}(i_k+1),p_k^{-1}(i_k)}
	\]
	is equivariant for the permutation $p_k^{-1} \pi_k$. To show this, we go from $k=1$ to $k=l$. If $k\notin S_p$, then both $p_k$ and $\pi_k$ are multiplied by the transposition $\tau_{i_k}$, so that $p_k^{-1} \pi_k$ does not change. If $k\in S_p$, it is multiplied by the transposition $(p_k^{-1}(i_k+1),p_k^{-1}(i_k))$, as expected.
	
	Hence the torus action for the convolution is given by multiplying the diagonal matrix 
	\[
	(1,\ldots,1,\underset{\text{position}\;p_{k}^{-1}(i_{k}+1)}{a_k},1,\dots,1,\underset{\text{position}\;p_{k}^{-1}(i_{k})}{-a_k^{-1}},1,\ldots,1)
	\]
	on the left by ${}^{p_k^{-1} \pi_k} t$ and on the right by ${}^{p_{k-1}^{-1} \pi_{k-1}} t^{-1}$. So $a_k$ is multiplied by
	\[
	\frac{({}^{p_k^{-1} \pi_k} t)_{p_{k}^{-1}(i_{k}+1)}}{({}^{p_k^{-1} \pi_k} t)_{p_{k}^{-1}(i_{k})}} = \frac{t_{\pi_k^{-1}(i_k+1)}}{t_{\pi_k^{-1}(i_k)}}.
	\]
\end{proof}

\subsection{Some examples}

\begin{example}[$T(2,4)$]
As an example, consider $n=2$ and $\beta=\sigma_1^4$. The variety $X=\rho(\beta)$ is $\BC^4$ with coordinates $z_1,\ldots z_4$. Multiplying the matrices $\begin{pmatrix} 0 & 1\\1 & z_i\end{pmatrix}$ produces
\[
\begin{pmatrix}
z_{2} z_{3} + 1 & z_{1} z_{2} z_{3} + z_{1} + z_{3} \\
z_{2} z_{3} z_{4} + z_{2} + z_{4} & z_{1} z_{2} z_{3} z_{4} + z_{1} z_{2} + z_{1} z_{4} + z_{3} z_{4} + 1
\end{pmatrix}.
\]
So $X_0$ is given by the equation $z_{2} z_{3} z_{4} + z_{2} + z_{4}=0$. The map $X_0\to \BC^*\subset \BC^{*2}$ is given by $z_2 z_3 +1$. The action of $\lambda\in\BC^*$ is given by 
\[
\lambda (z_1,z_2,z_3,z_4) = (\lambda z_1, \lambda^{-1} z_2, \lambda z_3, \lambda^{-1} z_4).
\]
There are two possible walks:
\begin{itemize}
	\item $(1,\tau_1,1,\tau_1,1)$ (go up, go down, go up, go down). The second walk corresponds to the closed subvariety $z_2=z_4=0$ in $X_0$. So it is isomorphic to $\BC^2$ with coordinates $(z_1,z_3)$, the action is by scaling, and the map to $\BC^*$ is constant $1$. 
	\item $(1,\tau_1,\tau_1,\tau_1,1)$ (go up, stay, stay, go down) corresponds to the open subvariety $z_2\neq 0$ in $X_0$. This condition implies $z_4\neq 0$. Then $z_3$ can be computed by
\[
z_3 = -\frac{z_2+z_4}{z_2 z_4}.
\]
So we obtain $\BC\times\BC^{*2}$ with coordinates $z_1,z_2,z_4$. The map to $\BC^*$ is given by $-\frac{z_2}{z_4}$. We can gauge out the $\BC^*$-action by setting $z_4=1$, so the quotient is the variety $\BC\times \BC^*$  with coordinates $(z_1,z_2)$ mapped to $\BC^*$ by the map $(z_1,z_2)\to -z_2$. 
\end{itemize}
We see that the preimage of $t\in\BC^*\setminus\{1\}$ intersects only the open subvariety, the intersection is isomorphic to $\BC\times\BC^*$ and the quotient by the $\BC^*$-action is $\BC$. On the other hand, the preimage of $t=1$ consists of two irreducible components and the quotient by the $\BC^*$-action is not a variety because of the fixed point $(0,0,0,0)$.
\end{example}

\section{Seifert surfaces}\label{sec:seifert}

\subsection{Construction}
Vivek Shende explained to me that the stratification of Section \ref{ssec:stratification} is equivalent to the stratification of the moduli space of sheaves with microlocal support constructed in \cite{shende2017legendrian} (in particular, see Proposition 6.31). Therefore, the cells are related to the moduli space of rank $1$ local systems on the corresponding Seifert surface and the 2-form should arise from the intersection form on the surface. Below we implement this idea in a precise combinatorial construction. We verify that the data coming out of Proposition \ref{prop:cell2form} can be conveniently encoded in terms of the homology groups of the corresponding Seifert surface. 

\begin{defn}
	A \emph{marked surface} $(\Sigma,A,B)$ is a (not necessarily connected) topological oriented surface $\Sigma$ with boundary, together with a finite sets  of marked points $A$ and $B$ on the boundary colored black and white respectively. We require that each connected component has at least one boundary component, each boundary component has at least two marked points, and the colors of the marked points on each boundary component interchange.
\end{defn}

The basic examples are presented on the following pictures:
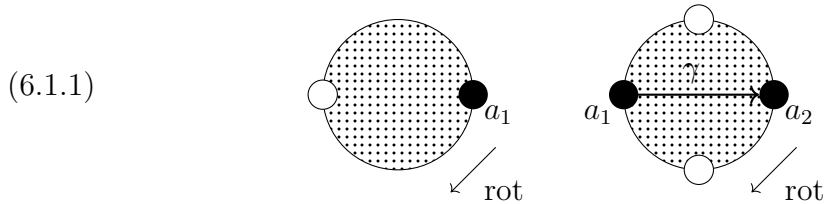
\begin{equation}\label{eq:marked surfaces}
\begin{tikzpicture}[baseline=0]
\draw[fill,pattern=dots] (0,0) circle (1);
\draw (1,0) node[circle,fill] {} node[below right]{$a_1$};
\draw (-1,0) node[circle,fill=white,draw] {};
\draw[<-] (0.7,-1.3) to node[midway,below right] {$\rot$} (1.3,-0.7);
\draw[fill,pattern=dots] (4,0) circle (1);
\draw (3,0) node[circle,fill] {} node[below left]{$a_1$};
\draw (4,1) node[circle,fill=white,draw] {};
\draw (4,-1) node[circle,fill=white,draw] {};
\draw (5,0) node[circle,fill] {} node[below right]{$a_2$};
\draw[<-] (4.7,-1.3) to node[midway,below right] {$\rot$} (5.3,-0.7);
\draw[->,thick] (3,0) to node[midway,above] {$\gamma$} (4.8,0);
\end{tikzpicture}
\end{equation}
Let $\Lambda=H_1(\Sigma, A)$, $\Lambda'=H_1(\Sigma, B)$. A crucial observation is that Thom duality provides us with a perfect pairing
\begin{equation}\label{eq:perfect pairing}
\Lambda\times\Lambda' \to \BZ.
\end{equation}
Let $\rot$ be the rotation operator, defined up to homotopy, that rotates the boundary components clockwise so that $A$ goes to $B$ and vice versa. There are actions
\[
\rot: \Lambda\to \Lambda', \;\Lambda'\to\Lambda,\;A\to B,\;B\to A.
\]
Define $\pi:A\to A$  by $\pi=\rot^2$.
Define the operators $\varphi:\Lambda \to \BZ^A$, $\psi:\BZ^A\to \Lambda$ by
\[
\varphi(\gamma) = \partial(\gamma), \quad \psi(v)\cdot\rot(\gamma) = (v,\partial(\gamma))\quad(\gamma\in\Lambda,\;v\in\BZ^A),
\]
$(\cdot,\cdot)$ is the standard scalar product on $\BZ^A$. Alternatively, $\psi$ applied to the basis vector corresponding to $a\in A$ is the path which goes from $a$ to $\pi(a)$ by following the boundary. In particular, we have 
\[
\varphi\circ \psi = \rot^2-1 = \pi-1.
\]
Thus, setting $\gamma=\psi(v')$, we obtain
\[
\psi(v)\cdot \rot(\psi(v')) = (v,\pi(v')) - (v,v').
\]
Another relation can be obtained by plugging in $\rot^{-2}(\gamma)$:
\[
(v,\pi^{-1} \varphi(\gamma)) = (v,\partial(\rot^{-2})(\gamma))=\psi(v)\cdot\rot^{-1}(\gamma) = \rot(\psi(v))\cdot \gamma.
\]
Introduce the following bilinear form on $\Lambda$
\begin{equation}\label{eq:omega definition for surface}
\omega(\gamma,\gamma') = \rot(\gamma)\cdot\gamma'.
\end{equation}
Then $\varphi$ and $\psi$ are related by
\begin{equation}\label{eq:phi psi omega}
\omega(\gamma, \psi(v)) = -(v,\varphi(\gamma)),\quad \omega(\psi(v),\gamma) = (\pi(v),\varphi(\gamma)).
\end{equation}
\begin{example}
	In the picture on the left of \eqref{eq:marked surfaces}, we have $\Lambda=0$, so all the data is trivial. In the picture on the right, we have $\Lambda=\BZ$ with the generator $\gamma=a_1 a_2$. So we have
	\[
	\gamma\cdot\rot(\gamma) = -1,\;\varphi(\gamma)=[a_2]-[a_1],\;\psi(a_1)=\gamma,\;\psi(a_2)=-\gamma.
	\]
\end{example}
\subsection{Gluing}\label{ssec:gluing}
Suppose $(\Sigma, A, B)$ is a marked surface, and let $a,a'\in A$ be two distinct black marked points such that $\pi(a)\neq a'$. In such situation, we can glue $(a,\rot(a))$ to $(a',\rot^{-1}(a'))$:
\[
\begin{tikzpicture}
\draw[fill,pattern=dots] (-0.5,1) to [out=-70,in=120] (0,0) node[left]{$a=a'$}
to (1,0) to [out=60,in=-110] (1.5,1) to (-0.5,1);
\draw[fill,pattern=dots] (0,0) to [out=-120,in=70] (-0.5,-1) to (1.5,-1) to [out=110,in=-60] (1,0) node[right]{$\rot(a)=\rot^{-1}(a')$} to (0,0);
\draw(0,0) node[circle,inner sep=2pt,fill]{};
\draw(1.5,1) node[circle,inner sep=2pt,fill]{} node[right] {$\pi^{-1}(a')$};
\draw(1.5,-1) node[circle,inner sep=2pt,fill]{} node[right] {$\pi(a)$};
\draw(1,0) node[circle,inner sep=2pt,fill=white,draw]{};
\draw(-0.5,1) node[circle,inner sep=2pt,fill=white,draw]{} node[left] {$\rot(a')$};
\draw(-0.5,-1) node[circle,inner sep=2pt,fill=white,draw]{} node[left] {$\rot^{-1}(a)$};
\end{tikzpicture}
\]
Denote the resulting surface by $\bar\Sigma$. Let $\bar\Lambda$, $\bar\varphi$ and so on denote the corresponding invariants of $\bar\Sigma$. We have a natural homomorphism
\[
\iota:\Lambda \to \bar\Lambda,
\]
which is surjective because every path on $\bar\Lambda$ which intersects the line of gluing can be deformed to make it pass through the point $a$ and then broken into parts, all in $\Sigma$. Similarly, there is a surjective homomorphism
\[
\iota':\Lambda' \to \bar\Lambda'.
\]
The existence of perfect pairing \eqref{eq:perfect pairing} implies that both $\iota$ and $\iota'$ are injective:
\[
\iota(\gamma)=0\;\;\Rightarrow\;\; \gamma\cdot \gamma'=\iota(\gamma)\cdot \iota'(\gamma')=0\;\text{for all $\gamma'\in\Lambda'$}\;\;\Rightarrow\;\; \gamma=0.
\]
So we have $\bar\Lambda=\Lambda$, $\bar\Lambda'=\Lambda'$. The boundary operator $\varphi$ remains unchanged, $\bar\varphi=\varphi$, but the new operator $\overline\rot$ is different from $\rot$. For $\gamma\in\Lambda$, the value of $\overline\rot(\gamma)$ differs from $\rot(\gamma)$ if the boundary of $\gamma$ contains $a$. More precisely, we have
\[
\overline\rot(\gamma) = \rot(\gamma)+([a],\varphi(\gamma)) \rot^{-1}\psi([a']) = \rot\left(\gamma+([a],\varphi(\gamma))\psi([\pi^{-1}(a')])\right).
\]
The new value of $\psi$ is given by
\[
\overline\psi(v) = \psi(v) + (v,[a])\psi([a']-[a])+(v,[\pi^{-1}(a')])\psi([a]),
\]
so it is convenient to identify $\bar{A}=A/\{a=a'\}$ with $A\setminus\{a\}$ to obtain a simpler formula
\begin{equation}\label{eq:gluing psi}
\bar\psi(v) = \psi(v) +(v,[\pi^{-1}(a')])\psi([a]) = \psi\left(v+(v,[\pi^{-1}(a')])[a]\right).
\end{equation}
To verify, we compute
\[
\bar\psi(v)\cdot\overline\rot(\gamma) = \left(v+(v,[\pi^{-1}(a')])[a],\; 
\varphi(\gamma+([a],\varphi(\gamma))\psi([\pi^{-1}(a')]))\right)
\]
\[
=(v, \varphi(\gamma)) + (v,[\pi^{-1}(a')])\;([a],\varphi(\gamma))+([a],\varphi(\gamma))\;(v,\varphi\psi([\pi^{-1}(a')]))
\]
\[+([a],\varphi(\gamma))\;(v,[\pi^{-1}(a')])\;([a],\varphi\psi([\pi^{-1}(a')])).
\]
The last term vanishes by the assumption $a\neq a'$, $a\neq \pi^{-1}(a')$, and after cancellations we obtain
\[
\bar\psi(v)\cdot\overline\rot(\gamma) = (v, \varphi(\gamma)) + ([a],\varphi(\gamma))(v,[a']) = \left(v,\;\bar\varphi(\gamma)\right),
\]
as expected. Finally, we compute
\begin{multline}\label{eq:gluing omega}
\bar\omega(\gamma,\gamma') = \omega(\gamma,\gamma') + ([a],\varphi(\gamma)) \;(\rot^{-1}\psi([a'])) \cdot \gamma'
\\
=\omega(\gamma,\gamma') + ([a],\varphi(\gamma)) \;\psi([a']) \cdot \rot(\gamma') = \omega(\gamma,\gamma') + ([a],\varphi(\gamma)) \;([a'],\varphi(\gamma')).
\end{multline}
\subsection{Toric TSVs from surfaces}\label{ssec:toric TSVs from surfaces}
\begin{defn}
	A marked surface $(\Sigma,A,B)$ is called \emph{labeled} if the set $A$ is identified with the set $1,\ldots,n$.
\end{defn}
To the data of labeled marked surface $(\Sigma,A,B)$ we associate a TSV whose underlying vector space is given by $X=\BC^*\otimes \Lambda$, the map to $G$ factors through $T$ and is given by $\varphi$, the action of $T$ is given by $\psi$. The form is given by the antisymmetrization of $\omega$. We claim that the resulting TSV is $\pi$-equivariant. To verify this, we use \eqref{eq:phi psi omega} as follows:
\begin{equation}\label{eq:omega translation}
\omega(\gamma+\psi(v), \gamma'+\psi(v')) = \omega(\gamma,\gamma')+(\pi(v),\varphi(\gamma'+\psi(v'))) - (v',\varphi(\gamma)).
\end{equation}
This should be compared with the form $\omega+({}^\pi t|\varphi|t^{-1})$ on $T\times X$, whose value on a pair of cocharacters $(\gamma,v),(\gamma',v')$ of $T\times X$ is given by the antisymmetrization of 
\[
\omega(\gamma,\gamma') + (\pi(v),\varphi(\gamma')) - (\varphi(\gamma), v') - (\pi(v),v').
\]
Subtracting this expression from \eqref{eq:omega translation} produces
\[
(\pi(v),\varphi(\psi(v'))+v') = (v,v'), 
\]
which is symmetric in $v$ and $v'$, so the antisymmetrization sends it to zero.

\begin{example}\label{ex:surface for s}
	To obtain a TSV similar to $s_{i,i'}$ of Example \ref{ex:torus TSV}, we take a disjoint union of a disk with $2$ black dots labeled $i, i'$, and $n-2$ copies of a disk with $1$ black dot distributing the remaining labels. The difference of this TSV from the one of Example \ref{ex:torus TSV} is absence of the minus sign in the map $\BC^*\to T$. This difference is not important.
\end{example}

The convolution of TSVs precisely corresponds to gluing of surfaces as follows. Suppose two labeled marked surfaces $(\Sigma,A,B)$ and $(\Sigma',A',B')$ are given. To distinguish $A$ from $A'$, we write $A=\{1,\ldots,n\}$, $A'=\{1',\ldots,n'\}$. Applying the operation of Section \ref{ssec:gluing} with $a=i$, $a'=i'$ for each $i=1,\ldots,n$ to the disjoint union of $\Sigma$ and $\Sigma'$ produces a surface $(\bar\Sigma,\bar A,\bar B)$ with the following invariants:
\begin{enumerate}
	\item The set $\bar A$ is identified with $A\sqcup A'\setminus A=A'$.
	\item The lattice $\bar\Lambda$ is the direct sum of the corresponding lattices for $\Sigma$ and $\Sigma'$.
	\item The operator $\bar\varphi$ is the sum of the corresponding operators $\varphi$ and $\varphi'$.
	\item By \eqref{eq:gluing psi}, the operator $\bar\psi$ is given by
	\[
	\bar\psi(v) = \psi'(v) + \sum_{i=1}^n (v,[\pi'^{-1}(i')]) \psi([i]) = \psi'(v) + \sum_{i=1}^n (\pi'(v),[i']) \psi([i]) 
	\]
	\[
	= \psi'(v) + \psi(\pi'(v))\quad(v\in\BZ^{A'}).
	\]
	\item The permutation $\bar\pi$ sends $\pi'^{-1}(a')$ to $\pi(a)'$, so it is given by the composition $\pi\pi'$.
	\item Finally, by \eqref{eq:gluing omega}, the $2$-form is given by $\bar\omega=\omega+\omega'+(\varphi|\varphi')$.
\end{enumerate}
So we have
\begin{prop}\label{prop:gluing}
	Let $(\Sigma,A,B)$ and $(\Sigma',A',B')$ be labeled marked surfaces with corresponding TSVs denoted by $X$, $X'$. Let $(\bar\Sigma,\bar A, \bar B)$ be the surface obtained by gluing edge $i \rot(i)$ to edge $i' \rot^{-1}(i')$ for each $i=1,\ldots,n$. Then the TSV associated to the labeled marked surface $(\bar\Sigma,\bar A, \bar B)$ is isomorphic to the convolution $X*X'$.
\end{prop}

\subsection{Non-degeneracy}
Let $(\Sigma,A,B)$ be a labeled marked surface. Similarly to Definition \ref{def:sympred}, for $t\in T$ in the image of $\varphi$ consider the quotient $Y(t)=\varphi^{-1}(t)/T^\pi$. This is a torus whose cocharacter lattice can be computed by taking the quotient of $\kernel\varphi$ by the image of $\psi$ on $\pi$-invariant cocharacters. We have an exact sequence
\begin{equation}\label{eq:first exact sequence}
0\to H_1(\Sigma) \to H_1(\Sigma,A) \xrightarrow{\varphi} H_0(A)\to H_0(\Sigma)\to 0, 
\end{equation}
which implies that $\kernel\varphi=H_1(\Sigma)$. Since the rotation operator acts trivially on $H_1(\Sigma)$, the restriction of $\omega$ to $\varphi^{-1}(t)$ is given by the intersection form. The image of $\psi$ on $\pi$-invariant characters is precisely spanned by the classes of boundary components, and we have an exact sequence
\begin{equation}\label{eq:second exact sequence}
0\to H_2(\bar\Sigma) \to H_1(\partial\Sigma) \to H_1(\Sigma) \to H_1(\bar\Sigma) \to 0,
\end{equation}
where $\bar\Sigma$ is obtained from $\Sigma$ by attaching disks along the boundary components. This shows that the cocharacter lattice of $Y(t)$ is precisely given by $H_1(\bar\Sigma)$, and $\omega$ is given by the intersection form, which is non-degenerate. Therefore we obtain

\begin{prop}
	For any labeled marked surface, the corresponding variety $Y(t)$ together with the restriction of $\omega$ is a symplectic torus.
\end{prop}

Proposition \ref{prop:torus lefschez} implies
\begin{cor}\label{cor:lefschetz for surface}
	The class of $\omega$ in $H^2(Y(t))$ satisfies the curious Lefschetz property with middle weight $\dim H_1(\bar\Sigma)$.
\end{cor}

\subsection{Connection to the stratification of braid varieties}\label{ssec:surface from walk}
For convenience, we recall the main characters of the story. $\beta=\sigma_{i_l} \cdots \sigma_{i_1}$ is a positive braid on $n$ strands with the associated permutation $\pi=\pi(\beta)$. $X=\rho(\beta)$ is an affine space equipped with a map $f$ to $G$ and we defined $X_0=\rho(\beta)_0=f^{-1}(B)$. $X$ is equipped with a holomorphic $2$-form $\omega$, which restricts to a closed form on $X_0$. $X_0$ is equipped with a map $g$ to $T$ and a $T$-action that are compatible in the way prescribed by the permutation $\pi$. For a point $t\in T$, we introduced the quotient stack $Y_\beta(t)=g^{-1}(t)/(T^\pi/\BC^*)$, where $T^\pi\subset T$ is the subtorus of $\pi$-invariant elements. Expression \eqref{eq:moment map} implies that $\omega$ induces a closed form on $Y_\beta(t)$.

The space $\rho(\beta)_0$ has a stratification so that each cell by Corollary \ref{cor:cell is a convolution}, Example \ref{ex:surface for s} and Proposition \ref{prop:gluing} looks like a product of an affine space and a TSV associated to a surface.

More explicitly, out of the data of a braid $\beta$ and a walk $p$ the corresponding labeled marked surface $\Sigma$ is constructed as follows. Begin by placing $n$ half-disks next to each other labeling them by numbers from $1$ to $n$:
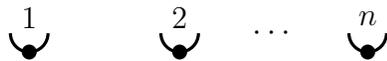
\[
\begin{tikzpicture}[scale=1,baseline=-2pt]
\draw[very thick] (-3,0) to [out=-90,in=180] (-2.75,-0.25) node[circle,inner sep=2pt,fill]{} to [out=0,in=-90] (-2.5,0);
\draw[very thick] (-1,0) to [out=-90,in=180] (-.75,-0.25) node[circle,inner sep=2pt,fill]{} to [out=0,in=-90] (-.5,0);
\draw[very thick] (1.5,0) to [out=-90,in=180] (1.75,-0.25) node[circle,inner sep=2pt,fill]{} to [out=0,in=-90] (2,0);
\draw (0.5,0) node{$\cdots$};
\draw(-2.75,0.2) node{$1$};
\draw(-.75,0.2) node{$2$};
\draw(1.75,0.2) node{$n$};
\end{tikzpicture} 
\]
Then go from $k=1$ to $k=l(\beta)$.	For each crossing $\sigma_{i_k}$ of $\beta$ which is not a stay of $p$, we draw the following picture:
	
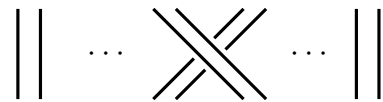
\[
	\begin{tikzpicture}[scale=0.6,baseline=-2pt]
	\draw[very thick] (-3,1) -- (-3,-1);
	\draw[very thick] (-2.5,1) -- (-2.5,-1);
	\draw[very thick] (0,1) -- (2,-1);
	\draw[very thick] (0.5,1) -- (2.5,-1);
	\draw[very thick] (2,1) -- (1.35,0.35);
	\draw[very thick] (2.5,1) -- (1.6,0.1);
	\draw[very thick] (0.5,-1) -- (1.15,-0.35);
	\draw[very thick] (0,-1) -- (0.9,-0.1);
	\draw[very thick] (4.5,1) -- (4.5,-1);
	\draw[very thick] (5,1) -- (5,-1);
	\draw (-1,0) node{$\cdots$};
	\draw (3.5,0) node{$\cdots$};
	\end{tikzpicture} 
	\]
	If the crossing is a stay of $p$, we draw the following picture we call ``the intersection'' instead:
	
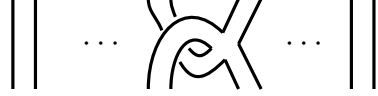
\[
	\begin{tikzpicture}[scale=0.6,baseline=-2pt]
	\draw[very thick] (-3,1) -- (-3,-1);
	\draw[very thick] (-2.5,1) -- (-2.5,-1);
	\draw[very thick] (1.7,-0.35) -- (2,-1);
	\draw[very thick] (2,0) -- (2.5,-1);
	\draw[very thick] (2,1) -- (1.65,0.25);
	\draw[very thick] (2.5,1) -- (2,0);
	\draw[very thick] (0.5,-1) to [out=90,in=135,looseness=1.5] (1.4,-0.1);
	\draw[very thick] (0,-1) to [out=90,in=135,looseness=1.5] (1.65,0.25);
	\draw[very thick] (4.5,1) -- (4.5,-1);
	\draw[very thick] (5,1) -- (5,-1);
	\draw[very thick] (1.7,-0.35) to [out=-135,in=-60,looseness=1.5] (0.7,-0.4);
	\draw[very thick] (1.4,-0.1) to [out=-135,in=-50,looseness=1.5] (0.9,-0.1);
	\draw[very thick] (0.3,0.3) to [out=110,in=-90,looseness=1] (0,1);
	\draw[very thick] (0.6,0.6) to [out=110,in=-90,looseness=1] (0.5,1);
	\draw (-1,0) node{$\cdots$};
	\draw (3.5,0) node{$\cdots$};
	\end{tikzpicture} 
	\]
	We stack these pictures on top of each another, and then place half-disks on the very top as follows:
		\[
	\begin{tikzpicture}[scale=1,baseline=-2pt]
	\draw[very thick] (-3,0) to [out=90,in=180] (-2.75,0.25) to [out=0,in=90] (-2.5,0);
	\draw[very thick] (-1,0) to [out=90,in=180] (-.75,0.25)  to [out=0,in=90] (-.5,0);
	\draw[very thick] (1.5,0) to [out=90,in=180] (1.75,0.25) to [out=0,in=90] (2,0);
	\draw[very thick] (-3,0) to [out=90,in=90,looseness=1.7] (-2.5,0);
	\draw[very thick] (-1,0) to [out=90,in=90,looseness=1.7] (-0.5,0);
	\draw[very thick] (1.5,0) to [out=90,in=90,looseness=1.7] (2,0);
	\draw (0.5,0) node{$\cdots$};
\draw(-2.75,0.25) node[circle,inner sep=2pt,fill=white,draw]{};
\draw(-.75,0.25) node[circle,inner sep=2pt,fill=white,draw]{};
\draw(1.75,0.25) node[circle,inner sep=2pt,fill=white,draw]{};
	\end{tikzpicture} 
	\]
	
	The main observation goes as follows:
	\begin{prop}
		If we go from bottom to top starting at half-disk $i$ always following the left side of the surface, after $k$ steps we are at the horizontal position $\pi_k(i)$. If we follow the right side instead, we are at $p_k(i)$. 
	\end{prop}
	The picture below shows how two paths following the left side look like around the intersection followed by an analogous picture for paths following the right side.
	
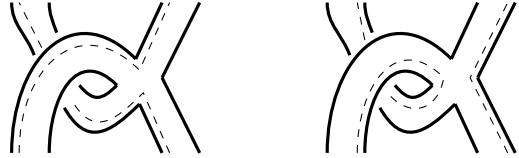
\[
	\begin{tikzpicture}[scale=1,baseline=-2pt]
	\draw[very thick] (1.7,-0.35) -- (2,-1);
	\draw[very thick] (2,0) -- (2.5,-1);
	\draw[very thick] (2,1) -- (1.65,0.25);
	\draw[very thick] (2.5,1) -- (2,0);
	\draw[very thick] (0.5,-1) to [out=90,in=135,looseness=1.5] (1.4,-0.1);
	\draw[very thick] (0,-1) to [out=90,in=135,looseness=1.5] (1.65,0.25);
	\draw[very thick] (1.7,-0.35) to [out=-135,in=-60,looseness=1.5] (0.7,-0.4);
	\draw[very thick] (1.4,-0.1) to [out=-135,in=-50,looseness=1.5] (0.9,-0.1);
	\draw[very thick] (0.3,0.3) to [out=110,in=-90,looseness=1] (0,1);
	\draw[very thick] (0.6,0.6) to [out=110,in=-90,looseness=1] (0.5,1);
	\draw[dashed] (0.1,-1) to [out=90,in=135,looseness=1.5] (1.7,0.1)
	to (2.1,1);
	\draw[dashed] (2.1,-1) to (1.75,-0.2)  to [out=-135,in=-60,looseness=1.5] (0.8,-0.3);
	\draw[dashed] (0.4,0.4) to (0.1,1);
	\end{tikzpicture} 
	\qquad\qquad
	\begin{tikzpicture}[scale=1,baseline=-2pt]
	\draw[very thick] (1.7,-0.35) -- (2,-1);
	\draw[very thick] (2,0) -- (2.5,-1);
	\draw[very thick] (2,1) -- (1.65,0.25);
	\draw[very thick] (2.5,1) -- (2,0);
	\draw[very thick] (0.5,-1) to [out=90,in=135,looseness=1.5] (1.4,-0.1);
	\draw[very thick] (0,-1) to [out=90,in=135,looseness=1.5] (1.65,0.25);
	\draw[very thick] (1.7,-0.35) to [out=-135,in=-60,looseness=1.5] (0.7,-0.4);
	\draw[very thick] (1.4,-0.1) to [out=-135,in=-50,looseness=1.5] (0.9,-0.1);
	\draw[very thick] (0.3,0.3) to [out=110,in=-90,looseness=1] (0,1);
	\draw[very thick] (0.6,0.6) to [out=110,in=-90,looseness=1] (0.5,1);
	\draw[dashed] (0.4,-1) to [out=90,in=135,looseness=1.5] (1.55,0)
	to [out=-100,in=-50,looseness=1.5] (0.8,-0.2);
	\draw[dashed] (2.4,-1) to (1.9,0) to (2.4,1);
	\draw[dashed] (0.5,0.5) to (0.4,1);
	\end{tikzpicture} 
	\]
	
	It is not hard to see that the surface we obtain by this construction is homeomorphic to the surface obtained by gluing the surfaces of Example \ref{ex:surface for s}.
	
	The surface without boundary $\bar\Sigma$ is obtained from $\Sigma$ by attaching a disk to each boundary component. To summarize, we have:
	\begin{cor}
		The cocharacter lattice of the torus component of the cell of $Y_\beta(t)$ corresponding to the walk $p$ is isomorphic to $H_1(\overline\Sigma)$, and the pairing induced by the $2$-form $\omega$ on the cocharacter lattice is the intersection form on $H_1(\overline\Sigma)$.
	\end{cor}
	
	\subsection{Connected case}
	The case when the surface $\Sigma$ is connected is important:
	\begin{prop}\label{prop:connected surface}
		For a braid $\beta$ and a walk $p$ the following conditions are equivalent:
		\begin{enumerate}
			\item The surface $\Sigma$ constructed in Section \ref{ssec:surface from walk} is connected.
			\item The stabilizer of each point in the corresponding cell $C_p$ is $\BC^*$.
			\item The image of $C_p\to T$ is the full subtorus $\{t\in T:\det(t)=\sign(\pi)\}$.
		\end{enumerate}
	\end{prop}
	\begin{proof}
		Recall that $C_p=\BC^{U_p}\times (\BC^*)^{S_p}$ and the map to $T$ factors through $(\BC^*)^{S_p}$. On the cocharacter lattice, this map is described by $\varphi$, which is the boundary map $H^1(\Sigma,A)\to H_0(A)$. The image of this map is the kernel of the surjective map $H_0(A)\to H_0(\Sigma)$ by \eqref{eq:first exact sequence}. So the corank of the image equals to the number of connected components. In particular, if the surface is connected, the image is given by vectors $(x_1,\ldots,x_n)$ satisfying $\sum_{i=1}^n x_i=0$. This proves equivalence of (i) and (iii).
		
		To analyze the stabilizers, it is enough to consider the subset $\{0\}\times (\BC^*)^{S_p}$ which is fixed by the $T$-action. The cocharacter lattice of the stabilizer is the kernel of $\psi:H_0(A) \to H_1(\Sigma,A)$. Recall that $\psi(a)$ is the path that connects $a$ to the next point in $A$ by moving along the boundary. In particular, the kernel of $\psi$ is contained in the lattice of $\pi$-invariant vectors, which is mapped to closed paths. So the kernel of $\psi$ is the kernel of the map $H_1(\partial\Sigma)\to H_1(\Sigma)$, which is isomorphic to $H_2(\bar\Sigma)$ by \eqref{eq:second exact sequence}. So its rank is $1$ precisely when the surface is connected. In the connected case, the kernel is spanned by the sum of all the boundary components, which is described by vector $(1,1,\ldots,1)$. This proves equivalence of (i) and (ii).
	\end{proof}

It is clear that for generic values of $t$, the preimage $g^{-1}(t)$ only intersects cells whose surfaces are connected. More precisely, we have
\begin{defn}\label{def:braid generic}
	A point $(t_1,\ldots,t_n)\in T$ is said to be \emph{generic} if we have
	\[
	\prod_{i=1}^n t_i = \sign(\pi),
	\]
	but for any proper $\pi$-invariant subset $S\subset \{1,\ldots,n\}$ we have
	\[
	\prod_{i\in S} t_i \neq \sign(\pi|_S).
	\]
\end{defn}

Then we obtain
\begin{prop}\label{prop:generic implies connected}
	If $t\in T$ is generic, then $g^{-1}(t)$ intersects only cells whose surfaces are connected.
\end{prop}
\begin{proof}
	If $t\in T$ intersects a cell $C_p$, then $t$ is contained in the subtorus $g(C_p)\subset T$. The cocharacter lattice of $g(C_p)$ is the kernel of $H_0(A)\to H_0(\Sigma)$. Suppose $\Sigma$ is not connected, $\Sigma=\Sigma'\sqcup\Sigma''$. This produces a decomposition $A=S\sqcup (A\setminus S)$ and the cocharacter lattice of the subtorus is contained in 
	\[
	\{(x_1,\ldots,x_n):\sum_{i\in S} x_i = 0.
	\]
	So the subtorus is given by condition
	\[
	\prod_{i\in S} t_i = \const,
	\]
	where the constant is $\pm 1$. Our construction of TSV from surface would produce $1$, but each $s_{i,i'}$ changes the sign (see Example \ref{ex:surface for s}). The total number of sign changes is the number of intersections in $\Sigma'$ (see Section \ref{ssec:surface from walk}). The parity of the number of intersections up to step $i$ is given by the parity of the permutation $p_i^{-1} \pi_i$ restricted to $S$. In the end we have $p_m=\Id$, so the parity is the parity of $\pi|_S$.
\end{proof}

\begin{prop}\label{prop:Y quotient}
	For a generic point $t\in T$, the affine GIT quotient 
	\[
	Y_\beta^{\GIT}(t)=g^{-1}(t)\sslash (T^\pi/\BC^*)
	\]
	is a geometric quotient and the projection $g^{-1}(t) \to Y_\beta^{\GIT}(t)$ is a Zariski locally trivial fibration. Hence the stack $Y_\beta(t)$ is represented by affine variety $Y_\beta^{\GIT}(t)$.
\end{prop}
\begin{proof}
	By Propositions \ref{prop:generic implies connected} and \ref{prop:connected surface}, the stabilizers are trivial. This implies that all orbits of $T^w/\BC^*$ acting on $g^{-1}(t)$ have the same dimension, and therefore all orbits are closed. Each fiber of the projection $g^{-1}(t) \to Y_\beta^{\GIT}(t)$ contains a unique closed orbit (see \cite{brion2010introduction}), therefore the fibers are precisely the orbits, which means that the quotient is geometric.
	
	Finally, we show local triviality. Pick a point $x\in g^{-1}(t)$. Recall that $g^{-1}(t)$ is a closed subset of $X_0$, which is a closed subset of $\BC^l$ on which $T$ acts linearly so that $T$ acts on each coordinate through a character of the form $\frac{t_j}{t_{j'}}$. Because the stabilizer is trivial, we can choose a subset of coordinates $I\subset \{1,\ldots,l\}$ such that for any $k\in I$ we have $x_k\neq 0$ and the induced action of $T^\pi/\BC^*$ on $(\BC^*)^I$ is free and transitive. Let
	\[
	U=\{x\in g^{-1}(t): x_k\neq 0\quad \text{for all $k\in I$}\}.
	\]
	Then $U$ is a $T^\pi/\BC^*$-invariant Zariski open neighborhood of $x$ such that the projection morphism $U \to U\sslash(T^\pi/\BC^*)$ is a trivial fibration.
\end{proof}

Finally, we obtain
\begin{thm}\label{thm:lefschetz for Y}
	For a generic point $t\in T$, the class of $\omega$ in $H^2(Y_\beta(t))$ satisfies curious Lefschetz property with middle weight $l(\beta) - n - c(\beta)+2$.
\end{thm}
\begin{proof}
By Proposition \ref{prop:weak cell lefschetz}, it is enough to check the statement for each cell in the cell decomposition of $Y_\beta(t)$. Each cell looks like a product of a torus and affine space, $\BC^{U_p}$. By Corollary \ref{cor:lefschetz for surface}, curious Lefschetz holds for the torus part with middle weight equal to $\dim H_1(\bar\Sigma)$. By \eqref{eq:first exact sequence}, the Euler characteristic of $\Sigma$ is
\[
\chi(\Sigma) = \dim H_0(A) - \dim H_1(\Sigma,A) = n - |S_p|.
\]
Since $\bar\Sigma$ is obtained by attaching $c(\pi)$ disks, we have
\[
\chi(\bar\Sigma) = n - |S_p| + c(\pi) \quad\Rightarrow\quad \dim H_1(\bar\Sigma) = 2\dim H_0(\bar\Sigma)-n +|S_p|-c(\pi). 
\]
Since the surface is connected, we obtain that the middle weight equals 
\[
2-n+|S_p|-c(\pi).
\]
Multiplying by the affine space increases the weight by $2 |U_p|$, and from $|S_p|+2|U_p|=l(\beta)$ we obtain that all cells have curious Lefschetz with same weight middle weight $2-n+l(\beta)-c(\pi)$.
\end{proof}
	
\section{Decomposing character varieties}\label{sec:decomposing}
	\subsection{One puncture of maximal ramification}
	Recall the notations of Section \ref{ssec:parabolic char stack}: $G=GL_n$, $B\subset G$ consists of upper-triangular matrices, $U\subset B$ consists of unipotent matrices and $T\subset B$ consists of diagonal matrices. $W$ is the permutation group on $n$ elements. $C_1,\ldots,C_k\in T$ are assumed to satisfy the genericity assumption (Definition \ref{def:generic}), and $g\geq 0$. The framed character variety $\CX_\parab$ is the variety consisting of $k+2g$-tuples of elements 
	\[
	\alpha_1,\ldots,\alpha_g, \beta_1,\ldots,\beta_g, \gamma_1,\ldots,\gamma_k \in G
	\]
	satisfying
	\[
	[\alpha_1,\beta_1]\cdots [\alpha_g, \beta_g] \gamma_1^{-1} C_1 \gamma_1\cdots\gamma_k^{-1} C_k \gamma_k = \Id.
	\]
	The gauge group is
	\[
	G_\parab = G \times Z(C_1) \times \cdots Z(C_k),
	\]
	which acts on $\CX_\parab$ so that $G$ conjugates $\alpha$-s and $\beta$-s, and multiplies $\gamma$-s on the right, while $Z(C_i)$ multiplies $\gamma_i$ on the left. We have the categorical quotient
	\[
	X_\parab = \CX_\parab/G_\parab,
	\]
	and $\CX_\parab\to X_\parab$ is  a principal $G_\parab/\BC^*$-bundle. 

	As a first step, we gauge out the action of $G\subset G_\parab/\BC^*$ using the fact that $G$ acts freely on $\gamma_k$. Let $\CX^0_\parab\subset \CX_\parab$ be defined by the condition $\gamma_k=\Id$. Then $\CX^0_\parab \to X_\parab$ is a principal bundle with structure group
	\[
	G_\parab^0:=(G_\parab/\BC^*)/G = (Z(C_1) \times \cdots \times Z(C_k))/\BC^*.
	\]
	The variety $\CX^0_\parab$ is described by the matrix equation
	\begin{equation}\label{eq:product1}
		[\alpha_1,\beta_1]\cdots [\alpha_g, \beta_g] \gamma_1^{-1} C_1 \gamma_1\cdots\gamma_{k-1}^{-1} C_{k-1} \gamma_{k-1} \;C_k = \Id.
	\end{equation}

	We proceed under the following assumption:
	\begin{assumption*}
		Suppose $C_k$ has distinct eigenvalues $a_1,\ldots,a_n\in\BC^*$.
	\end{assumption*}
	This implies that $Z(C_k)=T$. Further assume that the eigenvalues of all the $C_i$ are ordered nicely:
	\begin{defn}\label{defn:ordered}
		We say that a diagonal matrix $C_i$ is \emph{ordered nicely} if its list of diagonal entries $x_1,\ldots,x_n$ satisfies
		\[
		x_i=x_j \;\Rightarrow\; x_i=x_{i+1}=\cdots=x_j\;\text{for any $i<j$.}
		\]
	\end{defn}
	This corresponds to the condition that $Z(C_i)$ is formed by block-diagonal matrices. 
	
	\subsection{Covering by a vector bundle}
	As a second step, we pass to $\CX_\parab^1:=\CX_\parab^0\times U$. The group $G^0_\parab$ acts on $U$ through the conjugation action of $T/\BC^*$. Form the quotient
	\[
	\tilde X_\parab = \CX_\parab^1/G^0_\parab.
	\]
	The map $\tilde X_\parab \to X_\parab$ is a fiber bundle with fiber $U$. In particular, it is a homotopy equivalence.
	
	As a third step, we change the coordinates. Let $u$ be the coordinate on $U$. We conjugate $\alpha$-s and $\beta$-s by $u$ and multiply $\gamma$-s by $u^{-1}$ on the right. This changes the matrix equation \eqref{eq:product1} to 
	\begin{equation}\label{eq:product2}
	[\alpha_1,\beta_1]\cdots [\alpha_g, \beta_g] \gamma_1^{-1} C_1 \gamma_1\cdots\gamma_{k-1}^{-1} C_{k-1} \gamma_{k-1} \; u C_k u^{-1}=\Id.
	\end{equation}
	Let $u'= u C_k u^{-1} C_k^{-1}$. Since $C_k$ has distinct eigenvalues, the map $u\to u'$ is invertible on $U$, so we will use $u'$ instead of $u$. The new matrix equation is (we drop the $'$ from $u'$)
	\begin{equation}\label{eq:product3}
	[\alpha_1,\beta_1]\cdots [\alpha_g, \beta_g] \gamma_1^{-1} C_1 \gamma_1\cdots\gamma_{k-1}^{-1} C_{k-1} \gamma_{k-1} \;u = C_k^{-1}.
	\end{equation}
	We can view the above product as the convolution of equivariant TSVs (Definitions \ref{def:TSV}, \ref{def:equivariant TSV}). One type of TSVs comes from the factors $\gamma_i^{-1} C_i \gamma_i$, which correspond to the punctures, another one from  $[\alpha_i,\beta_i]$, which are the genus contributions. The variable $u$ corresponds to the TSV which we denote by $U$, corresponding to $U\subset G$ with the natural inclusion and zero form. Bruhat decompositions will produce a stratification of each TSV, and we will obtain a stratification of $\tilde X_\parab$ whose strata are related to the braid varieties.
	
	\subsection{Puncture contributions}\label{ssec:puncture contributions}
	Let $C$ be a diagonal matrix with ordered entries (Definition \ref{defn:ordered}). The centralizer $Z(C)$ is the group of block-diagonal matrices of the corresponding shape.  We consider the TSV 
	\[
	G/Z(C),\; x\to x C x^{-1},\;(x|C|x^{-1}).
	\]
	It is equivariant for the identity permutation.
	
	Let $P=Z(C) B$ be the parabolic subgroup containing $Z(C)$.
	We use the Bruhat decomposition of $B\backslash G /P$. 	  
The cells are indexed by elements of $W/W_C$ where $W_C=Z(C)\cap W$. View $W/W_C$ as a subset of $W$ by representing each coset by the shortest permutation. For a fixed $\pi\in W/W_C$, the cell can be parametrized by 
\[
x = v \pi p\quad (p\in P, \; v\in U_{\pi^{-1}}^-).
\]
Notice that the group $Z(C)$ acts only on $p$ and can be gauged out by restricting to the subset where $p$ is in the unipotent radical of $P$, denoted by $N$. So the original TSV is stratified by the TSVs of the form
\[
U^-_{\pi^{-1}}\times N,\; (v,p)\to v\pi p C p^{-1} \pi^{-1} v^{-1},\; (v\pi p|C|p^{-1} \pi^{-1} v^{-1}).
\]
Using Proposition \ref{prop:form vanishes}, we rewrite the form as 
\[
(v|\pi |p |C |p^{-1}| \pi^{-1} |v^{-1}).
\]
It is convenient to describe TSV by a single expression, for example $v\pi p C p^{-1} \pi^{-1} v^{-1}$, provided we know the domain of each variable, and we follow the convention that the form is obtained by inserting $|$ between every two symbols.
 Let $p'=C^{-1} p C p^{-1}$. The map $p\to p'$ is an algebraic automorphism of $N$, so we can use $p'$ as a coordinate instead of $p$. After performing the substitution, we drop $'$ from $p'$. So we obtain
 \[
 v\pi C p \pi^{-1} v^{-1}.
 \]
Next we study what happens when we take convolution with $U$. The product looks like
 \[
 v\pi C p \pi^{-1} v^{-1} u.
 \]
Note that $\pi$ is the shortest representative of a coset of $W/W_C$. This means that the inversions of $\pi$ can only appear between indices in different $W_C$-orbits. Therefore we have $N\supset U^-_{\pi}$ and there is a decomposition 
\[
N =U^-_{\pi} (U^+_{\pi}\cap N).
\]
So we can decompose $p$ as a product $p^- p^+$ and pass $p^+$ to the other side of the permutation:
\[
v \pi C p^- \pi^{-1} (\pi p^+ \pi^{-1}) v^{-1} u.
\]
with $\pi p^+ \pi^{-1}\in U$. We make a change of variables $u'=\pi p^+ \pi^{-1} v u$ and obtain\footnote{we keep dropping $'$ from $u'$ after each substitution}
\[
v \pi C p^- \pi^{-1} u.
\]
Note that the variable $p_i^+$ does not participate in the equation anymore. Now we notice that $p^- \pi^{-1} u$ can be reparametrized using \eqref{eq:bruhat} and Proposition \ref{prop:positive lift} to arrive at
\[
v \pi C u f_{\pi^{-1}}(z),
\]
where $f_{\pi^{-1}}$ comes from the TSV associated to the positive lift of the permutation $\pi^{-1}$, and $z\in\BC^{l(\pi)}$ is the coordinate. Performing the substitution $u'=C u C^{-1}$ we obtain
\[
v \pi u C f_{\pi^{-1}}(z).
\]
Remember that $v$ is a coordinate on $U_{\pi^{-1}}^-$, so we can reparametrize again as follows:
\[
u f_{\pi}(z') C f_{\pi^{-1}}(z).
\]
We present our result symbolically as follows:
\begin{equation}\label{eq:puncture contribution}
x C x^{-1} u \;\sim\; \bigsqcup_{\pi\in W/W_C} u f_{\pi}(z') C f_{\pi^{-1}}(z) \times (U^+_{\pi}\cap N).
\end{equation}
It is clear that $u$ can be moved to the next factor of \eqref{eq:product3} and so on. In the case of genus zero, we obtain that $\CX^1_\parab$ is stratified by cells corresponding to tuples $\bar\pi=(\pi_1,\ldots,\pi_k)$ ($\pi_i\in W/W_{C_i}$). Each cell has equation of the form
\[
u \prod_{i=1}^{k-1} f_{\pi_i}(z_i') C_i f_{\pi_i^{-1}}(z_i) = C_k^{-1}.
\]
We can move all $C_i$ to the right because of equivariance, to arrive at an equivalent equation
\[
\prod_{i=1}^{k-1} f_{\pi_i}(z_i') f_{\pi_i^{-1}}(z_i) = u C_{\bar\pi}\quad (
C_{\bar\pi}=C_k^{-1} \prod_{i=1}^{k-1} \pi_i(C_i^{-1})).
\]
Note that the genericity assumption of Definition \ref{def:generic} for $C_i$ implies the genericity assumption of Definition \ref{def:braid generic} for $C_{\bar\pi}$.

\begin{thm}\label{thm:stratification genus 0}
	Suppose $g=0$, $k\geq 1$, $C_1,\ldots,C_k$ are diagonal matrices that are ordered nicely (Definition \ref{defn:ordered}) satisfying the genericity assumption (Definition \ref{def:generic}), and suppose $C_k$ has distinct eigenvalues. Let $X_\parab$ be the corresponding character variety. Denote by $W_i\subset W$ the stabilizer of $C_i$. Then $X_\parab$ carries a vector bundle $\epsilon:\tilde X_\parab\to X_\parab$ of rank $\binom{n}{2}$ and a stratification $X_\parab=\bigsqcup_{\bar\pi} X_{\bar\pi}$ indexed by tuples $\bar\pi=(\pi_1,\ldots,\pi_{k-1})$ with $\pi_i\in W/W_i$ in such a way that for any $\bar\pi$ we have the following diagram:
	\[
	\begin{tikzcd}
	& \tilde X_{\bar\pi}=\epsilon^{-1}(X_{\bar\pi}) \arrow{ld}{\epsilon} \arrow{rd} &\\
	X_{\bar\pi} & & Y_{b_{\bar\pi}}(C_{\bar\pi}),
	\end{tikzcd}
	\]
	where the arrow on the right is a vector bundle of rank $r_{\bar\pi}$, and the braid $b_{\bar\pi}$ is given by the product
	\[
	b_{\bar\pi} = \sigma_{\pi_1} \sigma_{\pi_1^{-1}} \cdots \sigma_{\pi_{k-1}} \sigma_{\pi_{k-1}^{-1}},
	\]
	where $\sigma_{\pi_i}$ is the positive lift of the shortest permutation representing $\pi_i$. The rank is given by
	\[
	r_{\bar\pi} = \sum_{i=1}^{k-1} (l_i - l(\pi_i)),
	\]
	where $l_i=\dim N_i$, $N_i$ is the radical of $Z(C_i)B$. The pullbacks of the $2$-forms from $X_{\bar\pi}$ and $Y_{b_{\bar\pi}}(C_{\bar\pi})$ to $\tilde X_{\bar\pi}$ agree.
\end{thm}
The compactly supported cohomology does not change when we pass to a vector bundle $r$, except that the weights increase by $2r$. So by Theorem \ref{thm:lefschetz for Y}, we have curious Lefschetz on $X_{\bar\pi}$ of middle weight
\[
l(b_{\bar\pi}) - n - c(b_{\bar\pi}) +2+ 2\left(\sum_{i=1}^{k-1} (l_i - l(\pi_i))\right) - 2\binom{n}{2}.
\]
We have $c(b_{\bar\pi})=n$, $l(b_{\bar\pi})=2\sum_{i=1}^{k-1} l(\pi_i)$ and $2 l_i+ \dim Z(C_i)=n^2$. So the weight equals
\[
(k-2) n^2 + 2 - \sum_{i=1}^k \dim Z(C_i),
\]
which is precisely the dimension of the character variety. In particular, it does not depend on $\bar\pi$ and by Proposition \ref{prop:weak cell lefschetz} we obtain
\begin{cor}
	In the case of genus zero when $C_k$ has distinct eigenvalues, the form $\omega$ on the character variety satisfies curious Lefschetz with middle weight equal to the dimension.
\end{cor}

\subsection{Genus contributions}\label{ssec:genus contributions}
We use notation conventions as above. We start with an equivariant TSV
\[
\alpha \beta \alpha^{-1} \beta^{-1}.
\]
Use Bruhat decompositions for $\alpha$ and $\beta$ as follows:
\[
\alpha = v_1 \pi_1 t_1 u_1\quad (v_1\in U_{(\pi_1)^{-1}}^-, \; t_1\in T, \; u_1\in U),
\]
\[
\beta = u_2 \pi_2 t_2 v_2\quad (v_2\in U_{\pi_2}^-, \; t_2\in T, \; u_2\in U).
\]
Using Proposition \ref{prop:form vanishes} we check that the form does not change if we insert $|$ between every two symbols. So each cell is described by
\[
v_1 \pi_1 t_1 u_1 u_2 \pi_2 t_2 v_2 u_1^{-1} t_1^{-1} \pi_1^{-1} v_1^{-1} v_2^{-1} t_2^{-1} \pi_2^{-1} u_2^{-1}.
\]
After multiplying by $u$ on the right we see that $u_2^{-1}$ can be absorbed into $u$ to obtain
\[
v_1 \pi_1 t_1 u_1 u_2 \pi_2 t_2 v_2 u_1^{-1} t_1^{-1} \pi_1^{-1} v_1^{-1} v_2^{-1} t_2^{-1} \pi_2^{-1} u.
\]
Now $u_2$ appears only once, so we use it to absorb $u_1$, which will appear only once after that, so that it can absorb $v_2$:
\[
v_1 \pi_1 t_1 u_2 \pi_2 t_2  u_1^{-1} t_1^{-1} \pi_1^{-1} v_1^{-1} v_2^{-1} t_2^{-1} \pi_2^{-1} u.
\] 
Now $v_2$ appears only once, namely in the parametrization of the Bruhat cell
\[
v_2^{-1} t_2^{-1} \pi_2^{-1} u.
\]
Reparametrizing it we obtain
\[
v_1 \pi_1 t_1 u_1 u_2 \pi_2 t_2 v_2 u_1^{-1} t_1^{-1} \pi_1^{-1} v_1^{-1} u t_2^{-1} f_{\pi_2^{-1}}(z_2).
\]
Now $u$ can absorb $v_1^{-1}$ and here is all that remains:
\[
v_1 \pi_1 t_1 u_1 u_2 \pi_2 t_2 v_2 u_1^{-1} t_1^{-1} \pi_1^{-1} u t_2^{-1} f_{\pi_2^{-1}}(z_2),
\]
and all the non-torus variables appear only once. Decompose $u_1 = u_1^+ u_1^-$ with $u_1^{\pm}\in U_{\pi_1}^\pm$. Then $(u_i^{+})^{-1}$ can be moved past the permutation and absorbed by $u$.  Then we reparametrize and obtain
\[
v_1 \pi_1 t_1 u_2 \pi_2 t_2  u t_1^{-1} f_{\pi_1^{-1}}(z_1) \;t_2^{-1} f_{\pi_2^{-1}}(z_2).
\]
We proceed in a similar fashion decomposing $u_2 = u_2^- u_2^+$, absorbing $u_2^{+}$ and reparametrizing. So we are left with
\[
v_1 \pi_1 t_1 u \cdots.
\]
which can be reparametrized as 
\[
u f_{\pi_1}(z_1') t_1 \cdots,
\]
and the total product looks like
\begin{equation}\label{eq:product uft}
u f_{\pi_1}(z_1') t_1 f_{\pi_2}(z_2') t_2 t_1^{-1} f_{\pi_1^{-1}}(z_1) t_2^{-1} f_{\pi_2^{-1}}(z_2).
\end{equation}
Using the equivariance property of $f_\pi(z)$ (Corollary \ref{cor:equivariance}), we can group all the $t$-symbols together, for instance we have
\[
(f_{\pi_1^{-1}}(z_1)|t_2^{-1}) = ({}^{\pi_1^{-1}} t_2^{-1}| {}^{\pi_1^{-1}} t_2 | f_{\pi_1^{-1}}(z_1)|t_2^{-1}) = ({}^{\pi_1^{-1}} t_2^{-1}|f_{\pi_1^{-1}}(t_2(z_1))), 
\]
where $t_2(z_1)$ comes from the action of $T$ on $z_1$, so the change of variables $z_1\to t_2^{-1}(z_1)$ produces $({}^{\pi_1^{-1}} t_2^{-1}|f_{\pi_1^{-1}}(z_1))$. So we rearrange the product \eqref{eq:product uft} as follows:
\begin{equation}\label{eq:product uft2}
u f_{\pi_1}(z_1') f_{\pi_2}(z_2') \; {}^{\pi_2^{-1}} t_1 \;t_2 \; t_1^{-1}\;\; {}^{\pi_1^{-1}} t_2^{-1}\;\;  f_{\pi_1^{-1}}(z_1)  f_{\pi_2^{-1}}(z_2).
\end{equation}
The product ${}^{\pi_2^{-1}} t_1 \;t_2 \; t_1^{-1}\;\; {}^{\pi_1^{-1}} t_2^{-1}$ is equivariant with respect to the permutation $\pi_2^{-1} \pi_1^{-1} \pi_2\pi_1$, so the complete product \eqref{eq:product uft2} is equivariant for the identity permutation. As a result, we obtain a stratification
\[
\alpha \beta \alpha^{-1} \beta^{-1} u \sim \bigsqcup_{\pi_1,\pi_2\in W} u f_{\pi_1}(z_1') f_{\pi_2}(z_2')\; \;{}^{\pi_2^{-1}} t_1 \;t_2 \;t_1^{-1}\;\; {}^{\pi_1^{-1}} t_2^{-1} f_{\pi_1^{-1}}(z_1)  f_{\pi_2^{-1}}(z_2)\;\times U_{\pi_1}^+\times U_{\pi_2}^+.
\]

It turns out, that the TSV corresponding to the product ${}^{\pi_2^{-1}} t_1 \;t_2 \; t_1^{-1}\;\; {}^{\pi_1^{-1}} t_2^{-1}$ comes form a surface in the sense of Section \ref{ssec:toric TSVs from surfaces}. To construct the surface, consider the punctured torus
	
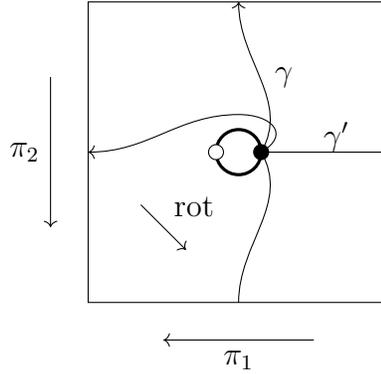
\[
	\begin{tikzpicture}
	\draw (0,0) rectangle (4,4);
	\draw[very thick] (2,2) circle (0.3);
	\draw[<-] (1,-0.5) -- (3,-0.5) node[midway,below] {$\pi_1$};
	\draw[<-] (-0.5,1) -- (-0.5,3) node[midway,left] {$\pi_2$};
	\draw (2,0) to [out=90,in=-60] (2.3,2);
	\draw[<-] (2,4) to [out=-90,in=60] (2.3,2);
	\draw (2.6,3) node {$\gamma$};
	\draw (3.3,2.2) node {$\gamma'$};
	\draw[->] (2.3,2) to [out=30,in=0,looseness=2] (2,2.5) to [out=180,in=0] (0,2);
	\draw (4,2) to (2.3,2);
	\draw (2.3,2) node[draw,circle,inner sep=2pt,fill] {};
	\draw (1.7,2) node[circle,inner sep=2pt,fill=white,draw]{};
\draw[->] (0.7,1.3) to node[midway,above right] {$\rot$} (1.3,0.7);
	\end{tikzpicture}
	\]
The torus is obtained by gluing the opposite sides of the rectangle above. The puncture is in the middle. We use $\pi_1$ and $\pi_2$ to construct an unramified covering of the punctured torus so that the monodromy along the paths $\gamma$, $\gamma'$ are given by $\pi_2^{-1}$, $\pi_1$ respectively. The preimages of the black dot are numbered by $\{1,\ldots,n\}$. The pullbacks of $\gamma,\gamma'$ are denoted by $\gamma_i$,$\gamma_i'$ in such a way that we have
	\[
	\gamma_i:i\to \pi_2^{-1}(i),\quad \gamma_i': \pi_1^{-1}(i)\to i.
	\]
This produces a labeled marked surface $(\Sigma,A,B)$.
\begin{prop}
	The TSV $(X,f,\omega)$ associated by Section \ref{ssec:toric TSVs from surfaces} to $(\Sigma,A,B)$ is described as follows. The paths $\gamma_i$, $\gamma_i'$ for $i=1,\ldots,n$ form a basis of $H_1(\Sigma,A)$, which identifies $H_1(\Sigma,A)$ with $\BZ^n\times\BZ^n$. This identifies $X$ with $T\times T$, with coordinates $t_1\in T$, $t_2\in T$. The map $f$ and the form $\omega$ then come from the product ${}^{\pi_2^{-1}} t_1 \;t_2 \; t_1^{-1}\;\; {}^{\pi_1^{-1}} t_2^{-1}$.
\end{prop}
\begin{proof}
	Let $\bar\gamma_i$, $\bar\gamma_i'$ denote the paths defined analogously to the paths $\gamma_i$, $\gamma_i'$, but with the picture reflected centrally. We can label them in such a way that 
	\[
	\gamma_i\cdot\bar\gamma_j=0,\;\gamma_i'\cdot\bar\gamma_j'=0,\;\gamma_i\cdot\bar\gamma_j'=\pm\delta_{i,j},\;\gamma_i'\cdot\bar\gamma_j=\pm\delta_{i,j}.
	\]
	In particular, it follows that the classes of $\gamma_i$ and $\gamma_i'$ in $\Lambda=H_1(\Sigma, A)$ are linearly independent and span a primitive sublattice of $\Lambda$. On the other hand, by \eqref{eq:first exact sequence} the rank of $\Lambda$ is equal to $2n$ ($\chi(\Sigma)=-n$). So we conclude that $\gamma_1,\ldots,\gamma_n,\gamma_1',\ldots,\gamma_n'$ form a basis of $\Lambda$.

	The homomorphism $\varphi$ is given by the boundary map $\partial:\Lambda\to\BZ^n$, which we can explicitly write as follows:
	\[
	\varphi = (\pi_2^{-1}-1,1-\pi_1^{-1}).
	\]
	This precisely matches the product ${}^{\pi_2^{-1}} t_1 \;t_2 \; t_1^{-1}\;\; {}^{\pi_1^{-1}} t_2^{-1}$.

	We make sure that following the paths according to  $\pi_2^{-1}\pi_1^{-1}\pi_2\pi_1$ amounts to going around the boundary component counter-clockwise, so we have
	\[
	\pi=\pi_2^{-1}\pi_1^{-1}\pi_2\pi_1.
	\]
	
	To find $\omega$, we compute
	\[
	\rot(\gamma_i)\cdot\gamma_j= \delta_{i,j}-\delta_{\pi_2^{-1}(i),j},\quad
	\rot(\gamma_i')\cdot\gamma_j'=\delta_{i,j}-\delta_{i,\pi_1^{-1}(j)}.
	\]
	\[
	\rot(\gamma_i') \cdot \gamma_j = -\delta_{i,j},\quad
	\rot(\gamma_i)\cdot \gamma_j'  = -\delta_{\pi_2^{-1}(i),\pi_1^{-1}(j)}+\delta_{i, \pi_1^{-1}(j)}+\delta_{\pi_2^{-1}(i),j}. 
	\]
	The terms $\delta_{i,j}$ in $\rot(\gamma_i)\cdot\gamma_j$ and $\rot(\gamma_i')\cdot\gamma_j'$ disappear after anti-symmetrization, and the remaining six terms match the six terms of the $2$-form
	\[
	\left({}^{\pi_2^{-1}} t_1 \;|\;t_2 \;|\; t_1^{-1}\;\;|\;\; {}^{\pi_1^{-1}} t_2^{-1}\right).
	\]
	
	The operator $\psi$ is given by 
	\[
	\psi = \begin{pmatrix}(\pi_1^{-1}-1)\pi_2\pi_1\\(\pi_2^{-1}-1)\pi_2\pi_1
	\end{pmatrix}.
	\]
	This matches the $T$-action implied by \eqref{eq:product uft2}. Note that $\pi_2\pi_1$ is necessary because $f_{\pi_1^{-1}}(z_1) f_{\pi_2^{-1}}(z_2)$ in \eqref{eq:product uft2} is $\pi_1^{-1} \pi_2^{-1}$-equivariant. As a sanity check, we compute the composition:
	\[
	\varphi\circ\psi =  \left((\pi_2^{-1}-1)(\pi_1^{-1}-1)+(1-\pi_1^{-1})(\pi_2^{-1}-1)\right)\pi_2\pi_1 = \pi_2^{-1}\pi_1^{-1}\pi_2\pi_1-1 = \pi-1,
	\]
	as it should be.
\end{proof}

Applying the procedure described above to each term $[\alpha_i,\beta_i]$, we obtain a stratification of $\CX^1_\parab$ indexed by tuples 
\[
\bar\pi=(\pi_1,\ldots,\pi_k,\pi_1^{1},\pi_2^{1},\ldots,\pi_1^{g}, \pi_2^{g})\quad (\pi_i\in W/W_{C_i},\; \pi_1^i,\pi_2^i\in W).
\]
Denote the cell corresponding to $\bar\pi$ by $\CX_{\bar\pi}$ and its quotient by $T/\BC^*$ by $X_{\bar\pi}$. Then we have 
\[
\CX_{\bar\pi} = \CX_{\bar\pi}'\times \prod_{i=1}^g U_{\pi_1^i}^+ \times U_{\pi_2^i}^+ \times \prod_{i=1}^{k-1} U_{\pi_i}^+\cap N_i,
\]
where $\CX_{\bar\pi}'$ has equation of the form
\[
\text{product of $t$-s and $f$-s} \quad=\quad u C_{\bar\pi}.
\]
The left hand side is understood as the convolution of braid TSVs and torus TSVs. Stratifying the braid TSVs we can ignore the torus TSVs. So $\CX_{\bar\pi}'$ is stratified by cells associated to walks for the braid $b_{\bar\pi}$,
\[
b_{\bar\pi} = \sigma_{\pi_1^1} \sigma_{\pi_2^1} \sigma_{(\pi_1^1)^{-1}} \sigma_{(\pi_2^1)^{-1}}\cdots \sigma_{\pi_1^g} \sigma_{\pi_2^g} \sigma_{(\pi_1^g)^{-1}} \sigma_{(\pi_2^g)^{-1}} \; \sigma_{\pi_1} \sigma_{\pi_1^{-1}} \cdots \sigma_{\pi_{k-1}} \sigma_{\pi_{k-1}^{-1}}.
\]
In the generic case, we can restrict our attention to connected surfaces, which means in particular that there is not stabilizers. Let $X_{\bar\pi}'=\CX_{\bar\pi}'/(T/\BC^*)$. Similarly to Theorem \ref{thm:lefschetz for Y}, taking into account contributions of torus TSVs, which increase the weight by the dimension, we obtain that $X_{\bar\pi}'$ together with the $2$-form satisfies curious Lefschetz of middle weight
\[
l(b_{\bar\pi}) - 2n +2 + 2 g n. 
\]
The space $X_{\bar\pi}$ is a vector bundle over $X_{\bar\pi}'$ whose rank is given by (for some of the notations, see Theorem \ref{thm:stratification genus 0})
\[
\sum_{i=1}^{k-1} \dim (U_{\pi_i}^+\cap N_i) + \sum_{i=1}^g (\dim U_{\pi_1^i}^+ +\dim U_{\pi_2^i}^+) = \sum_{i=1}^{k-1} (l_i-l(\pi_i)) + \sum_{i=1}^g (n(n-1)-l(\pi_1^i)-l(\pi_2^i)).
\]
So $X_{\bar\pi}$ satisfies curious Lefschetz of middle weight
\[
(2g-2)n +2 +\sum_{i=1}^{k-1} 2 l_i + 2 g n (n-1) = 2 g n^2 - 2n + 2 + \sum_{i=1}^{k-1} 2 l_i.
\]
We have $2 l_i+\dim Z(C_i) = n^2$, so the weight equals
\[
(2 g +k-1)n^2 - 2n + 2 - \sum_{i=1}^{k-1} Z(C_i).
\]
In particular, the weights of all cells are equal, so by Proposition \ref{prop:weak cell lefschetz} we obtain curious Lefschetz for $\tilde X_\parab$ of the same weight. Since $\tilde X_\parab$ is a vector bundle over $X_\parab$ of rank $\binom{n}{2}$, we obtain that $X_\parab$ satisfies curious Lefschetz of weight
\[
(2g + k -2) n^2 +2 - \sum_{i=1}^{k} Z(C_i),
\]
which is the dimension of $X_\parab$. We conclude with
\begin{thm}\label{thm:stratification genus g}
	For arbitrary genus, suppose $k\geq 1$, $C_1,\ldots,C_k$ are diagonal matrices that are ordered nicely (Definition \ref{defn:ordered}) satisfying the genericity assumption (Definition \ref{def:generic}), and suppose $C_k$ has distinct eigenvalues. Then the corresponding character variety $X_\parab$ satisfies curious Lefschetz with middle weight equal to the dimension of $X_\parab$. 
\end{thm}

\begin{rem}
	If we enlarge the class of braid varieties to allow convolutions of braid TSVs and TSVs associated to surfaces, we can formulate a statement similar to Theorem \ref{thm:stratification genus 0} for arbitrary genus. We may think of the Seifert surfaces associated to the cells of the stratification as covers of the surface of genus $g$ ramified in one point, somehow perturbed in the neighborhood of the ramification point using the braid and the walk.
\end{rem}

\section{Monodromic $W$-action}\label{sec:monodromic action}
In this section we will prove results without the assumption that $C_k$ has distinct eigenvalues. In particular, the case of twisted character variety of complete curve (no punctures) studied in \cite{hausel2008mixed} we be covered. Our plan is as follows. By introducing extra puncture if necessary, we may assume $C_k=\Id$. To deal with this situation we perturb $C_k$, i.e. set $C_k=\lambda$ for $\lambda\in U_\varepsilon$,
\[
U_\varepsilon=\left\{(\lambda_1,\ldots,\lambda_n)\in T\;:\;\prod_{i=1}^n \lambda_i=1,\quad \lambda_i\neq\lambda_j,\quad |\lambda_i-1|<\varepsilon\right\},
\]
where $\varepsilon>0$ is sufficiently small. The character variety does not depend on the order of the eigenvalues, so we have a family $X_{U_\varepsilon/W}$ of character varieties over $U_\varepsilon/W$. Let $x_0\in U_\varepsilon$ be a base point and let $X_{x_0}$ be the fiber over $x_0$. Then the cohomology spaces of fibers form local systems over $U_\varepsilon$, so the cohomology of $X_{x_0}$ comes with an action of $\pi_1(U_\varepsilon/W,x_0)$. We show that this action comes from an action of $W$ via the natural homomorphism $\pi_1(U_\varepsilon/W,x_0)\to W$. We show that the class of the $2$-form $\omega$ and the weight filtration are invariant under the $W$-action, which implies that the curious Lefschetz holds for the anti-invariant part. Then we show that the cohomology of the character variety for $C_k=\Id$ together with the weight filtration and the action of $\omega$ is isomorphic to the anti-invariant part by identifying our action with the $W$-action coming from the Grothendieck-Springer sheaf.

Below we denote by $X^{=1}$ the character variety for $C_k=1$, and if $Y$ is has a map to $T$ we denote by $Y^t$ its fiber at $t\in T$.

\subsection{Singular character variety}
Suppose $C_1,\ldots,C_k$ is a collection of diagonal matrices satisfying Definitions \ref{def:generic} (generic) and \ref{defn:ordered} (ordered nicely), and further satisfying $C_k=\Id$. Define $V\subset T$ by
\[
V = \left\{t=(t_1,\ldots,t_n)\in T\;:\;(C_1,\ldots,C_{k-1},t)\;\text{is generic}\right\}.
\]
This subset is given by the condition $\det t=1$, and non-vanishing of a finite collection of functions on $T$. In particular, $V$ is an affine variety. We need to consider two versions of character and representation varieties.

The \emph{singular representation variety} $\CX_\sing$ is defined by
\[
\CX_{\sing} = \Big\{\alpha_1,\ldots,\alpha_g, \beta_1,\ldots,\beta_g, M_1,\ldots,M_{k-1},M_k \in G\;:
\]
\[
[\alpha_1,\beta_1]\cdots [\alpha_g, \beta_g] M_1 M_2 \cdots M_k = \Id,\;M_i\sim C_i\,(i<k),\; \chi(M_k)\in V/W \Big\},
\]
where $\chi$ is the characteristic polynomial map $G/_\ad G\to T/W$. Note that $V\subset T$ is $W$-invariant. This variety has an action of $G$ by conjugation. The \emph{singular character variety} $X_\sing$ is the quotient $X_\sing=\CX_\sing/G$. The variety $\CX_{\sing}$ is a smooth affine algebraic variety and $\CX_{\sing}/G$ is a good quotient, so $X_\sing$ is also smooth. The natural map $\mu:X_\sing\to V/W$ induced by $\chi$ is however not smooth, and in particular its fiber over $1$ is usually singular.
\begin{prop}\label{prop:singular quotient}
	The affine GIT quotient $X_\sing=\CX_\sing/(G/\BC^*)$ is a geometric quotient and the projection $\CX_\sing \to X_\sing$ is an \'etale locally trivial fibration.
\end{prop}
\begin{proof}
	The action does not have any fixed points, which implies that the affine GIT quotient parametrizes $G$-orbits, so the quotient is geometric.
	
	To construct a section $X_\sing\to \CX_\sing$ in a neighborhood of a point $[x_0]\in X_\sing $ for $x_0\in\CX_\sing$, it is sufficient to construct an \'etale neighborhood $Z\to X_\sing$ of $[x_0]$ and a $G/\BC^*$-equivariant map
	\[
	Z\times_{X_\sing} \CX_\sing \to G/\BC^*.
	\]
	
	We view points of $\CX_\sing$ as irreducible representations of the free non-commutative algebra $F_{2g+k}$ on $2g+k$ generators. For a point $x\in\CX_\sing$ and an element $P$ of $F_{2g+k}$ denote the corresponding matrix by $P_x$. By the density theorem (see \cite{etingof2011introduction}, the map $F_{2g+k}\to \Mat_{n\times n}$ is surjective for each representation. We can choose a non-commutative polynomial $P\in F_{2g+k}$ such that $P_{x_0}$ has distinct roots. The coefficients of the characteristic polynomial of $P_x$ are functions on $X_\sing$. Let $\Delta$ be the locus where the discriminant of the characteristic polynomial vanishes, $[x_0]\notin \Delta$. Let $Z\to X_\sing\setminus\Delta$ be the finite \'etale cover over which the characteristic polynomial has a root given by a regular function $f$ on $Z$. Let us pass to this cover and replace $x_0$ by one of its preimages in $Z\times_{X_\sing} \CX_\sing$. So $P_{x}-f(x)$ has one-dimensional kernel $K(x)$ for each $x\in Z\times_{X_\sing} \CX_\sing$. Choose elements $Q_2,\ldots,Q_n$ of $F_{2g+k}$ so that 
	\[
	K(x_0),Q_{2x_0} K(x_0),\ldots,Q_{nx_0} K(x_0)
	\]
	are linearly independent. Let $Z'\subset Z$ be the Zariski open subset where 
	\[
	K(x),Q_{2x} K(x),\ldots,Q_{nx} K(x)
	\]
	are linearly independent. Given a non-zero local section $k(x)$ of the line bundle $K(x)$, we construct a $G$-valued function by forming a matrix with columns 
	\[
	k(x),Q_{2x} k(x),\ldots,Q_{nx} k(x).
	\]
	Changing the section multiplies this matrix by a scalar function, so we obtain a well-defined function $Z'\times_{X_\sing} \CX_\sing \to G/\BC^*$, which is clearly $G/\BC^*$-equivariant.
\end{proof}

\subsection{Smooth character variety}
The \emph{smooth representation variety} $\CX_{\sm}$ is defined by 
\[
\CX_{\sm} = \Big\{\alpha_1,\ldots,\alpha_g, \beta_1,\ldots,\beta_g, M_1,\ldots,M_{k-1},M_k \in G,\;F\in G/B\;:
\]
\[
	[\alpha_1,\beta_1]\cdots [\alpha_g, \beta_g] M_1 M_2 \cdots M_k = \Id,\;M_i\sim C_i\,(i<k),\; M_k\in B_F,\; \chi_F(M_k)\in V \Big\},
\]
where $F$ is a flag, $B_F\subset G$ is the subgroup of elements preserving $F$, $\chi_F:B_F\to T$ is the natural projection. This variety has an action of $G$ by conjugation. The \emph{smooth character variety} is the quotient $X_{\sm}=\CX_\sm/G$. We will see that this is a good quotient. We have a map $\CX_\sing \to G$ given by reading off $M_k$. This induces a map
\[
\CX_\sing \times_{T/W} T \to G\times_{T/W} T.
\]
The Grothendieck-Springer resolution $\tilde G$ is the space of pairs
\[
\tilde G = \{g\in G, F\in G/B \;:\; g\in B_F\}.
\]
Forgetting $F$ gives a map $\tilde G\to G$ and it is clear that $\CX_\sm$ is the fiber product 
\[
\CX_\sm = \CX_\sing\times_G \tilde G.
\]
Note that $\CX_\sm$ is not affine.

\begin{prop}\label{prop:smooth quotient}
	Each point of $\CX_\sm$ has a Zariski open neighborhood $G$-equivariantly isomorphic to $V\times G/\BC^*$ where $V$ is a quasi-projective variety. Hence geometric quotient $X_\sm=\CX_\sm/(G/\BC^*)$ exists as a scheme and $\CX_\sm\to X_\sm$ is a Zariski locally trivial fibration. Moreover, $X_\sm$ is quasi-projective.
\end{prop}
\begin{proof}
	The proof of the first part is similar to the proof of Proposition \ref{prop:singular quotient}, but simpler because for $K(x)$ we can use the one-dimensional subspace of the flag $F$. In particular, we do not need to pass to a covering. We use notations $F_{2g+k}$, $P_x$ from the proof of Proposition \ref{prop:singular quotient}.

	Let us denote by $K(x)$ the subspace mentioned above. Let $x_0\in \CX_\sing$.	Choose elements $Q_2,\ldots,Q_n$ of $F_{2g+k}$ so that 
	\[
	K(x_0),Q_{2x_0} K(x_0),\ldots,Q_{nx_0} K(x_0)
	\]
	are linearly independent. Let $Z\subset \CX_\sm$ be the Zariski open subset where 
	\[
	K(x),Q_{2x} K(x),\ldots,Q_{nx} K(x)
	\]
	are linearly independent. Given a non-zero local section $k(x)$ of the line bundle $K(x)$, we construct a $G$-valued function by forming a matrix with columns 
	\[
	k(x),Q_{2x} k(x),\ldots,Q_{nx} k(x).
	\]
	Changing the section multiplies this matrix by a scalar function, so we obtain a well-defined function $f:Z \to G/\BC^*$, which is clearly $G/\BC^*$-equivariant. So we have $Z=V\times G/\BC^*$ where $V=f^{-1}(\Id)$.
	
	To see that $X_\sm$ is quasi-projective, note that the determinant of the matrix with columns $k(x),Q_{2x} k(x),\ldots,Q_{nx} k(x)$ is a $G$-equivariant morphism of $G$-equivariant line bundles $K^{\otimes n} \to O$. Let $L_0$ be a $G$-equivariant very ample line bundle on $G/B$. What we have shown means that for any point $x_0\in \CX_\sing$ there exists a $G$-equivariant section $s\in \Gamma(\CX_\sing, K^{\otimes -n})$ such that $s(x_0)\neq 0$, and the restriction of $L_0$ to $U_s$ embeds $U_s/G$ into a projective space, where $U_s$ is the set where $s$ is not zero. Cover $\CX_\sm$ by opens sets $U_{s_i}$ corresponding to sections $s_1,\ldots,s_N$. For each $i$ let $f_{i1},\ldots,f_{ir}$ be sections of $L_0$ which give embedding of $U_s/G$ into $\BP^{r-1}$. For sufficiently large $m$ the product $f_{ij} s_i^m$ can be extended to a section of $L_0 \otimes K^{\otimes -mn}$ on the whole of $\CX_\sm$ for all $i,j$. So the collection of sections $f_{ij}s_i^m$ embeds $X_\sm$ into $\BP^{rN-1}$.
\end{proof}

\subsection{Springer action}
There are maps $X_\sing\to V/W$, $X_\sm\to V$, given by $\chi$ and $\chi_F$ respectively and we need to understand how the cohomology groups of fibers of these maps are related. We will be working with the following diagram:
\begin{equation}\label{eq:cd of smooth and singular}
\begin{tikzcd}
X_\sm \arrow{d}{\pi_X} & \CX_\sm \arrow{r} \arrow{d}{\pi_\CX} \arrow{l} & \tilde G \arrow{d}{\pi}\\
X_\sing & \CX_\sing \arrow{l} \arrow{r} & G
\end{tikzcd}
\end{equation}
\begin{prop}
	Both squares in the above diagram are cartesian.
\end{prop}
\begin{proof}
	The square formed by $\CX_\sm, \CX_\sing, \tilde G, G$ is clearly cartesian. It remains to show that the square formed by $X_\sm, X_\sing, \CX_\sm, \CX_\sing$ is cartesian. Consider the natural map $\CX_\sm \to X_\sm\times_{X_\sing} \CX_\sing$. This is an isomorphism locally \'etale on $X_\sing$ by Propositions \ref{prop:singular quotient} and \ref{prop:smooth quotient}. Since \'etale covers are faithfully flat, this is an isomorphism.
\end{proof}

Let $G_\reg\subset G$ resp. $T_\reg\subset T$ be the locus of semisimple regular elements in $G$ resp. $T$. The map $\pi|_{\pi^{-1}(G_\reg)}$ is a Galois covering with structure group $W$. Thus $W$ acts on the direct image of the constant sheaf $\pi_* \BC_{\pi^{-1}(G_\reg)}$. This action can be extended to the derived direct image $R \pi_* \BC_{\tilde G}$ (see \cite{achar2014weyl}). By base change, since $\pi$ is proper, the derived direct image $R \pi_{\CX*} \BC_{\CX_\sm}$ equals to the pull-back of $R \pi_* \BC_{\tilde G}$ from $G$, so it also inherits a $W$-action. Both the sheaf $R \pi_{\CX*} \BC_{\CX_\sm}$ and its $W$ action are $G/\BC^*$-equivariant. Since $\CX_\sm$ and $\CX_\sing$ are principal homogeneous spaces, the categories of equivariant sheaves are simply the categories of sheaves on the quotient, so we obtain
\begin{prop}
	$\pi_X$ restricted to $\chi_F^{-1}(T_\reg)$ is a Galois covering with image $\chi^{-1}(T_\reg/W)$, so $\pi_{X*}\BC_{X_\sm}$ restricted to $\chi^{-1}(T_\reg/W)$ has a natural action of $W$. This action extend to a $W$-action on $R \pi_{X*} \BC_{X_\sm}$.
\end{prop} 

Denote $X_{\sm}^t=\chi_F^{-1}(t)$, $X_{\sing}^t=\chi^{-1}(t)$ and similarly for $\CX_\sm^t$, $\CX_\sing^t$. In particular, $X_{\sm}^1$ and $X_{\sing}^1$ correspond to the condition that $M_k$ is unipotent. We denote by $X_\sm^{=1}, X_\sing^{=1}, \CX_\sm^{=1}, \CX_\sing^{=1}$ the corresponding varieties with $M_k=1$. Note that $X_\sing^{=1}=X^{=1}$ is the variety we are interested in. The diagram \eqref{eq:cd of smooth and singular} base changes to the following diagram:
\begin{equation}\label{eq:cd of smooth and singular unipotent}
\begin{tikzcd}
X_\sm^1 \arrow{d}{\pi_X} & \CX_\sm^1 \arrow{r} \arrow{d}{\pi_\CX} \arrow{l} & \tilde N \arrow{d}{\pi}\\
X_\sing^1 & \CX_\sing^1 \arrow{l} \arrow{r} & N
\end{tikzcd}
\end{equation}
Here $N\subset G$ denotes the set of unipotent matrices and $\tilde N$ its preimage in $\tilde G$. The last diagram further base changes to 
\begin{equation}\label{eq:cd of smooth and singular unipotent2}
\begin{tikzcd}
X_\sm^{=1} \arrow{d}{\pi_X} & \CX_\sm^{=1} \arrow{r} \arrow{d}{\pi_\CX} \arrow{l} & G/B \arrow{d}{\pi}\\
X_\sing^{=1} & \CX_\sing^{=1} \arrow{l} \arrow{r} & 1
\end{tikzcd}
\end{equation}
Let $i$, $i_X$, $i_{\CX}$ denote the closed embedding of $1$, $X_\sing^{=1}$, $\CX_\sing^{=1}$ in $N$, $X_\sing^1$, $\CX_\sing^1$ respectively:
\begin{equation}\label{eq:cd of smooth and singular embeddings}
\begin{tikzcd}
X_\sing^{=1} \arrow{d}{i_X}& \CX_\sing^{=1} \arrow{l} \arrow{r} \arrow{d}{i_\CX}& 1 \arrow{d}{i}\\
X_\sing^1  & \CX_\sing^1 \arrow{r} \arrow{l} & \tilde N 
\end{tikzcd}
\end{equation}

There are Gysin maps
\[
R \pi_{X*} \BC_{X_\sm^{=1}} \to \BC_{X_\sing^{=1}}[-2\dim G/B], \quad R \pi_{*} \BC_{G/B}=H^*(G/B,\BC) \to \BC[-2\dim G/B],
\]
pulling back the latter one along $\CX_\sing^{=1}\to 1$ produces the former one. 
Applying $i_{X*}$ to the Gysin map we obtain a map between complexes of sheaves on $X_\sing^1$.
\begin{equation}\label{eq:sheaves on uni sing}
R (i_X \pi_X)_* \BC_{X_\sm^{=1}} \to i_{X*} \BC_{X_\sing^{=1}}[-2\dim G/B].
\end{equation}
Notice that $i_X \pi_X=\pi_X i_X$, where on the right hand side we have also denoted by $i_X$ the embedding of $X_\sm^{=1}$ into $X_\sm^1$:
\begin{equation}\label{eq:i and pi}
\begin{tikzcd}
X_\sm^{=1} \arrow{r}{i_X} \arrow{d}{\pi_X} & X_\sm^1 \arrow{d}{\pi_X}\\
X_\sing^{=1} \arrow{r}{i_X} & X_\sing^1
\end{tikzcd}
\end{equation}

We have a restriction map (adjunction morphism)
\begin{equation}\label{eq:sheaves on uni sm}
\BC_{X_\sm^1} \to i_{X*} \BC_{X_\sm^{=1}}.
\end{equation}
Composing $R\pi_{X*}$ applied to \eqref{eq:sheaves on uni sm} with \eqref{eq:sheaves on uni sing}, we obtain a natural map of sheaves on $X_\sing^1$
\begin{equation}\label{eq:sheaves on uni sing 2}
R \pi_{X*} \BC_{X_\sm^1} \to R (\pi_X i_X)_* \BC_{X_\sm^{=1}} =  R (i_X \pi_X)_* \BC_{X_\sm^{=1}}\to i_{X*} \BC_{X_\sing^{=1}}[-2\dim G/B].
\end{equation}
Denote by $(\cdot)^-$ the operation of projection to the sign representation of $W$ and by $(\cdot)^{\neq -}$ the projection to the rest. We have
\begin{prop}
	The composition of maps \eqref{eq:sheaves on uni sing 2} vanishes on $\left(R \pi_{X*} \BC_{X_\sm^\uni}\right)^{\neq -}$ and restricts to an isomorphism of sheaves on $X_\sing^1$
	\[
	\left(R \pi_{X*} \BC_{X_\sm^1}\right)^- \;\xrightarrow{\sim}\; i_{X*} \BC_{X_\sing^{=1}}[-2\dim G/B].
	\]
\end{prop}
\begin{proof}
	We first convert the statement to the statement about equivariant sheaves on $\CX_\sing^1$. There the map in question is the pullback from the corresponding map of $G$-equivariant sheaves on $N$ via the map $\CX_\sing^1\to N$. So we have to show that the corresponding map of sheaves on $N$ is zero on $(R\pi_{*} \BC_{\tilde N})^{\neq -}$ and is an isomorphism on $(R\pi_{*} \BC_{\tilde N})^{-}$:
	\begin{equation}\label{eq:maps over N}
R\pi_{*} \BC_{\tilde N} \to R (\pi i)_* \BC_{G/B} =  R (i \pi)_* \BC_{G/B}\to i_{*} \BC[-2\dim G/B],
	\end{equation}
	where $\pi$ and $i$ are given in the following diagram analogous to \eqref{eq:i and pi}:
	\[
	\begin{tikzcd}
	G/B \arrow{r}{i} \arrow{d}{\pi} & \tilde N \arrow{d}{\pi}\\
	1 \arrow{r}{i} & N
	\end{tikzcd}
	\]
	We know that the direct image $R\pi_{*} \BC_{\tilde N}$ has a decomposition as follows:
	\[
	R\pi_{*} \BC_{\tilde N} = \bigoplus_{\lambda} s_\lambda\otimes IC_{\bar N_\lambda}[-\codim N_\lambda],
	\]
	where $\lambda$ is a partition, $N_\lambda$ is the corresponding unipotent orbit, $\bar N_\lambda$ its closure, $IC$ is its intersection cohomology sheaf, $\codim N_\lambda$ the codimension, and $s_\lambda$ is the corresponding irreducible $W$-representation. The sign representation corresponds to the orbit $1$, so we obtain 
	\[
	(R\pi_{*} \BC_{\tilde N})^-\cong i_* \BC[-2\dim G/B],
	\]
	Let us compute the stalks at $1\in N$. We obtain a decomposition
	\[
	i^* R\pi_{*} \BC_{\tilde N} = H^*(G/B,\BC) = \bigoplus_{\lambda} s_\lambda\otimes i^* IC_{\bar N_\lambda}[-\codim N_\lambda].
	\]
	Since $\dim N = 2 \dim G/B$, we see that $i^*(R\pi_{*} \BC_{\tilde N})^-$ can only be isomorphically mapped to the top degree cohomology of $G/B$, which is one-dimensional. Hence all the other $W$-representations are mapped to smaller cohomological degrees. Note that the Gysin map in this case simply reads off the top degree cohomology.
	
	The stalk at $1$ of the map $(R\pi_{*} \BC_{\tilde N})^{\neq -} \to i_{*} \BC[-2\dim G/B]$ is zero, and the other stalks are zero automatically. Hence this map is zero. Both sheaves $(R\pi_{*} \BC_{\tilde N})^{-}$ and $i_* \BC[-2\dim G/B]$ are supported at $1$, and the map on the stalks at $1$ is an isomorphism, so the map $(R\pi_{*} \BC_{\tilde N})^{-} \to i_* \BC[-2\dim G/B]$ is an isomorphism.
\end{proof}

Note that the maps $\pi$ and $i$ are proper, so the direct image can be replaced by the direct image with compact support. Applying the functor $R\chi_!$ for the map $\chi:X_\sing\to T/W$, we obtain:
\begin{cor}\label{cor:W action on sm at 1}
	\begin{enumerate}
		\item The complex $R(\chi\pi_X)_! \BC_{X_\sm}$ on $T/W$ has a $W$-action.
		\item There is a natural morphism from the stalk of the above complex at $1$ to the cohomology of the character variety for $C_k=1$: $(R(\chi\pi_X)_! \BC_{X_\sm})_1 \to H^*_c(X^1,\BC)[-2\dim G/B]$.
		\item The morphism above vanishes on $(R(\chi\pi_X)_! \BC_{X_\sm})_1^{\neq -}$ and is an isomorphism on the sign component $(R(\chi\pi_X)_! \BC_{X_\sm})_1^-$.
	\end{enumerate}
\end{cor}

\subsection{Locally constant property}
It is not clear how to describe the Springer $W$-action explicitly. Instead we will use a special property of our situation which guarantees that the $W$-action is unique.
\begin{prop}\label{prop:locally constant}
For the family $\chi_F:X_\sm\to V\subset T$, the cohomology sheaves of $R \chi_{F!} \BC_{X_\sm}$ are locally constant. Also for any $i,k\in\BZ$ the associated graded sheaves for the weight filtration $\Gr^W_k H^i(R \chi_{F!} \BC_{X_\sm})$ are locally constant.
\end{prop}
\begin{proof}
By Proposition \ref{prop:smooth quotient}, the representation variety $\CX_\sm$ is a Zariski locally trivial $G/\BC^*$-fibration over $X_\sm$. Hence the subset 
\[
\CX_{\sm}^B = \Big\{(\alpha_1,\ldots,F)\in\CX_\sm\;:\; F=\text{standard flag}\Big\}
\]
 is a Zariski locally trivial $B/\BC^*$-fibration over $X_\sm$. Hence the quotient $\tilde X_\sm=\CX_{\sm}^B/T$ is an affine bundle over $X_\sm$ of relative dimension $\binom{n}{2}$. Denote the structure map by $\epsilon:\tilde X_\sm \to X_\sm$. In particular, $\epsilon$ is a smooth morphism and there is a morphism of sheaves
\[
R \epsilon_! \BC_{\tilde X_\sm} \to \BC_{X_\sm}\left[-2\dim \tilde X_\sm/X_\sm\right].
\]
It is an isomorphism locally on $X_\sm$, hence an isomorphism everywhere. So we can replace $X_\sm$ by $\tilde X_\sm$.

The space $\tilde X_\sm$ has a description given by \eqref{eq:product3} and in Sections \ref{ssec:puncture contributions} and \ref{ssec:genus contributions} we explained how to decompose such space into cells which are vector bundles over braid varieties, and these braid varieties are further decomposed into cells which look like $Z=\BC^a \times \BC^{*b}\cap \chi_F^{-1}(V)$ so that the map $\chi_F:Z\to T$ is given by a homomorphism $\BC^{*b}\to T$. The genericity assumption \ref{def:generic} and Proposition \ref{prop:generic implies connected} implies that only cells whose surfaces are connected have a non-empty intersection with $\tilde X_\sm$. Proposition \ref{prop:connected surface} implies that the image of $\BC^{*b}$ is the full subtorus $T_1=\{t\in T:\det t=1\}$, and $C^{*b}$ can be factored as $\BC^{*(b-n+1)}\times T_1$ so that the map to $T$ is just the projection to $T_1$. Therefore $Z$ can be factored as a product $Z=\BC^a \times \BC^{*(b-n+1)} \times V$ so that $\chi_F|_Z$ is a projection to $V$. So the pushforward with compact supports has constant cohomology sheaves.

So the complex $R \epsilon_! \BC_{\tilde X_\sm}$ can be obtained by taking successive cones of complexes whose cohomology sheaves are constant. Since the category of locally constant sheaves is abelian and closed under extensions, we obtain that the resulting object also has locally constant cohomology sheaves. The same holds for the associated graded sheaves.
\end{proof}

The property is important because of the following
\begin{lem}
	Suppose a finite group $\Gamma$ acts on a Hausdorff topological space $X$, and suppose on an open dense $\Gamma$-invariant subset $X_\reg\subset X$ the action is free. 
	Let $\CF$ be a locally constant sheaf on $X$ endowed with a $\Gamma$-equivariant structure over $X_\reg$. Let $p:X\to X/W$ be the natural map. The direct image $(p_* \CF)|_{X_\reg}$ has a natural $\Gamma$ action. We have a bijection between
	\begin{enumerate}
		\item extensions of the equivariant structure to the whole $\CF$ on $X$, and
		\item extensions of the $\Gamma$ action to the whole $p_*\CF$ on $X/\Gamma$.
	\end{enumerate}
	In particular, since there is at most one extension of the equivariant structure, the same is true for the extension of $\Gamma$-action.
\end{lem}
\begin{proof}
	Each extension of the equivariant structure clearly provides us with extension of the action. To give the opposite construction, suppose the action extends to $p_*\CF$. Let $x\in X$. Let $U$ be a neighborhood of $x$ such that $gU=U$ if $gx=x$ and $gU\cap U=\varnothing$ if $gx\neq x$ for all $g\in \Gamma$. By making $U$ smaller if necessary we can assume that $\CF$ is constant on $\Gamma U$. Let $y\in U\cap X_\reg$. Let $g_1,\ldots,g_m\Gamma$ be such that $g_i x\neq g_j x$ for $i\neq j$ and $\{g_1 x,\ldots,g_m x\}=\Gamma x$. Then we have a $\Gamma$-action on
	\[
	\CF(\Gamma U) = \bigoplus_{i=1}^m \CF(g_i U).
	\]
	This action respects the direct sum decomposition because the restriction map to the stalks at $\Gamma y$ imply that there is a $\Gamma$-equivariant injective map
	\[
	\bigoplus_{i=1}^m \CF(g_i U) \to \bigoplus_{g\in\Gamma} \CF_{gy},
	\]
	and the action respects the decomposition on the right. So we obtain a $\Gamma$-equivariant structure over $\Gamma U$. These structure clearly glue together to a $\Gamma$-equivariant structure over $X$.
	
	Since $\CF$ is locally constant and $X_\reg$ is dense, any map $\CF\to g^*\CF$ is uniquely determined by its restriction to $X_\reg$. So the uniqueness statement follows.
\end{proof}

Over $V_\reg=V\cap T_\reg$ the natural maps $X_\sm \to X_\sing\times_{T/W} T$ and $\tilde G\to G\times_{T/W} T$ become isomorphisms, so the Spinger action coincides with the action coming from the equivariant structure obtained by $W$-action on $T$. So the above Lemma applies.

\begin{cor}
	There exists a $W$-equivariant structure on the cohomology sheaves $H^i(R \chi_{F!} \BC_{X_\sm})$ extending the natural structure over $V_\reg$.
\end{cor}

The action of $w\in W$ on a local section $f$ of $\CF=H^i(R \chi_{F!} \BC_{X_\sm})$ over a value $t\in V$ can be explicitly computed as follows. Choose a point $t_0\in V_\reg$ and $\gamma:[0,1]\to V$ be a path from $t$ to $t_0$. Parallel transport along $\gamma$ gives a map $\CF_{t}\to \CF_{t_0}$. Then we use isomorphism of $X_{\sm}^{t_0}$ with $X_{\sm}^{w t_0}$ to obtain a map $\CF_{t_0} \to \CF_{w t_0}$. Finally, we trace $w\gamma$ backwards to obtain a map $\CF_{w t_0}\to \CF_{wt}$. The result is a composition
\[
\CF_t \to \CF_{t_0} \to \CF_{w t_0} \to \CF_{wt}.
\]
Parallel transport maps preserve the weight filtration. The Hodge filtration varies continuously and the point $t_0$ can be chosen arbitrarily close to $t$, so we obtain
\begin{cor}
	The action $H^i_c(X_{\sm}^t)\to H^i_c(X_{\sm}^{wt})$ preserves the mixed Hodge structure for any $w\in W$, $t\in V$.
\end{cor}
In particular, we have a well-defined action on the associated graded:
\[
\Gr^W_k H^i_c(X_{\sm}^t)\to \Gr^W_k H^i_c(X_{\sm}^{wt}).
\]

Setting $t=1$ we obtain an action of $W$ on $H_c^*(X_\sm^1)$ which preserves the mixed Hodge structure. We have the restriction map 
\[
H_c^*(X_\sm^1)\to H_c^*(X_\sm^{=1})
\]
followed by the Gysin map
\[
H_c^*(X_\sm^{=1}) \to H_c^{*-2\dim G/B}(X^{=1})
\]
coming from the fact that $X_\sm^{=1}$ is a relative flag variety over $X^{=1}$. The restriction map preserves the mixed Hodge structure, while the Gysin map shifts weights, so we obtain a morphism of mixed Hodge structures
\[
H_c^i(X_\sm^1) \to H_c^{i-2\dim G/B}(X^{=1})(-\dim G/B).
\]
By (iii) of Corollary \ref{cor:W action on sm at 1} this is a surjective map. Hence the mixed Hodge structure on $H_c^{*-2\dim G/B}(X^{=1})(-\dim G/B)$ is induced from the mixed Hodge structure on $H_c^*(X_\sm^1)$. On the other hand since the $W$-action respects the mixed Hodge structure, the direct sum decomposition $H^i_c(X_{\sm}^1)=H^i_c(X_{\sm}^1)^-\oplus H^i_c(X_{\sm}^1)^{\neq -}$ is compatible with the mixed Hodge structures, so the mixed Hodge structure induced on 
\[
H^i_c(X_{\sm}^1)^-\cong H_c^{i-2\dim G/B}(X^{=1})(-\dim G/B)
\]
viewed as a quotient coincides with the filtration obtained by viewing $H^i_c(X_{\sm}^1)^-$ as a sub.
\begin{cor}
	The isomorphism $H^i_c(X_{\sm}^1)^- \cong H_c^{i-2\dim G/B}(X^{=1})(-\dim G/B)$ respects the mixed Hodge structures and we have an isomorphism
	\[
	(\Gr^W_k H^i_c(X_{\sm}^1))^- \cong \Gr^W_{k-2\dim G/B} H_c^{i-2\dim G/B}(X^{=1}).
	\]
\end{cor}

\subsection{Flatness of the $2$-form}
As another ingredient, we need to check that the operator of multiplication by the cohomology class of the $2$-form commutes with $W$. The $W$-action is given by a composition of parallel transport and the isomorphism $X_{\sm,t_0} \cong X_{\sm,w t_0}$ for $t_0\in V_\reg$. The latter preserves $\omega$ because the form on a semisimple orbit does not depend on the choice of order of eigenvalues. In notations of Section \ref{ssec:2 forms} we have:
\[
(g|w C_i w^{-1}| g^{-1}) = (g|w |C_i |w^{-1}| g^{-1}) = (g w |C_i| (g w)^{-1}),
\]
for $g$ a variable with values in $G$, $C_i$ constant diagonal matrix and $w\in W$. So the only problem is to verify that operator of multiplication by $\omega$ is flat with respect to the Gauss-Manin connection. It turns out, $\omega$ is only flat modulo $W_2$, the second step in weight filtration, which is sufficient for our purposes.

We explicitly compute the covariant derivative of $\omega$. The form $\omega$ in the fibers of $\chi_F$ can be represented by the following form on $\CX_\sm^B$:
\[
\omega=(\alpha_1|\beta_1|\cdots|\alpha_g^{-1}|\beta_g^{-1}|\gamma_1^{-1}|C_1|\gamma_1|\cdots|\gamma_{k-1}^{-1}|C_{k-1}|\gamma_{k-1}||u t)
\]
Using formulas from Section \ref{ssec:d omega}, we see that the form is closed. On the other hand, the form is equivariant in the fibers, but not globally, because for a variable $b$ on $B$ we have
\[
b^* \omega = (b \alpha_1 b^{-1}|b\beta_1 b^{-1}|\cdots|b \gamma_{k-1}^{-1}|C_{k-1}|\gamma_{k-1} b^{-1}| b u t b^{-1})
\]
\[
= (b |\alpha_1 |b^{-1}|b|\beta_1|b^{-1}|\cdots|b |\gamma_{k-1}^{-1}|C_{k-1}|\gamma_{k-1}| b^{-1}| b| u t |b^{-1}) - (b|u t|b^{-1})
\]
\[
= \omega - (b|u t|b^{-1}),
\]
\[
(b|u t|b^{-1}) = \sum_{i=1}^n d\log b_{ii} \wedge d\log t_i - \sum_{i=1}^n d\log t_i \wedge d\log b_{ii} = 2\sum_{i=1}^n d\log b_{ii} \wedge d\log t_i.
\]
We obtain
\[
b^*\omega = \omega + 2\sum_{i=1}^n d\log t_i \wedge d\log b_{ii},
\]
which means that the covariant derivative of the cohomology class of $\omega$ is given by 
\[
\nabla [\omega] = 2\sum_{i=1}^n d\log t_i\otimes  c_1(L_i),
\]
where $L_i=F_i/F_{i-1}$ is the canonical line bundle on $X_\sm$. Since characteristic classes are pure, we obtain $\nabla [\omega]$ belongs to $W^2$, and therefore parallel transport preserves the class of $\omega$ in $\Gr^W_4 H^2(X_{\sm}^t,\BC)$. We have established
\begin{cor}
	The action of $w\in W$ sends the class of $\omega$ in $\Gr^W_4 H^2(X_{\sm}^t,\BC)$ to the class of $\omega$ in $\Gr^W_4 H^2(X_{\sm}^{wt},\BC)$.
\end{cor}

For $t=1$ we obtain that $\omega$ is preserved by the $W$-action. The form $\omega$ when restricted to $X_\sm^{=1}$ equals to the pullback of the corresponding form from $X^1$, so we have a commutative square
\[
\begin{tikzcd}
\Gr^W_k H_c^i(X_\sm^1) \arrow{r} \arrow{d}{\cup \omega}& \Gr^W_{k-2\dim G/B} H_c^{i-2\dim G/B}(X^1) \arrow{d}{\cup \omega}\\
\Gr^W_{k+4} H_c^{i+2}(X_\sm^1) \arrow{r} & \Gr^W_{k-2\dim G/B+4} H_c^{i-2\dim G/B+2}(X^1)\\
\end{tikzcd}
\]
which induces a commutative square
\[
\begin{tikzcd}
(\Gr^W_k H_c^i(X_\sm^1))^- \arrow{r}{\sim} \arrow{d}{\cup \omega}& \Gr^W_{k-2\dim G/B} H_c^{i-2\dim G/B}(X^1) \arrow{d}{\cup \omega}\\
(\Gr^W_{k+4} H_c^{i+2}(X_\sm^1))^- \arrow{r}{\sim} & \Gr^W_{k-2\dim G/B+4} H_c^{i-2\dim G/B+2}(X^1)\\
\end{tikzcd}
\]
It is clear now that curious Lefschetz property for $X_\sm^1$ with middle weight $d$ implies curious Lefschetz for $X^1$ with middle weight $d-2\dim G/B$.

The curious Lefschetz property for $X_\sm^1$ can be proved in two ways. In one way, as explained in the proof of Proposition \ref{prop:locally constant}, we can cover $X_\sm^1$ by affine bundle of dimension $\binom{n}2$, and then use the cell decomposition, as we did in the proof in the case when $C_k$ is regular semisimple, Theorem \ref{thm:stratification genus g}. Another approach is to use parallel transport to a point $t\in V_\reg$. The parallel transport (solving Picard-Fuchs differential equations) preserves the weight filtration, and by flatness of the class of the form in $\Gr^W_4 H^2(X_{\sm}^t,\BC)$, the action of $\omega$ on the associated graded is preserved by parallel transport.

Note that the difference of dimensions $\dim X^1_\sm-\dim X^1$ is $n^2-n=2\dim G/B$. So the proof of the following is complete:
\begin{thm}
	For arbitrary genus and any $k\geq 0$, any tuple of diagonal matrices $C_1,\ldots,C_k$ satisfying the genericity assumption, the corresponding character variety satisfies curious Lefschetz with middle weight equal to the dimension.
\end{thm}

\section*{Acknowledgments}
This work started as an attempt to find cell decompositions for character varieties analogous to the decompositions in \cite{gunnels2018torus}. I am grateful to Fernando Rodriguez-Villegas for bringing my attention to this work. Vivek Shende explained some parts of the paper \cite{shende2017legendrian} to me and suggested the approach to prove non-degeneracy of the $2$-form (Section \ref{sec:seifert}). I am indebted to him for his advices. I am very grateful to Sasha Beilinson for answering my questions about perverse sheaves and mixed Hodge structures. I discussed various parts of this work with Marco Bertola, Ben Davison, Iordan Ganev, Tamas Hausel, Quoc Ho, Emmanuel Letellier, Penghui Li, Andrei Negut, Alexei Oblomkov, Carlos Simpson and they gave me good advices, for which I am thankful.

Parts of this work were performed at the ICTP and SISSA in Trieste and IST Austria in Maria Gugging. I am grateful to them for hospitality. 

I am grateful to the Austrian Science Fund (FWF) for supporting this work through the projects Y963-N35 and P-31705.


\bibliographystyle{amsalpha}
\bibliography{../refs}

\end{document}